\documentclass[a4paper]{amsart}
\usepackage{fullpage}
\usepackage[T1]{fontenc}
\usepackage[utf8]{inputenc}
\usepackage[english]{babel}
\usepackage{paralist,hyperref}
\usepackage{enumitem}
\usepackage{pgf}
\usepackage{tikz}
\usepackage{tikz-cd}
\usetikzlibrary{cd}
\usetikzlibrary{calc}
\usetikzlibrary{arrows}
\usetikzlibrary{shapes}
\usepackage{barycentric}

\usepackage{cleveref}
 \usepackage{amsmath}
\usepackage{amssymb}
\usepackage[mathscr]{eucal}

\renewcommand{\geq}{\geqslant}
\renewcommand{\leq}{\leqslant}

\newcommand{\supp}{\operatorname{supp}}
\usepackage{amsthm}
\usepackage[disable,colorinlistoftodos,bordercolor=orange,backgroundcolor=orange!20,linecolor=orange,textsize=scriptsize]{todonotes}

\theoremstyle{plain} 
\newtheorem{thm}{Theorem}[section]
 
\newtheorem{cor}[thm]{Corollary} 
\newtheorem{lem}[thm]{Lemma}
 
\newtheorem{prop}[thm]{Proposition} 

\theoremstyle{definition}
\newtheorem{defn}{Definition}[section]
\newtheorem{definition}[defn]{Definition}

\theoremstyle{remark}
\newtheorem{example}[thm]{Example}
\newtheorem{remark}[thm]{Remark}

\newcommand{\myitem}[1]{{\em(\ref{#1})}}
\newcommand{\mytop}{\mathbf{t}}
\newcommand{\mybot}{\mathbf{b}}
\newcommand{\R}{\mathbb{R}}

\newcommand{\tmax}{\mathbb{R}_{\max}}

\newcommand{\tmin}{\mathbb{R}_{\min}}
\newcommand{\tmaxn}[1]{(\tmax)^{#1}}
\newcommand{\tminn}[1]{(\tmin)^{#1}}

\newcommand{\tangent}{\mathcal{T}}
\newcommand{\supc}{\operatorname{sup}^C}
\newcommand{\infc}{\operatorname{inf}^C}
\newcommand{\supe}{\operatorname{sup}^E}
\newcommand{\infe}{\operatorname{inf}^E}
\newcommand{\limsupc}{\operatorname{limsup}^C}
\newcommand{\liminfc}{\operatorname{liminf}^C}
\newcommand{\alcoved}[1]{\mathscr{A}(#1)}
\newcommand{\finitesets}[1]{\mathscr{P}_{\!f}(#1)}
\newcommand{\powerset}[1]{\mathscr{P}(#1)}
\newcommand{\sigmaset}{\mathscr{S}}
\newcommand{\tauset}{\mathscr{T}}
\newcommand{\piset}{\mathscr{P}}
\newcommand{\lowerclosure}[1]{\operatorname{clo}^{\downarrow} #1}
\newcommand{\upperclosure}[1]{\operatorname{clo}^{\uparrow}  #1}
\newcommand{\range}[1]{\operatorname{Im}#1}

\newcommand{\proj}{\operatorname{proj}}
\newcommand{\argmin}{\operatornamewithlimits{argmin}}
\newcommand{\argmax}{\operatornamewithlimits{argmax}}
\newcommand{\skeletton}{\operatorname{Sk}}
\newcommand{\pmax}[1]{P^{\max}_{#1}}
\newcommand{\pmin}[1]{P^{\min}_{#1}}
\newcommand{\barqm}[1]{\bar{Q}^-_{#1}}
\newcommand{\barqp}[1]{\bar{Q}^+_{#1}}

\newcommand{\qm}[1]{{Q}^-_{#1}}
\newcommand{\qp}[1]{{Q}^+_{#1}}
\newcommand{\qpm}[1]{{Q}^\pm_{#1}}

\newcommand{\tropicalbasisvector}{e^{\operatorname{trop}}}

\title{Ambitropical geometry, hyperconvexity and zero-sum games}
\author{Marianne Akian}
\address{INRIA and CMAP, IP Paris, \'Ecole polytechnique, CNRS. Postal address: CMAP, \'Ecole polytechnique, 91128 Palaiseau C\'edex}
\email{Marianne.Akian@inria.fr}
\author{St\'ephane Gaubert}
\address{INRIA and CMAP, IP Paris, \'Ecole polytechnique, CNRS. Postal address: CMAP, \'Ecole polytechnique, 91128 Palaiseau C\'edex}
\email{Stephane.Gaubert@inria.fr}
\author{Sara Vannucci}
\address{Czech Technical University in Prague}
\email{vanucsar@fel.cvut.cz}
\thanks{Version of \today. This work was initiated when Sara Vannucci was with Universit\`a di Salerno. This work was partially supported by the National Group for Algebraic and Geometric Structures, and their Applications (GNSAGA - INdAM)}


\begin{document}
\begin{abstract}
  Shapley operators of undiscounted zero-sum two-player games are order-preserving maps that commute with the addition of a constant.
  We characterize the fixed point sets of Shapley operators, in finite dimension (i.e., for games with a finite state space). Some of these characterizations are of a lattice theoretical nature, whereas some other rely on metric or tropical geometry. More precisely, we show that fixed point sets of Shapley operators are special instances of hyperconvex spaces: they are sup-norm non-expansive retracts of $\R^n$, and also lattices in the induced partial order. Moreover, they retain properties of convex sets, with a notion of ``convex hull'' defined only up to isomorphism. This provides an effective construction of the injective hull or tight span, in the case of additive cones. For deterministic games with finite action spaces,  these fixed point sets are supports of polyhedral complexes, with a cell decomposition attached to stationary strategies of the players, in which each cell is an alcoved polyhedron of $A_n$ type. 
  We finally provide an explicit local representation of the latter fixed point sets, as polyhedral fans canonically associated to lattices included in the Boolean hypercube.
\end{abstract}
\maketitle
\section{Introduction}
\todo[inline]{mention the flip $X\to -X$ which preserves ambitropical sets, SV: mentioned at page 22}
\subsection{Motivation}
\label{subsec-motiv}
Shapley operators play a fundamental role in the study of zero-sum repeated games, see~\cite{Sha53,neymansurv,MSZ2015}. For infinite horizon problems, including the cases of a discounted payoff, of a total payoff up to a stopping time, or of a mean payoff (in which the payoff is given by a time average), the fixed point of a suitable Shapley operator determines
the value of the game as well as optimal stationary strategies.
The mean payoff problem is somehow the most difficult one. Then, the notion
of fixed point is defined in a projective sense (up the action
of additive constants), see~\cite[Th.~VI.1]{moulinsmf}.
The existence and uniqueness of such fixed points have been
investigated in the setting of non-linear Perron-Frobenius theory~\cite{arxiv1,nussbaumlemmens,1078-0947_2020_1_207}. The absence
of fixed points is tied to the time-dependent nature of the nearly optimal strategies~\cite{blackwellferguson,BK76}, whereas multiple fixed points yield multiple optimal stationary strategies~\cite{KY92}.

The structure of the fixed point set of Shapley operators
has been studied especially in the deterministic one player case,
as part of tropical spectral theory~\cite{BCOQ92,KM97,AGW09},
or in the setting of viscosity solutions
of ergodic Hamilton-Jacobi partial differential equations and weak-KAM theory,
see~\cite{Fathi2005,ishii2007,Fat}. The results there show in particular that
fixed points are uniquely determined by their restriction
to a suitable subset or ``boundary'' of the state space,
in a way somehow analogous to the Poisson-Martin representation of harmonic
functions~\cite{Dynkin}.
For one player deterministic problems, the role of the ``boundary'' is played either
by a distinguished subset of the state space (``critical nodes''~\cite{BCOQ92},
``projected Aubry set''~\cite{Fathi2005}) or by a metric boundary (the horoboundary) of the state space~\cite{ishii2007,AGW09}. Fixed point sets appear to be wilder objects in the two-player case, even when the space space
is finite. The geometry of these fixed point sets is the main subject of this paper.

A further motivation arises from
algebra in characteristic one~\cite{connesconsani} and tropical geometry~\cite{maclagan_sturmfels}. In this setting,
one needs to work in categories in which the objects
are tropical analogues of linear spaces,
and the arrows are linear maps.
However, tropical duality results
require to consider at the same time properties of linearity
in a primal and in a dual sense, and sometimes to compose
maps that are linear in each of these senses (e.g., compositions
of min-plus and max-plus linear maps). This is the case, for instance, of
projections onto linear spaces, which are generally
non-linear~\cite{CGQ96a,cgq02}, but which are still Shapley operators.
Hence, it may be desirable to develop a broader, ``self-dual'', framework, in which the objects include both tropical (max-plus) linear spaces and their duals,
and the arrows are Shapley operators.

Thinking of Shapley operators in abstract terms leads to a metric geometry approach, exploiting
a connection with the theory of nonexpansive mappings. Recall that a self-map
$T$ of a metric space $(X,d)$ is {\em nonexpansive} if $d(T(x),T(y))\leq d(x,y)$.
For undiscounted games
with state space $[n]$, the Shapley operators are precisely the self-maps $T$ of $\R^n$ that are nonexpansive with respect to the metric associated
with the sup-norm and that commute with the action of additive constants~\cite{crandall,KM97}. As shown in~\cite{gunawardenakeane}, this is equivalent to
$T$ being nonexpansive in 
the weak (non-symmetric) metric $d(x,y)=\mathsf{t}(y-x)$ where
$\mathsf{t}(z):=\max_{i\in n}z_i$.
The map $\mathsf{t}(\cdot)$ here is an example
of {\em weak Minkowski norm} or {\em hemi-norm}, and 
it is a special case of the local norm associated to the Finsler structure
of the ``Funk metric'', studied in Hilbert's geometry, see~\cite{papadopoulos,walshHilbert,GV10}.

From this perspective, the study of fixed point sets of Shapley operators becomes tied to the classical topic of fixed point sets of nonexpansive
mappings and nonexpansive retracts (i.e., fixed points of idempotent nonexpansive maps) in Banach spaces, see~\cite{reichsurvey} for background.
Recall that a nonexpansive retract $C$ of a Banach space $X$ that satisfies
a technical condition (the so called {\em fixed point property for spheres})
valid in particular in finite dimension,
must be {\em metrically convex}~\cite{Bruck73},
meaning that for all $x,y\in C$ and for all $\alpha,\beta\geq 0$
such that $\alpha+\beta=1$, there must exist a point $z$ in $C$
such that $d(x,z)=\alpha d(x,y)$ and $d(z,y)=\beta d(x,y)$. In particular,
nonexpansive retracts of strictly convex Banach spaces are closed
and convex. Conversely, any closed and convex subset $C$ of a Hilbert
space $X$ is a nonexpansive retract of this space, and indeed,
a retraction is given by the best approximation map,
which associates to a point of $X$
the nearest point in $C$, see e.g.~\cite[Th.~3.6]{goebelreich}.
Moreover, if $X$ is a Banach
space of dimension at least $3$, it is known that all the closed and convex
subsets of $X$ are nonexpansive retracts if and only if $X$ is a Hilbert space,
see~\cite[Prop.~2.2]{reich77}, and also~\cite{KAKUTANI1940,klee,Bruck1974}
for earlier results of this nature.

In this way, classical convexity appears to be linked to the geometry of Euclidean retractions. Then, one may wonder whether other types of (weak) metric spaces are tied
  to interesting  convexity theories. In particular, we may ask whether fixed point sets of Shapley operators may be thought of as ``convex sets'' in a useful sense. We give here a positive answer to this question, by establishing links between the theory of fixed point sets
  of Shapley operators, tropical geometry, order preserving retracts of lattices, and hyperconvexity. We focus on the finite dimensional case.

\subsection{Summary of results}
  We define a subset $C$ of $\R^n$ to be an {\em ambitropical} cone if $C$ is invariant by translation by constant vectors,
  and if $C$ is a lattice in the order induced by the standard partial order of $\R^n$ (but not necessarily a sublattice of $\R^n)$. This includes as special cases the max-plus and min-plus convex cones arising in
  idempotent analysis~\cite{litvinov} and tropical geometry~\cite{cgq02,TropConv}. In contrast with the classes of max-plus and min-plus cones, the class of ambitropical cones is self-dual since it is invariant by the ``flip'' (change of sign) operation. Hence, we use the name {\em ambitropical}, as it includes both tropical convexity and its dual. We call {\em Shapley retract}
  the image of an idempotent Shapley operator.

  \Cref{th-main0} below shows that Shapley retracts are precisely
  closed ambitropical cones.
    Further, we show that there are canonical retractions on a closed ambitropical
    cone $C$, characterized as the composition of nearest point projection
    mappings on the tropical convex cone and dual tropical convex cone
    generated by $C$, see~\Cref{prop-caracq}.
    This leads to a further characterization
    of ambitropical cones, by an analogue of the ``best co-approximation property''
    arising in the theory of Banach spaces~\cite{singer},
    see~\Cref{th-main1} below.
    We also show that the class of ambitropical cones admits an analogue
    of the ``convex hull'' operation: although the intersection
    of ambitropical cones may not be ambitropical, there is a well
    defined notion of ambitropical hull, the minimal closed ambitropical cone containing a given
    set, which is unique up to isomorphism (the morphisms being Shapley operators).
    One main result, \Cref{Hypambihull}, shows that ambitropical cones
    are precisely hyperconvex sets that are additive cones. Since we provide
    an effective construction of the ambitropical hull, in terms
    of the range of a tropical Petrov-Galerkin projector (\Cref{prop-ambihull}),
    this leads to an explicit construction of the hyperconvex hull of an additive cone.

    We subsequently study ambitropical cones with a semilinear structure, which
    we call {\em ambitropical polyhedra}. The building blocks of ambitropical
    polyhedra are the alcoved polyhedra of the root system $A_n$, studied
    in~\cite{AlcovedPolytopes}, which include as special cases
    Stanley's order polyhedra~\cite{stanley86}.
    Ambitropical polyhedra are defined as ambitropical cones that are finite
    unions of alcoved polyhedra. \Cref{th-ambitropicalpoly} shows that ambitropical polyhedra
    coincide with fixed point sets of Shapley operators
    of deterministic games with finite action spaces.
      Further, we show that ambitropical polyhedra are polyhedral complexes whose cells are associated
      to pairs of stationary policies of both players that are optimal in the mean payoff problem, see~\Cref{th-polycomplex}. We study, in particular, the
      case of {\em homogeneous} ambitropical polyhedra, i.e., ambitropical polyhedra
      that are invariant by the multiplicative action of positive scalars.
      We show that
      homogeneous ambitropical polyhedra arise when considering tangent cones
      of ambitropical polyhedra, so, they provide a {\em local} description of these      sets. \Cref{th-mainweyl} provides a 
      characterization of homogeneous
      ambitropical polyhedra, as polyhedral fans associated
      to subsets of $\{0,1\}^n$ that are lattices in the induced
      order.

The hierarchy of classes of sets considered in this paper is presented
on~\Cref{table-hierarchy}.
\begin{table}

  \begin{tikzpicture}
    \node (hyper) at (0,3/2) {\begin{tabular}{c}hyperconvex sets\\ \Cref{th-hyper} %
    \end{tabular}};
      \node (nonclosed) at (5.5*1.5,3/2) {\begin{tabular}{c}ambitropical cones
    \end{tabular}};

  \node (top) at (0,0) {\begin{tabular}{c}closed ambitropical cones\\ Theorems~\ref{th-main0}, \ref{th-main1} \end{tabular}};
  \node at (-2*1.5,-2/2) (left) {closed tropical cones};
  \node at (2*1.5,-2/2) (right) {closed dual tropical cones};
  \node at (0*1.5,-3.5/2) (mid) {\begin{tabular}{c}ambitropical polyhedral cones\\ \Cref{th-ambitropicalpoly}\end{tabular}};
  \node at (-2*1.5,-6/2) (botleft) {tropical polyhedral cones};
  \node[fill=white] at (2*1.5,-6/2) (botright) {dual tropical polyhedral cones};
   \node at (5.5*1.5,-5.5/2) (botrightright) {\begin{minipage}{6.5cm}\begin{tabular}{c} homogeneous ambitropical polyhedra\\ $\simeq$ lattices in $\{0,1\}^n$\\\Cref{th-mainweyl}\end{tabular}\end{minipage}};
   \node at (0,-8/2) (botmid) {alcoved polyhedra};%
   \node at (5.5*1.5,-10/2) (botbotright) {order polyhedra};%
   \draw[thick] (top) -- (hyper);
        \draw[thick] (top) -- (nonclosed);
  \draw[thick] (top) -- (left);
  \draw[thick] (top) -- (right);
  \draw[thick] (top) -- (mid);
  \draw[thick] (mid) -- (botright);
  \draw[thick] (mid) -- (botleft);
  \draw[thick] (botleft) -- (left);
  \draw[preaction={draw=white, -, line width=6pt},thick] (botright) -- (right);
  \draw[thick] (mid) -- (botmid);
  \draw[thick] (botleft) -- (botmid);
  \draw[thick] (botright) -- (botmid);
  \draw[thick] (mid) -- (botrightright);
  \draw[thick] (botmid) -- (botbotright);
  \draw[thick] (botrightright) -- (botbotright);
    \node[fill=white] at (2*1.5,-6/2) (botright) {dual tropical polyhedral cones};
\end{tikzpicture}
\caption{The hierarchy of ambitropical cones}\label{table-hierarchy}
\end{table}

\subsection{Related work}
As indicated in~\S\ref{subsec-motiv},
the present results are related to several series of works.
A first source of inspiration is the theory of ``best approximation''
in tropical geometry~\cite{cgq02,AGNS10}, in which (non-linear) projectors
onto max-plus / min-plus spaces have been studied.
In particular, the canonical retractions on ambitropical cones turn
out to coincide with the tropical analogue of Petrov-Galerkin
projectors, introduced in~\cite{CGQ96a} and applied in~\cite{a6}
to the numerical solution of optimal control problems.

The representation of ambitropical polyhedra (\Cref{th-polycomplex})
as the support of a polyhedral complex is somehow
inspired by the characterization of Develin and Sturmfels~\cite{TropConv}
of polyhedral complexes arising from
tropical polyhedral cones, in terms of the
duals to regular subdivisions of the product of two simplices.
The latter property makes use of (classical) convex duality, and so this does
not carry over to the ambitropical case given its ``minimax'' nature.
However, we still get a complete combinatorial characterization in the
special case of {\em homogeneous} ambitropical polyhedra,
see~\Cref{th-mainweyl}, showing these are equivalent to lattices
included in $\{0,1\}^n$ with the induced order. Such lattices
are precisely the order preserving retracts of $\{0,1\}^n$
studied by Crapo~\cite{crapo}.

\todo[inline]{SV : I removed the above part (you can find it in the source file with percentage symbol) because it refereed to the lattice section that we deleted}
Moreover, sup-norm nonexpansive retracts (not necessarily order preserving)
have been studied in the setting of {\em hyperconvexity}, a
notion introduced by Aronszajn and Panitchpakdi~\cite{aronszajn},
see~\cite{baillon,Espinola2001} for more information. Hyperconvex
spaces are metrically convex spaces in which the collection
of closed balls has Helly number two.
It follows from~\cite[Th.~9]{aronszajn}
that the sup-norm nonexpansive
retracts of $\R^n$ are precisely the closed subsets of $\R^n$
that are hyperconvex. Hence, closed ambitropical cones are special
instances of closed hyperconvex sets, with an additional structure
induced by the order and the additive homogeneity. Furthermore, an equivalence holds,  namely that an additive cone of $\R^n$ is a closed ambitropical cone if and only if it is hyperconvex for the sup-norm metric. This result allows us to compute the hyperconvex hull of an additive cone. The problem of computing the hyperconvex hull has received much attention, Isbell and Dress~\cite{Isbell1964,Dress1984} shows that this hull can be realized as a ``tight-span''. Even in dimension $2$, computations of hyperconvex hull are difficult~\cite{KILIC2016693}. Our results solve this problem for the special case of additive cones.

As mentioned above, it is an open problem to understand what Fathi's characterization of weak-KAM solutions, as spaces of Lipchitz functions on the ``projected Aubry set''~\cite{Fathi2005,Fat}, becomes in the ``two-player'' case. Our results answers the analogue of this question in the discrete, finite dimensional case: whereas one-player solution spaces are alcoved polyhedra, in the two-player case, we show that the solution spaces are precisely ambitropical sets, which in the finite action case are obtained by gluing alcoved polyhedra. This interpretation is elaborated in~\Cref{th-polycomplex}.

\todo[inline]{I added the above part regarding the results that we have in the hyperconvex section}

Finally, our treatment of Shapley operators is inspired by the ``operator approach'' of zero-sum games, early work in these directions include~\cite{Eve57,Koh74,BK76}, see~\cite{rosenbergsorin,Ren11,BGV15,ziliotto} for more recent developments.

\section{Preliminary results}\label{sec-prelim}
In this section, we establish, or recall,
basic results concerning tropical cones.
\subsection{Additive cones, tropical cones and semimodules}
The tropical semifield, $\tmax$, is the set
$\R\cup\{-\infty\}$ equipped with the addition
$(x,y) \mapsto x\vee y := \max(x,y)$ and with
the multiplication $(x,y) \mapsto x+y$.
It admits a zero element, equal to $-\infty$,
and a unit element, equal to $0$.
We shall also use the min-plus version
of the tropical semifield, $\tmin$, which is the set
$\R\cup\{+\infty\}$ equipped with the addition
$(x,y) \mapsto x\wedge y := \min(x,y)$ and with
the multiplication $(x,y) \mapsto x+y$.
The semifields $\tmax$ and $\tmin$ are isomorphic.
We denote by $\leq$ the partial coordinatewise order of $(\R\cup\{\pm\infty\})^n$, and we use the notation
$\lambda + x:=(\lambda+x_i)_{i\in [n]}$, 
for all $\lambda\in \tmax$ and $x\in \tmaxn{n}$, 
and similarly for  $\lambda\in \tmin$ and $x\in \tminn{n}$.
We extend the notation $\vee$ to denote the supremum of vectors of $(\R\cup\{\pm\infty\})^n$. Similarly, $\wedge$ denote
the infimum of vectors.

\begin{defn}
  An {\em additive cone} of $\R^n$ is a subset $C$ of $\R^n$ such that
  \begin{enumerate}
\item\label{it-def-t2} $x\in C,\lambda\in \R \implies \lambda + x\in C$.
  \end{enumerate}
  A {\em tropical cone} is an additive cone $C$ such that:
    \begin{enumerate}\setcounter{enumi}1
    \item\label{it-def-t1}    $x,y\in C\implies x\vee y\in C$.
      \end{enumerate}
  A {\em dual tropical cone} is defined similarly, by requiring that
  $x,y\in C\implies x\wedge y\in C$, instead of \eqref{it-def-t1}.
\end{defn}
Tropical cones are, essentially, special cases of semimodules over the tropical
semifield. Recall that semimodules (modules over semirings)
are defined in a way similar to modules over rings, see~\cite{cgq02}.
In particular, semimodules over idempotent semirings have been studied under the name of {\em idempotent spaces} in~\cite{litvinov}.
A simple example of semimodule over $\tmax$ is the $n$-fold Cartesian product of $\tmax$, $\tmaxn{n}$; the internal law is
$(x,y) \mapsto x\vee y:= (x_i\vee y_i)_{i\in [n]}$, for $x,y\in \tmaxn{n}$,
and the action of $\tmax$ on $\tmaxn{n}$ is defined by
$(\lambda,x) \mapsto \lambda + x$. This yields a free, finitely
generated semimodule. If $C\subset\R^{n}$ is a tropical cone,
then $C\cup\{(-\infty,\dots,-\infty)\}$ is a subsemimodule
  of $\tmaxn{n}$, and vice versa.

  We shall consider, in particular, tropical cones
 satisfying a topological assumption. We equip $\tmax$ 
with the topology defined by the metric $(a,b)\mapsto |e^a-e^b|$.
The semimodule $\tmaxn{n}$, equipped with the topology
of the metric $d_\infty(x,y) = \max_{i\in[n]}|e^{x_i}-e^{y_i}|$
is a topological semimodule (meaning that the structure laws are continuous).
Observe
that the induced topology on $\R^n\subset \tmaxn{n}$
is the Euclidean topology. Dual considerations
apply to $\tminn{n}$.

\begin{definition}\label{def-lowerclosure}
Given a subset $C\subset \R^n$,
we define the {\em lower closure} of $C$, $\lowerclosure{C}\subset \tmaxn{n}$, to be the set of limits of nonincreasing sequences
of elements of $C$.
Similarly, we define the {\em upper closure}
$\upperclosure{C}\subset \tminn{n}$ to be the set of limits of nondecreasing sequences of elements of $C$.
\end{definition}
For instance, if $C=\{x\in \R^2\mid |x_1-x_2|\leq 1\}$,
$\lowerclosure{C}=C\cup\{(-\infty,-\infty)\}$, whereas if
$C=\{x\in\R^2\mid x_1\geq x_2\}$, $\lowerclosure{C} = \{x\in \tmaxn{2}\mid x_1\geq x_2\}$. 
The following proposition shows a correspondence between closed tropical cones
and a class of closed tropical subsemimodules.
\begin{prop}\label{prop-related}
  The map $C\mapsto V:=\lowerclosure{C}$ establishes a bijective
  correspondence between the nonempty closed tropical cones $C\subset \R^n$ and the
  closed subsemimodules $V$ of $\tmaxn{n}$ such that $V\cap \R^n\neq\emptyset$.
  The inverse map is given
  by $V\mapsto V\cap \R^n$.
\end{prop}
The dual property, concerning $\upperclosure{C}$ and $\tminn{n}$
instead of $\lowerclosure{C}$ and $\tmaxn{n}$,
also holds. 
We denote by $\operatorname{supp}y:=\{i\in[n]\mid y_i>-\infty\}$ the {\em support} of a vector $y \in\tmaxn{n}$.
\begin{proof}
  Suppose $C\subset \R^n$ is a tropical cone. Since $\tmaxn{n}$
  is a topological semimodule, $\lowerclosure{C}$ is a semimodule over $\tmax$.
  We next show that $\lowerclosure{C}$ is closed.
Let $x^k$ denote a sequence of elements of $\lowerclosure{C}$
converging to an element $x\in \tmaxn{n}$.
Let $\epsilon_k$
denote an arbitrary sequence of positive numbers decreasing to $0$.
Let $y^k:= x^k+\epsilon_k e\in \lowerclosure{C}$. Observe that
$\eta_k:=d_\infty(x^k,y^k) = (e^{\epsilon_k}-1) \max_{i\in[n]}e^{x^k_i}
\to 0$ as $k\to \infty$. Moreover, 
by definition of $\lowerclosure{C}$, we can find $z^k\in C$
such that $y^k\leq z^k$ and $d_\infty(y^k,z^k)\leq \eta_k$.
It follows that the sequence $z^k$ also converges to $x$.
We claim that there is a subsequence $z^{n_k}$ of $z^k$
that is nonincreasing. Indeed,
suppose by induction that $z^{n_1}\geq \cdots\geq z^{n_k}$
has already been selected. 
Let $I:=\supp x$. Since $z^l\to x$ as $l\to\infty$,
we get $z^{l}_i \to x_i\in\R$ for $i\in[n]$,
and $z^l_i\to-\infty$
for $i\in [n]\setminus I$. However, by construction,
$z^{n_k} \geq \epsilon_{n_k}e + x$, and so $z^{n_k}_i \geq \epsilon_{n_k}+x_i$
for all $i\in I$. We deduce that there is an index $l$ such that
$z^l\leq z^{n_k}$, and we set $n_{k+1}:=l$. This shows that
$x\in\lowerclosure{C}$, and so, $\lowerclosure{C}$ is closed.

Conversely, suppose that $V$ is a closed subsemimodule of $\tmaxn{n}$.
Then, it is immediate that $V\cap \R^n$ is a closed tropical cone.
It remains to show that the correspondence is bijective.
We have trivially $\lowerclosure{C}\cap \R^n=C$ for all nonempty closed
tropical cones $C\subset \R^n$. Conversely, if $V$ is a closed
tropical subsemimodule of $\tmaxn{n}$ such that $V\cap \R^n$ is nonempty,
we must show that $V=\lowerclosure{(V\cap\R^n)}$.
let $x\in V$, and let us choose an arbitrary
element $y\in V\cap \R^n$.
Then, the path $\lambda \mapsto \gamma(\lambda):=x\vee (\lambda +y)$,
defined for $\lambda\in (-\infty,0]$ is such that $\gamma(\lambda)\in V\cap\R^n$, and $\lim_{\lambda\to -\infty} \gamma(\lambda)= x$.
It follows that $x\in \lowerclosure{(V\cap\R^n)}$, hence,
$V\subset \lowerclosure{(V\cap\R^n)}$. The other inclusion
is immediate. 
\end{proof}

Recall that an element $u$ of a tropical subsemimodule $V\subset \tmaxn{n}$
is an {\em extreme generator} of $V$ if $u=v\vee w$ with $v,w\in V$
implies that $u=v$ or $u=w$. A {\em tropical linear combination}
of elements of $V$ is a vector of the form $\vee_{i\in I}( \lambda_i + a_i)$
where $(\lambda_i)_{i\in I}\subset \tmax$ and $(a_i)_{i\in I} \subset
V$ are finite families. We say that $G\subset V$ is a {\em tropical
generating set} if every element of $V$ is a tropical linear
combination of a family of elements of $G$.
We say also that two vectors are {\em tropically
proportional} if they differ by an additive constant.
The next result summarizes results
from \cite{GK06,BUTKOVIC2007394};
it shows that a closed tropical subsemimodule of $\tmaxn{n}$
is generated by its extreme rays.
 \begin{thm}[See Theorem~3.1 in~{\cite{GK06}} or Theorem~14 in~\cite{BUTKOVIC2007394}]\label{th-sum}
   Suppose that $V$ is a closed tropical subsemimodule of $\tmaxn{n}$.
   Then, every element of $V$ is a tropical linear combination
   of at most $n$ extreme generators of $V$. Moreover,
   these extreme generators are characterized as follows.
   For all $i\in [n]$, let $V_i:= \{x\in V\mid x_i = 0\}$,
    and let $\operatorname{Min}V_i$ denote the set of minimal
    elements of $V_i$. Then, every extreme generator of $V$
    is tropically proportional to an element of 
    $\cup_{i\in [n]} \operatorname{Min}V_i$.
 \end{thm}
 This is reminiscent of the classical Carath\'eodory theorem
 for closed convex pointed cones.

\subsection{Alcoved polyhedra and metric closures}

An important class of tropical cones and dual tropical cones consists
of {\em alcoved polyhedra}.
The latter were introduced in~\cite{AlcovedPolytopes}:
in general, an alcoved
polyhedron associated to a root system is a polyhedron whose
facets have normals that are proportional to vectors of this root system.
Here, the root system is $A_n$, the collection of vectors $\{e_i-e_j\mid i,j\in[n], i\neq j\}$, where $e_i$ denotes the $i$th vector of the canonical
basis of $\R^n$.
\begin{defn}
  An {\em alcoved polyhedron}~\cite{AlcovedPolytopes}
  is a
  polyhedron
  of the form
    \begin{align}\label{e-def-alcoved}
  \alcoved{M}=\{x\in \R^n\mid   x_i \geq M_{ij} + x_j, \quad \forall 1\leq i,j\leq n\}
  \end{align}
  for some matrix $M=(M_{ij})\in \tmaxn{n\times n}$.
\end{defn}
    {\em Order polyhedra} are remarkable examples of alcoved polyhedra. They are of the form $\{x\in \R^n\mid x_i \geq x_j \text{ if } (i,j) \in E\}$ where $E\subset [n]\times [n]$ is a partial order relation on the set $[n]$.
    Intersection of order polyhedra with the hypercube $[0,1]^n$ are known as
    {\em order polytopes}, they were studied by Stanley~\cite{stanley86}.

We shall denote by $\vee$ the tropical
addition of matrices, so that, for all
$A,B\in \tmaxn{m\times n}$, $(A_{ij})\vee (B_{ij}):= (A_{ij}\vee B_{ij})\in \tmaxn{m\times n}$. The tropical multiplication of matrices will be denoted
by concatenation, i.e, for $A\in \tmaxn{m\times n}$ and $B\in \tmaxn{n\times p}$,
$AB\in \tmaxn{m\times p}$ is the matrix with $(i,j)$-entry $\vee_{k\in[n]} (A_{ik}+B_{kj})$.
Then, when $m=n$, for all $r$,
the {\em $r$th tropical power} of $A$ is denoted by 
$A^r:= A\cdots A$ ($A$ is repeated $r$ times).

There are well known relations between alcoved polyhedra
and operations of metric closures which we next
recall. The {\em tropical Kleene star} of $M$ is defined by
$M^* := I \vee M \vee M^2 \vee \cdots  $.
This supremum may be infinite,  i.e., in general
$(M^*)_{ij}$ may take the value $+\infty$.
Recall that to the matrix $M$ is associated
a digraph with set of nodes $[n]$ and an arc $i\to j$
of weight $M_{ij}$ whenever $M_{ij}>-\infty$. Then,
$(M^k)_{ij}$ yields the maximal weight of a path
of length $k$ from $i$ to $j$, and $(M^*)_{ij}$
yields the supremum of the weights of paths
from $i$ to $j$, of arbitrary length.
We have $(M^*)_{ij}< +\infty$
for all $i,j$ if and only if there is no circuit with positive weight
in the digraph of $M$. Then,
$%
M^* = M^0 \vee \dots \vee M^{n-1}$,
see e.g.\ Prop~2.2 of~\cite{agw04}.

The {\em critical circuits}
of the matrix $M^*$ are the circuits in the digraph
of $M^*$ with weight $0$. The union of the critical
circuits constitutes the {\em critical digraph}.
The following result is well known: the first
statements follow readily from the results recalled
above, whereas the characterization of generators
is a special case of the characterization of tropical
eigenspaces, see e.g.~\cite[Th.~3.100]{BCOQ92}, Theorem~6.4 of~\cite{agw04},
or~\cite[Th.~4.4.5]{butkovic}.

\begin{lem}\label{prop-def-alcoved}
  The polyhedron $\alcoved{M}$ is non-empty if and only there is no circuit of positive
  weight in the digraph of $M$. Then,
  \begin{align}
    \label{e-alcoved2}
    \alcoved{M} =\{x\in \R^n\mid M^* x\leq x\} = \{x\in \R^n\mid M^*x = x\}
= \{M^* y\mid y\in \R^n\}  
    \enspace .
  \end{align}
  Moreover, a tropical generating set of $\lowerclosure\alcoved{M}$
  is obtained as follows: denote by $C_1,\dots,C_s$ the strongly
  connected components of the critical digraph of $M^*$,
  and select indices $i_1\in C_1,\dots,i_s\in C_s$ in an arbitrary
  manner. Then, the set of columns of $M^*$ indexed
  by $i_1,\dots,i_s$ tropically generates $\lowerclosure\alcoved{M}$,
  and every generating set of $\lowerclosure\alcoved{M}$
  contains at least one scalar multiple of each of these
  columns.\hfill\qed
\end{lem}
\begin{remark}
  It follows from \Cref{prop-def-alcoved} that the minimum number of elements
  of a tropical generating family of the tropical semimodule
    $\lowerclosure\alcoved{M} \subset \tmaxn{n}$
never exceeds $n$. The same observation
applies to dual tropical generating families
of $\upperclosure\alcoved{M}$. In contrast, there
are finitely generated tropical subsemimodules
of $\tmaxn{n}$ with an arbitrarily large number
of generators, see e.g.~\cite{AlGG09}.
\end{remark}

\section{Shapley operators and Ambitropical Cones}
\subsection{Abstract Shapley Operators}
Shapley operators are dynamic programming operators allowing one
to compute the value function of zero-sum games. The typical example
of Shapley operator $T: \R^n\to \R^n$ is of the form
\begin{align}
  T_i(x) = \inf_{a\in A_i} \sup_{b\in B_i} (r_i^{ab} + \sum_{j\in [n]} P_{ij}^{ab} x_j) \enspace,
  \label{e-def-shap}
\end{align}
where $[n]=\{1,\dots,n\}$ is the state space, $A_i, B_i$ are the sets of actions
available in state $i$, of the two players, called ``Min'' and ``Max'',
$r_i^{ab}$ is a payment made by by Player Min to Player Max at a given stage,
assuming that Min selected action $a$ and that Max selected action
$b$, and $P_{ij}^{ab}\geq 0 $ is the probability of transition
from $i$ to $j$, so that $\sum_j P_{ij}^{ab}=1$. Such operators
capture zero-sum perfect or ``turned based'' information games, without
discount. In the original model considered by Shapley, 
the two players play simultaneously and the actions are randomized,
which can be cast as~\eqref{e-def-shap}, in which
$A_i$ and $B_i$ are simplices~\cite{rosenbergsorin,neymansurv}.\todo{give references}
Actually, many variants of Shapley operators (depending on the nature
of the turns and on the information structure) can be considered,
and so, it will be convenient to introduce a general
definition.

\begin{defn}[Shapley operator]\label{def-shapley}
An (abstract) \textit{Shapley operator} is a map $T : \mathbb{R}^{n} \to \mathbb{R}^{p}$ with the following properties:
\begin{enumerate}
\item\label{i-mon} $x \leq y$ implies $T(x) \leq T(y)$, for all $x,y\in \R^n$;
\item\label{i-hom} $T(\lambda + x)=\lambda + T(x)$, for all $x\in \R^n$ and $\lambda\in \R$.
\end{enumerate}
\end{defn}
The example~\eqref{e-def-shap} of Shapley operator obviously satisfies these properties.
Conversely, Kolokoltsov showed that every abstract Shapley operator $\R^n\to\R^n$ can be written as~\eqref{e-def-shap} (see~\cite{KM}) and Singer and Rubinov~\cite{RS01b} showed that the transition probabilities can even be chosen to be $0/1$. 

Note that, taking into account the semimodule structure of $\R^n$ and $\R^p$, the canonical choice of morphisms to consider would be tropically linear maps (maps that commute with the supremum and with the addition of a constant). {\em Shapley operators} constitute a larger class of maps and they are precisely canonical morphisms of additive cones (see Def.~\ref{def-shapley}). 

In the square case, i.e., when $n=p$, Shapley operators 
arise as dynamic programming operators of two-player zero-sum games, see e.g.~\cite{rosenbergsorin,neymansurv}.
It is known that Shapley operators are nonexpansive in the sup-norm;
this observation plays a key role in the ``operator approach'' of zero-sum games~\cite{crandall,rosenbergsorin,neymansurv,GV10}. Furthermore, the following
observation, made in~\cite[Prop.~1.1]{gunawardenakeane}, shows
that Shapley operators are characterized by a nonexpansiveness property.
Recall that $\mytop({x}):=\max_{i\in [n]} x_i$ denotes the ``top''
hemi-norm of a vector $x\in \R^n$. It will also be convenient
to use the notation $\mybot({x}):=-\mytop{(-x)}=\min_{i\in [n]}x_i$.
\begin{prop}[{\cite[Prop.\ 1.1]{gunawardenakeane}}]\label{prop-equiv}
  Let $T:\R^n\to\R^p$. The following
  assertions are equivalent:
  \begin{enumerate}
    \renewcommand{\theenumi}{\roman{enumi}}
  \item\label{i-shapley} $T$ is a Shapley operator;
  \item\label{i-nexp} $\mytop{(T(x)-T(y))}\leq \mytop{(x-y)}$ for all $x,y\in \R^n$;
      \item\label{i-nexp2} $\mybot{(T(x)-T(y))}\geq \mybot{(x-y)}$ for all $x,y\in \R^n$.
  \end{enumerate}
\end{prop}
The following result is a consequence of a general
result on \cite{burbanks} concerning the continuous extension
of positive homogeneous maps defined on the interior of polyhedral cones.
\begin{prop}[Corollary of \protect{\cite[Theorem 3.10]{burbanks}}]
\label{prop-extends}
  A Shapley operator $T: \R^n\to \R^p$ admits a unique
  continuous extension $T_-: \tmaxn{n}\to \tmaxn{p}$,
  given by
  \[
{T}_-(x) =\inf\{T(y)\mid y\geq x,\;y\in \R^n\} \enspace .
\]
Similarly, $T$ has a unique continuous extension
$\tminn{n}\to \tminn{p}$, given by
  \[
{T}_+(x) =\sup\{T(y)\mid y\leq x,\;y\in \R^n\} \enspace .
\]
\end{prop}
\begin{remark}
\Cref{prop-extends} provides only one-sided extensions of $T$,
either to $(\R\cup\{-\infty\})^n$ or
to $(\R\cup\{+\infty\})^n$.
A Shapley operator defined on $\R^n$ generally does not extend
canonically to $(\R\cup\{\pm\infty\})^n$ (consider $n=2$ and $T_i(x)= ((x_1+x_2)/2 )$ for $i=1,2$).
\end{remark}
A Shapley operator is said to be {\em tropically linear}
if $T(x\vee y) = T(x)\vee T(y)$ holds for all $x,y\in \R^n$.
It is said to be {\em dual tropically linear}
if $T(x\wedge y) = T(x)\wedge T(y)$.
Denoting by $\tropicalbasisvector_j:=(-\infty,\dots,-\infty,0,-\infty,\dots$, $-\infty)$ (with $0$ in the $j$th position) the $j$th vector of the tropical canonical basis
of $\tmaxn{n}$, and $M_{ij}:= (T_{-}(\tropicalbasisvector_j))_i$, we see that if $T$ is tropically linear, then
\[
(T(x) )_i = \vee_{j\in[n]}(M_{ij} + x_j ) \enspace ,
\]
i.e., $T$ is represented by a matrix product. A similar representation
holds for dual tropically linear Shapley operators.

\subsection{Ambitropical cones}\label{sec-ambitropical}
We next introduce our main object of study: ambitropical cones.
We are looking for a class of objects which includes tropical
cones and their duals. This leads to the following
definitions.
We recall that $\R^n$ is equipped with the standard partial order.
\begin{defn}\label{def-ambi}
  An {\em ambitropical cone} is a non-empty additive cone $C$ of $\R^n$ such that 
$C$ is a lattice in the induced order of $(\R^n,\leq)$.
\end{defn}
Recall that $C$ being a lattice means that every two elements $x,y$
of $C$ have a {\em least  upper bound}
\[ x\vee^C y := \min \{z\in C\mid z\geq x, \; z\geq y\},
\]
and a {\em greatest lower bound}
\[ 
x\wedge^C y = \max \{z\in C\mid z\leq x, \; z\leq y\},
\]
where the symbols ``max'' and ``min'' indicate
the greatest and smallest elements of a set.
Similarly, we will use the notation
$\supc X$ and $\infc X$ for the least upper bound and greatest
lower bound in $C$ of a subset $X\subset C$, when it exists. In particular,
$x\vee^c y = \supc \{x,y\}$ and
$x\wedge^c y = \infc \{x,y\}$. The operations $\supc$ and
$\infc$ defined for subsets of $C$
should not be confused with the restriction of the operations
$\sup=\operatorname{sup}^{\R^n}$ and $\inf=\operatorname{inf}^{\R^n}$
defined for subsets of $\R^n$: indeed, for all $X\subset C$ that has a least upper bound in $C$, $\supc X \geq \sup X$, and similarly $\infc X \leq \inf X$
if $X$ has a greatest lower bound in $C$.
In other words, an ambitropical cone is a lattice but it {\em may not be a sublattice} of $\R^n$.

We shall especially consider ambitropical cones
satisfying the following property.

\begin{defn}[Conditionally complete lattice]
A lattice $L$ is said \textit{conditionally complete} if every nonempty subset of $L$ that has an upper bound has a join (a least upper bound), and if every nonempty subset of $L$ that has a lower bound has a meet (a greatest lower bound).
\end{defn}

The following observation is elementary.
\begin{lem}\label{lem-arch}
  Let $C$ be an ambitropical cone.
  A subset of $\R^n$ is bounded from above by an element of $\R^n$
  if and only if it is bounded from above by an element of $C$.
  The dual statement holds for subsets bounded from below.
\end{lem}
\begin{proof}
  Let $X\subset \R^n$ and suppose that there exists
  $u\in\R^n$ such that $x\leq u$ holds for all $x\in X$.
  Let $y\in C$. Then, $u\leq z:=y + \mytop(u-y)$. It follows
  that $x\leq z\in C$ holds for all $x\in X$.
\end{proof}

\begin{prop}\label{prop-closed}
  An ambitropical cone is a conditionally complete lattice if and only
  if it is closed in the Euclidean topology.
\end{prop}
\begin{proof}
  Suppose that the ambitropical cone $C$ is closed in the Euclidean topology,
  and let $X\subset C$ be a nonempty set bounded above by some element $y\in C$.
  Let $\finitesets{X}$ denote the set of nonempty finite subsets of $X$.
  For all $F\in \finitesets{X}$, let $u_F:= \supc F$. Then,
  $(u_F)_{F\in\finitesets{X}}$ is a nondecreasing net of elements of $C$, bounded above by $y$. Since $C$ is closed
  in the Euclidean topology, the limit of a net of elements of $C$ belongs to $C$,
  and so $u:=\lim_F u_F \in C$.
  By construction, $u\geq x$ holds for all $x\in X$. Moreover, if $z\in C$ is an
  upper bound of $X$, we get $z \geq u_F$ for all $F\in\finitesets{X}$, and so $z\geq u$. This
  shows that $u$ is the least upper bound of $X$. A dual
  argument works for greatest lower bounds. Hence, $C$ is
  a conditionally complete lattice.

  Conversely, suppose that $C$ is a conditionally complete lattice. Observe
  that for every bounded sequence $(x_k)$ of elements of $C$,
  the following ``liminf'' and ``limsup'' constructions both define elements
  that belong to $C$:
  \[
  \limsupc_{k\to\infty} x_k:= \infc_{k\geq 1}\supc_{\ell \geq k} x_\ell ,\qquad
  \liminfc_{k\to\infty}x_k:= \supc_{k\geq 1}\infc_{\ell \geq k} x_\ell
  \enspace .
  \]
  We shall use the fact that $  \limsupc_{k\to\infty} x_k \geq   \liminfc_{k\to\infty}x_k$.
  This inequality, which is standard when $C=\R^n$, is still valid
  in general. Indeed, for all $k,m\geq 1$, we have
  $\supc_{\ell \geq k} x_\ell \geq
  \infc_{\ell \geq m} x_\ell$, and so
  $\supc_{\ell \geq k} x_\ell \geq
  \supc_{m\geq 1} \inf_{\ell \geq m} x_\ell=\liminfc_{r\to\infty}x_r$.
Hence,
  \begin{align}
    \limsupc_{k\to\infty}x_k = \infc_{k\geq 1}\supc_{\ell\geq k} x_\ell
    \geq \liminfc_{r\to\infty}x_r
\enspace .\label{e-compar}
\end{align}
  Suppose
  that the sequence $(x_k)_{k\geq 1}$ of elements of $C$ converges
  to $x\in \R^n$. Then, for all $\epsilon>0$, there exists an index
  $m$ such that $\|x_\ell -x\|\leq \epsilon$ for all $\ell \geq m$.
  In particular, $x_\ell \leq \|x_\ell -x_{m}\| + x_{m} \leq  2\epsilon + x_{m}$.
  We deduce that $\limsupc_{\ell\to\infty} x_\ell \leq 2\epsilon + x_m \leq 3\epsilon +x $.
  Since the latter inequality holds for all $\epsilon>0$,
  we deduce that $\limsupc_{\ell\to\infty} x_\ell \leq x$. A dual
  argument sows that $\liminfc_{\ell\to\infty}x_\ell \geq x$.
  Using~\eqref{e-compar},
  we conclude that $x=   \limsupc_{\ell\to\infty} x_\ell= \liminfc_{\ell\to\infty}x_\ell
  \in C$, showing that $C$ is closed in the Euclidean topology.
\end{proof}
In the sequel, when writing that an ambitropical is {\em closed}, we shall
always refer to the Euclidean topology.

We define, for all closed ambitropical cones $C$, and for all $x\in\R^n$,
the following canonical retractions:
\begin{align}
   \qm{C}(x):= \supc \{y \in C\mid  y\leq  x\} \enspace ,
\qquad
\qp{C}(x):= \infc \{y \in C\mid  y\geq  x\} \enspace.
\label{e-def-proj}
\end{align}
We shall denote by $\range f:=\{f(x)\mid x \in X\}$ the {\em image} or {\em range} of a map $f:X\to Y$.
Recall that a {\em retraction} %
onto $C$ is a map $P:\R^n\to C$ that is a continuous, idempotent map from $\R^n$ to $\R^n$ with range $C$. 
Then, $C$ is a retract
of $\R^n$. The following result shows that any closed ambitropical
cone is a Shapley retract, i.e., the image of a retraction
that is a Shapley operator.

\begin{prop}\label{idem-shapley}
  Suppose that $C$ is a closed ambitropical cone of $\R^n$.
  Then, $\qm{C}$ is an idempotent Shapley operator,
  i.e., $(\qm{C})^2=\qm{C}$, and the range of $\qm{C}$ is
  $C$.   The same is true for $\qp{C}$.
\end{prop}
\begin{proof}
  Since $C$ is conditionally complete, for all $x\in \R^n$,
  $\qm{C}(x)$ is well
  defined and $\qm{C}(x)\in C$. Moreover, $\qm{C}$ trivially
  fixes $C$, implying that $\range \qm{C}=C$ and $(\qm{C})^2=\qm{C}$.
  We also have, for all $x\in \R^n$ and $\lambda\in \R$,
  $   \qm{C}(\lambda + x)=
  \supc \{y \in C\mid  y\leq  \lambda + x\}
  =\supc \{y \in C\mid  -\lambda + y\leq  x\}
  =  \supc\{\lambda + z\in C\mid z\leq x\}
  = \lambda + \qm{C}(x)$. The operator $\qm{C}$
  is trivially order preserving, hence it is a Shapley operator.
    Dual arguments apply to $\qp{C}$.
  \end{proof}

\begin{thm}
Let $C$ be a closed ambitropical cone contained in $\R^n$. The set of Shapley retractions onto $C$ (i. e. idempotent Shapley operators with range $C$) constitutes a complete lattice, with bottom element $\qm{C}$ and top element $\qp{C}$.
\end{thm}

\begin{proof}
 First of all, we shall prove that for any Shapley retraction $P$ onto $C$ and for every $x \in \R^n$ $\qm{C}(x) \leq P(x) \leq \qm{C}(x)$ . 
 Let $y\in C$ such that $y\leq x$, we get $y=P(y) \leq P(x)\in C$,
  and so $\qm{C}(x) = \supc \{y\in C\mid y\leq x\} \leq P(x) $.  
 Similarly, $P(x) \leq \qp{C}(x)$.
 Let $(Q_{\alpha})_{\alpha \in A} : \R^n \to \R^n$ be a collection of Shapley retractions onto $C$.
 Since $\qm{C}(x) \leq Q_{\alpha}(x) \leq \qp{C}(x)$ for all $x \in \R^n$, the family 
 \[
 (Q_{\alpha}(x))_{\alpha \in A} \subset C
 \]
 is bounded from above and from below. 
 By \Cref{prop-closed} $C$ is conditionally complete, which allows us to define 
  \[
 Q(x) = \supc \{ Q_{\alpha}(x) \mid \alpha \in A\}.
 \]
 
 It is immediate that $Q$ is a Shapley operator that fixes $C$ and that $C$ is its range. So, $Q$ is the least upper bound of the family  $(Q_{\alpha})_{\alpha \in A}$ in the set of Shapley retractions.
\end{proof}

\Cref{idem-shapley} shows that closed ambitropical cone are
Shapley retracts of $\R^n$. The following result shows
that we have actually an equivalence.

\begin{thm}\label{th-main0}
  Let $E$ be a subset of $\R^n$. The following assertions are equivalent
  \begin{enumerate}    %
    \item\label{ambitropical} $E$ is a closed ambitropical cone;
    \item\label{shapleyretract}  $E$ is a Shapley retract of $\R^n$;
    \item\label{fpshapley} $E$ is the fixed point set of a Shapley operator; 
  \end{enumerate}
\end{thm}

\begin{proof}
  \myitem{ambitropical} $\Rightarrow$ \myitem{shapleyretract}.
  If $E$ is a closed ambitropical cone, then, by~\Cref{idem-shapley},
  $E$ is the image of $\qm{E}$ and $\qm{E}$ is an idempotent Shapley operator.

  \myitem{shapleyretract} $\Rightarrow$ \myitem{fpshapley} is trivial.

  \myitem{fpshapley} $\Rightarrow$ \myitem{ambitropical}. Since $E$ is the fixed point set of a Shapley operator $T$, $E$ is a closed cone. We shall now show that it is also a lattice in the induced order.
Let $x, y \in E$ and we claim that $x \vee_E y = \lim_{k \to \infty} T^k (x \vee_{\R^n} y)$. We have that $x, y \leq x \vee_{\R^n}y$, so $x=T(x) \leq T(x \vee_{\R^n}y)$ and $y = T(y) \leq T(x \vee_{\R^n} y)$. 
Applying again $T$ to these inequalities and passing to the limit, we obtain that $\lim_{k \to \infty} T^k (x \vee_{\R^n} y)$ is an upper bound of $x$ and $y$. 
Let $z \in E$ such that $x, y \leq z$. Then $x \vee_{\R^n} y \leq z$ and $T(x \vee_{\R^n} y) \leq T(z)=z$; applying again $T$ to this inequality and passing to the limit, we obtain that $\lim_{k \to \infty} T^k (x \vee_{\R^n} y) \leq z$. 
We have now to prove that $\lim_{k \to \infty} T^k (x \vee_{\R^n} y) \in E$. $T(\lim_{k \to \infty}T^k (x \vee_{\R^n} y))=\lim_{k \to \infty} T^k (x \vee_{\R^n}y)$ by definition. So, $\lim_{k \to \infty} T^k (x \vee_{\R^n} y)$ is a fixed point of $T$ and it belongs to $E$.
The case of $\inf$ is dual.

\end{proof}

\section{From tropical hulls to ambitropical hulls}
\subsection{Decomposition of canonical retractions in terms of projections on tropical cones}
Appropriate tropical analogues of Hilbert's spaces are obtained by considering spaces that are closed
by taking suprema~\cite{litvinov,cgq02}.
Considering suprema of bounded sets
leads to the notion of {\em b-complete idempotent spaces}
in~\cite{litvinov}, whereas allowing unconditional
suprema leads to the notion
of {\em complete semimodules}~\cite{cgq02}.

Hence, we shall perform a (one sided) conditional
completion.
If $E$ is a nonempty subset of $\R^n$, we shall denote by $E^{\max}$ the
subset of $\R^n$ consisting of
tropical linear combinations of {\em possibly infinite} families
of elements of $E$,
i.e., the set of elements of the form
\begin{align}
\sup \{\lambda_f + f\mid f\in E\}
\label{e-complete-trop}
\end{align}
where the $\lambda_f\in\tmax$ are such that the family of elements
$(\lambda_f+f)_{f\in E}$ is bounded from above
and the $\lambda_f$ are not identically $-\infty$.
Up to the adjunction of a bottom element, the set $E^{\max}$ is the
b-complete idempotent
space generated by $E$ in the sense of~\cite{litvinov}.

We shall also need to consider the $\tmax$-semimodule obtained
by taking the {\em lower closure} of $E^{\max}$, a notion
already introduced in~\Cref{def-lowerclosure}:
\[
\bar{E}^{\max} := \lowerclosure{E^{\max}} \enspace .
\]
Similarly, we shall denote by $E^{\min}$ the
set of elements of the form
$\inf \{\lambda_f + f\mid f\in E\}$
where the $\lambda_f\in\tmin$ are such that the family of elements
$(\lambda_f+f)_{f\in E}$ is bounded from below,
and the $\lambda_f$ are not identically $+\infty$.
We also set $\bar{E}^{\min} := \upperclosure{E^{\min}}$.
\begin{prop}\label{prop-severalclosure}
  Let $C\subset \R^n$ be an additive cone. Then, the following statements
  are equivalent:
  \begin{enumerate}
  \item\label{ip-closed} $C$ is closed;
  \item\label{ip-nondec} $C$ is stable by limits of bounded nondecreasing sequences;
  \item\label{ip-noninc} $C$ is stable by limits of bounded nonincreasing sequences.
    \end{enumerate}
\end{prop}

\begin{proof}
  The implication  (\ref{ip-closed})$\Rightarrow$(\ref{ip-nondec}) is trivial.

  We next show that (\ref{ip-nondec})$\Rightarrow$(\ref{ip-noninc}).
  Let $x_k$ be a bounded nonincreasing sequence of elements of $C$
  converging to $x\in \R^n$.
  Consider the sequence $y_k:= x_k -2\|x-x_k\|_\infty e\in C$.
  We have $y_k \leq - \|x-x_k\|_\infty e +x$. Moreover, $y_k$
  also converges to $x$. It follows that for all $k$, %
  we can
  find an index $l\geq k$ such that $y_l \geq y_k$. Hence, we
  can construct a nondecreasing subsequence $y_{n_k}$ converging to $x$.
Applying (\ref{ip-nondec}), we conclude that $x\in C$.

  We finally show that (\ref{ip-noninc})$\Rightarrow$(\ref{ip-closed}).
  Suppose $x_k$ is a sequence of elements of $C$
  converging to $x\in \R^n$. Consider now $y_k:= 2\|x-x_k\|_\infty +x_k $.
  Then, arguing as in the proof of the previous implication,
  we deduce that we can construct a nonincreasing
  subsequence $y_{n_k}$ still converging to $x$. Applying~(\ref{ip-noninc}),
  we conclude that $x\in C$.
  \end{proof}

\begin{cor}\label{lem-closed}
  Let $E$ denote a non-empty subset of $\R^n$. Then,
  $E^{\max}$ is a closed tropical cone. Similarly,
 $E^{\min}$ is a closed dual tropical cone.
\end{cor}
\begin{proof}
  By definition, $E^{\max}$ is a tropical cone.
  Let us consider a bounded nondecreasing sequence $x_k\in E^{\max}$.
  We can write $x_k = \sup\{ \lambda_f^k + f \mid f\in E\}$
  where for each $k$, the family $(\lambda_f^k)_{f\in F}$ is not identically
  $-\infty$.  Let $x:=\lim_k x_k=\sup_k x_k \in \R^n$.
  From $\lambda_f^k + f\leq x_k \leq x$, we deduce that  $\lambda_f^k \leq \mybot(x-f)$. So, the sequence $(\lambda_f)_{k\geq 1}$ is bounded from above. It follows
  that $\lambda_f:= \sup\{\lambda_f^k\mid k \geq 1\}<+\infty$. Moreover,
  using the associativity of the supremum operation, we get $x= \sup x_k =\sup \{\lambda_f + f\mid f\in E\}\in E^{\max}$. Hence, $E^{\max}$ is stable by limits
  of nondecreasing sequences. It follows from \Cref{prop-severalclosure}
  that $E^{\max}$ is closed in the Euclidean topology.
  A dual argument applies to $E^{\min}$.
  \end{proof}
  
  If $C$ is a closed tropical cone, then $C$
is closed by tropical linear combinations,
and it is also closed by taking the supremum
of nondecreasing sequences, it follows
that the supremum $\supc$ relative to $C$
coincides with the supremum %
of $\R^n$. We deduce the following result.

\begin{lem}\label{lem-compar}
  If $C$ is a closed
  tropical cone, then $\qm{C} \leq I$. Similarly,
  if $C$ is a
  closed
  dual tropical cone, then $\qp{C} \geq I$.\hfill\qed
\end{lem}

\begin{cor}
  Let $E$ denote a non-empty subset of $\R^n$. 
  Then, $\bar{E}^{\max}$ is a closed subsemimodule of $\tmaxn{n}$.
  Similarly, $\bar{E}^{\min}$ is a closed subsemimodule of $\tminn{n}$
\end{cor}
\begin{proof}
 By definition, $E^{\max}$ is a tropical cone, and 
 by \Cref{lem-closed}, it is a closed subset of $\R^n$.
 Then, it follows from \Cref{prop-related} that
 $\bar{E}^{\max}=\lowerclosure E^{\max}$ is a closed
 subsemimodule of $\tmaxn{n}$.
 \end{proof}
For all nonempty subsets $E$ of $\R^n$, and for all $x\in \R^n$, we define the tropical projections
\[\pmax{E}(x):= \sup\{ y \in E^{\max} \mid   y\leq x\}\enspace,
\qquad
\pmin{E}(x) := \inf\{y \in E^{\min}\mid  y \geq x\}
\enspace .
\]
This is a specialization
of the notion of projectors $\qm{C}$
and $\qp{C}$ to $C=E^{\max}$ or $C=E^{\min}$,
introduced in~\eqref{e-def-proj}.
Indeed, if $C=E^{\max}$, the operation $\supc$ coincides
with the ordinary supremum $\sup$ of $\R^n$. The dual
property holds for $C=E^{\min}$. If $G$ is any tropical generating
set of $E^{\max}$, then, we have the explicit representation%
\begin{align}\label{explicit-pmax}
\pmax{E}(x) = \sup_{g\in G} \Big(\mathbf{b}(x-g) + g\Big)  \enspace , \qquad \forall x\in \R^n \enspace,
\end{align}
see~\cite[Th.~5]{cgq02}, and a dual formula applies to $\pmin{E}$.
The next proposition tabulates elementary properties of these projectors.
\begin{prop}\label{prop-systematic}
  Let $E$ be a nonempty subset of $\R^n$. Then, $\pmax{E}$ and $\pmin{E}$ are Shapley operators from $\R^n \to \R^n$ such that:
  \begin{align}
    \pmax{E} &\leq I, \qquad \pmin{E}
    \geq I  \enspace.
    \label{e-comparepmaxmin}
  \\
    \range \pmax{E} &= E^{\max}, \qquad \range\pmin{E}=E^{\min}\enspace .
    \label{range}
    \\
        \pmax{E}&=(\pmax{E})^2, \qquad
    \pmin{E}=(\pmin{E})^2 \enspace .\label{e-idempotent}
   \end{align}
\end{prop}
\begin{proof}
  The inequalities~\eqref{e-comparepmaxmin} follow from~\Cref{lem-compar}.
  By definition, $\pmax{E}$ fixes $E^{\max}$, and $\pmax{E}(\R^n)\subset E^{\max}$, so $\pmax{E}=(\pmax{E})^2$. The same property holds for $P^{\min}$,
  showing~\eqref{e-idempotent}. The last claim follows from
  the fact that $\pmax{E}$ fixes $E^{\max}\supset E$, and from the same property
  for $\pmin{E}$.
\end{proof}
The maps $\barqm{E}$ and $\barqp{E}$ defined in the next proposition will
play a key role.
\begin{prop}
\label{prop-systematic2}  Let $E$ be a nonempty subset of $\R^n$. 
 Then, the maps
  \[ \barqm{E}:= \pmin{E}\circ \pmax{E}, \quad\text{and}
  \quad
  \barqp{E}:= \pmax{E}\circ \pmin{E}
  \]
satisfy the following properties
  \begin{enumerate}[label=(\roman*)]%
\item \label{shapley} $\barqm{E}$; $\barqp{E}$ are Shapley operators;
\item \label{fix} $\barqm{E}$ and $\barqp{E}$ fix $E$;
\item\label{idemqm}
  \label{idemqp} $(\barqm{E})^2=
    \barqm{E}$; 
    $(\barqp{E})^2=
    \barqp{E}$;
  \item\label{regular}\label{regular2}
    $\barqp{E}\circ \barqm{E} \circ \barqp{E} = \barqp{E}$;
  $\barqm{E}\circ \barqp{E} \circ \barqm{E} = \barqm{E}$;
  \item \label{item-ineq} \label{item-ineq-dual}
    $\barqp{E} \leq \barqm{E}\circ \barqp{E}$;
$\barqm{E} \geq \barqp{E}\circ \barqm{E}$.
    \end{enumerate}
\end{prop}
\begin{proof}

{\em \ref{shapley}}. We showed in \Cref{prop-systematic} that
  $\pmin{E}$ and $\pmax{E}$ are both Shapley operators. The collection of
  Shapley operators is stable by composition.

  {\em \ref{fix}}. This follows from the fact
  that $\pmin{E}$ and $\pmax{E}$ fix $E^{\max}$ and $E^{\min}$,
  respectively (\Cref{prop-systematic}), which both contain $E$.

  {\em\ref{idemqm}}.  Using the second inequality in~\eqref{e-comparepmaxmin},
  and the first equality in~\eqref{e-idempotent},
  we get $(\barqm{E})^2=\pmin{E}\circ \pmax{E}
  \circ \pmin{E}\circ \pmax{E}\geq \pmin{E}\circ \pmax{E}\circ \pmax{E}
  = \pmin{E}\circ \pmax{E}$.
  Using now the first inequality in~\eqref{e-comparepmaxmin}, and the second equality in~\eqref{e-idempotent}, we get
  $\pmin{E}\circ \pmax{E}
  \circ \pmin{E}\circ \pmax{E}\leq \pmin{E}\circ \pmin{E}\circ \pmax{E}
  = \pmin{E}\circ \pmax{E}$, showing that $(\barqm{E})^2=\barqm{E}$.
The second property in~{\em\ref{idemqm}} is dual.
  
    {\em\ref{regular}}.
    Using~\eqref{e-idempotent}, we get
    $\barqp{E}\circ \barqm{E} \circ \barqp{E} =
    \pmax{E}\circ \pmin{E}\circ \pmin{E}\circ \pmax{E}\circ \pmax{E}
    \circ \pmin{E}
    =
    \pmax{E}\circ \pmin{E}\circ \pmax{E}\circ \pmin{E}=
    (\barqp{E})^2= \barqp{E}$,
    by  {\em\ref{idemqp}}. The second property is dual.

  {\em\ref{item-ineq}}.
  We have that $\barqm{E}\circ \barqp{E}=\pmin{E} \circ \pmax{E} \circ \pmax{E} \circ \pmin{E} = \pmin{E} \circ \pmax{E} \circ \pmin{E} \geq \pmax{E} \circ \pmin{E} (z) = \barqp{E}$,
  using $\pmin{E} \geq I$. The second inequality is dual.
  \end{proof}

The next theorem motivates the introduction
of the operators $\barqm{E}$ and $\barqp{E}$
above, it deals with the situation in which $E$
is a closed ambitropical cone.

\begin{thm}[Factorizations of canonical retractions on closed ambitropical cones]
  \label{prop-caracq}
For all closed ambitropical cones $C$,
we have 
\[ \qm{C}  = \barqm{C}
= \pmin{C}\circ \pmax{C}
\quad\text{and}\quad
\qp{C}  = \barqp{C}
= \pmax{C}\circ \pmin{C}
\]
\end{thm}
\begin{proof}
  Suppose $C$ is a closed ambitropical cone.
  Observe first that for all $z$ in $\R^n$,
  \begin{align}
  \pmax{C}(z) = \sup\{x\in C\mid x\leq z\}\enspace .\label{e-supr}
  \end{align}
  Indeed,
  $\pmax{C}(z) = \sup\{x\in C^{\max}\mid x\leq z\}
  \geq \sup\{x\in C\mid x\leq z\}$. However, an element
  $u\leq z$ of $C^{\max}$ can be written as $u=\vee_{y\in C} (\lambda_y + y)$
  with $\lambda_y \in \tmax$, and $\lambda_y +y\leq z$. So,
  $\lambda_y + y \leq \sup\{x\in C\mid x\leq z\}$,
  and so, $u\leq \sup\{x\in C\mid x\leq z\}$, which entails~\eqref{e-supr}.
  Dually, $\pmin{C}(z) = \inf\{x\in C\mid x\geq z\}$.

  Using~\eqref{e-supr}, and the dual property, we get
  \begin{align*}
\barqm{C}(x) = \pmin{C} \circ \pmax{C}(x)& =
 \inf \{y \in C \mid y \geq \pmax{C}(x)\}\\
    & =\inf\{y \in C\mid y\geq \sup \{ z\in C\mid z\leq x\} \} \\
    &= \inf \{y \in C\mid (z\leq x, z\in C) \implies z \leq y\}
    \\
    &= \supc \{z\in C\mid z\leq x\}= \qm{C}(x)
    \enspace.
  \end{align*}
  The proof that $\barqp{C}=\qp{C}$ is dual.
\end{proof}

For closed tropical cones $C$, the projection
$Q_C^- = P^{\max}_C$ has the property that $Q_C^-(x)$ is a point
of $C$ with minimal distance to $C$ with respect to Hilbert's
seminorm, see~\cite{cgq02}, hence, it is a tropical analogue
of the ``nearest point projection'' arising in Euclidean spaces.
A basic property of this projection on a closed
convex set of an Euclidean space is that it is sunny.
Recall that 
  a retraction $F $ from $\R^n$ to a subset of $\R^n$
  is \emph{sunny} if,
  given $x$, $y \in \R^n$,
 for every $z$ in the segment $[x,y]$,
$F(x) = y \implies F(z) = y$.
 The following result shows that the canonical projections on closed
 tropical or dual tropical cone are also sunny.

\begin{prop}\label{prop-sunny}
  If $C$ is a closed tropical cone, then, the projection
  $Q_C^- = P^{\max}_C$ is sunny.
  Similarly, if $C$ is a closed dual tropical cone,
  then, the projection $Q_C^+ = P^{\min}_C$ is sunny.
\end{prop}
\begin{proof}
   It is trivial that $Q_C^- = P^{\min}_C \circ P^{\max}_C= P^{\max}_C$.
  Let $y:=P^{\max}_C(x)$, so by \eqref{e-comparepmaxmin} $y\leq x$. 
  Then any point $z=(1-t)y+tx$ with $0<t<1$ satisfies $y\leq z\leq x$,
  and so, $y=P^{\max}_C(y)\leq P^{\max}_C(z)\leq P^{\max}(x)=y$,
  implying that $P^{\max}_C$ is a sunny retraction.
  \end{proof}

However, the following example shows that~\Cref{prop-sunny}
does not carry over to closed ambitropical cones, the canonical retractions $Q_C^-= \pmin{C}\circ \pmax{C}$
  and $  Q_C^+= \pmax{C}\circ \pmin{C}$ are generally not sunny. 
\begin{example}
  Consider the Shapley operator
  \[
  T\left(\begin{array}{c} x_1\\x_2\\x_3
\end{array}\right)=
  \left(\begin{array}{c} \max(x_1,x_3)\\ \max(x_2,x_3)
    \\ -\frac{1}{2}+\frac{x_1+x_2}2 \end{array}\right)
    \]
    The fixed point set of $T$ is the closed
    ambitropical cone $E = \{ x \in \R^{3} \mid \frac{x_1 + x_2}{2} = x_3 + \frac{1}{2}, x_1\geq x_3, x_2\geq x_3\}$ shown on ~\Cref{fig-cex-ma}. The cross section of $E$ in $\R^{2}$ defined by $x_3 = 0$ is displayed by the black segment,
    as well as the cross section of $E^{\max}$ (triangle in light gray above
    this segment) and the cross section of $E^{\min}$ (triangle in dark gray
    below this segment).  
 In this case $\pmin{E}\circ \pmax{E}$ is not a sunny retraction.
 Indeed, consider the point $x=(2,\alpha,0)$, with $0<\alpha<1$.
 We have $P^{\max}_E(x)=u:=(1,\alpha,0)$,
 and $P^{\min}_E(u)=y:=(1-\alpha/2,\alpha/2,0)$.
 However, consider the mid-point $x':=(y+x)/2 = (3/2 - \alpha/4, 3 \alpha / 4, 0)$.
 We have $u':=P^{\max}_E(x')= (1, 3 \alpha /4, 0)$ and $y':=P^{\min}_E(u') = ( 1 - 3 \alpha /8, 3 \alpha / 8, 0)
 \neq y$ showing that $\pmin{E}\circ \pmax{E}$
 is not sunny.  This is illustrated in the figure.

\end{example}
 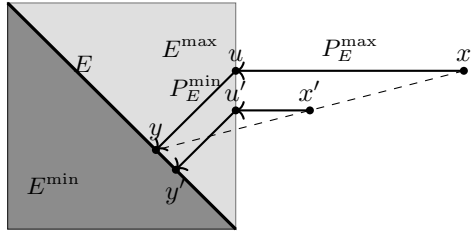
\begin{figure}[h!]
 \begin{center}
 \begin{tikzpicture}[scale = 3]
\coordinate %
(A) at (0,0);
\coordinate %
(B) at (1,0);
\coordinate %
(C) at (0,1);
\coordinate %
(P) at (1,1);
\filldraw[gray,draw=black,opacity=0.9] (A) -- (B) -- (C) -- cycle;
\filldraw[lightgray,draw=black,opacity=0.5] (B) -- (C) -- (P) -- cycle;
\draw[very thick] (C) -- (B);
\coordinate [label={[font=\small]
center:$E^{\max}$}] (Q) at (0.8,0.8);
\coordinate [label={[font=\small]
center:$E^{\min}$}] (K) at (0.2,0.2);
\coordinate [label={[font=\small]
center:$E$}] (E) at (0.33,0.73);
\draw coordinate [label=$x$] (x) at (2, 0.7);
\draw coordinate [label=$u$] (u) at (1, 0.7);
\fill (1, 0.7) circle (0.5pt);
\fill (2, 0.7) circle (0.5pt);
\draw coordinate [label=$y$] (y) at (13/20, 7/20);
\fill (13/20, 7/20) circle (0.5pt);
\draw [ ->, thick] (x) -- (1, 0.7)
node [pos=0.5, above, font=\small] {$\pmax{E}$};
\draw [ ->, thick] (1, 0.7) -- (y)
node [pos=0.5, above, font=\small] {$\pmin{E}$};
\draw[dashed] (x) -- (y);

\draw coordinate [label=$x'$] (x') at (53/40, 21/40);
\fill (53/40, 21/40) circle (0.5pt);
\draw coordinate [label=$u'$] (u') at (1, 21/40);
\fill (1, 21/40) circle (0.5pt);
\draw [ ->, thick] (x') -- (u');
\draw coordinate [label=below:$y'$] (y') at (59/80, 21/80);
\fill (59/80, 21/80) circle (0.5pt);
\draw [ ->,  thick] (u') -- (y');

\label{not-sunny}
\end{tikzpicture}
 \end{center}
 \caption{An ambitropical cone such that the canonical retractions $Q_E^\pm$ are not sunny.}\label{fig-cex-ma}
 \end{figure}

 \subsection{Characterization of ambitropical cones in terms of best co-approximation}
 We saw in~\Cref{prop-caracq} that the two canonical retractions
 $\qm{C}$ and $\qp{C}$ on a closed ambitropical cone can be decomposed
 in terms of the projection operators $\pmax{C}$ and $\pmin{C}$.
 In a perhaps surprising way, we shall see that this 
 leads to an order-theoretical analogue of a notion of best co-approximation,
 introduced by Papini and Singer~\cite{singer} to characterize nonexpansive retracts. If $E$ is a subset of a Banach space $(X,\|\cdot\|)$,
$E$ is said to be a \emph{set of existence of best co-approximation}
if, for all $z\in X$, the set
\[
B^{\|\cdot\|}_E(z):= \{x\in X\mid \| y-x\| \leq \|y-z\|,\;\forall y\in E\}
\]
contains an element of $E$. It is immediate that if $E$ is a nonexpansive
retract of $X$, then $E$ is a set of existence of best co-approximation.
The converse is known to hold in $L^p$ spaces with $1\leq p<\infty$,
see~\cite{westphal} and the references therein. Here we shall be interested in Shapley retracts of $\R^n$. 
By \Cref{prop-equiv}, these are precisely the images of $\R^n$ by idempotent
maps that are nonexpansive in the ``top''
hemi-metric $\mytop{x}=\max_{i\in [n]}x_i$. In view of this
property, we introduce the following analogue
of the set $B_E^{\|\cdot\|}(z)$.

\begin{defn}
Let $E$ be a subset of $\R^n$.
For any $y, z \in \R^n$ we define $B(y, z)= \{ x \in \R^n \mid y + \mybot(z-y) \leq x \leq y + \mytop (z-y) \}$ where  $\mytop(x) = \max_{i \in [n]} x_i$ and $\mybot (x)=\min_{i \in [n]}x_i$ for any $x \in \R^n$.
Then, we define $B_{E}(z)=\cap_{y \in E} B(y, z)$.
\end{defn}

\begin{defn}
Let $E$ be a subset of $\R^n$, we say that $E$ is a 	\emph{set of existence of best tropical co-approximation} if $B_{E}(z) \cap E \neq \emptyset$ for every $z \in \R^n$.
\end{defn}

\begin{lem}\label{lem-car-B}
We have $B_{E}(z)= \{ x \in \R^n \mid \pmax{E}(z) \leq x \leq \pmin{E}(z)\}$.
\end{lem}

\begin{proof}
We shall prove that $\sup_{y \in E} (\mybot(z-y) + y) = \pmax{E}(z)=\sup\{x \in E^{\max} \mid x \leq z\}$. Let $x \in E^{\max}$, $x=\sup_{y \in E} (\lambda + y)$. 
Suppose $x \leq z$, then we have that $\lambda  \leq \mybot (x-y) \leq \mybot (z-y)$.
In a similar way, we see that $\pmin{E}(z)=\inf_{y \in E} \mytop(z-y) + y$.
\end{proof}

\begin{lem}
Let $E$ be a subset of $\R^n$. If $E$ is a set of existence of best tropical co-approximation, then $\bar{B}_{E}(z):=[\barqm{E}(z), \barqp{E}(z)] \cap E \neq \emptyset$ for any $z \in \R^n$.
\end{lem}

\begin{proof}
  We know that, if $E$ is a set of best tropical co-approximation then for any $z \in \R^n$, there exists $u \in E$ such that $\pmax{E}(z) \leq u \leq \pmin{E}(z)$. Composing the first inequality by $\pmin{E}$, and composing the
  second inequality by $\pmax{E}$,
  we get $\barqm{E}(z)= \pmin{E}\circ \pmax{E}(z) \leq \pmin{E}(u)=u$,
  and $u=\pmax{E}(u) \leq \pmax{E}\circ \pmin{E}(z)=\barqp{E}(z)$.
\end{proof}

The following result completes \Cref{th-main0}, it characterizes
ambitropical cones in terms of best co-approximation and of the projections $\pmax{E}$ and $\pmin{E}$.
\begin{thm}\label{th-main1}
    Let $E$ be a subset of $\R^n$. The following assertions are equivalent
    \begin{enumerate}    %
          \item\label{shapleyretractnew}  $E$ is a closed ambitropical cone of $\R^n$;
\item\label{coapprox} $E$ is a set of existence of best tropical co-approximation;
\item\label{interval} for all $z\in \R^n$, $[\pmax{E}(z),\pmin{E}(z)]\cap E \neq\emptyset$;
\item\label{sara}
    $\pmin{E} (z) \in E$ holds for all $z \in E^{\max}$;
\item\label{saradual}
    $\pmax{E} (z) \in E$ holds for all $z \in E^{\min}$;

\item\label{fpqp} $E$ is the fixed point set of the operator $\barqp{E} = \pmax{E} \circ \pmin{E}$;

\item\label{fpqm}  $E$ is the fixed point set of the operator $\barqm{E}=\pmin{E} \circ \pmax{E}$.
  \end{enumerate}
\end{thm}
\begin{proof}

  \myitem{shapleyretractnew} $\Rightarrow$ \myitem{coapprox}.
  By~\Cref{th-main0}, we have that $E=P(\R^n)$ where $P=P^2$ is a Shapley operator. Then,
  for all $y\in E$, and for all $z\in \R^n$,
  $\mytop{(P(z)-y)}=\mytop{(P(z)-P(y))}\leq \mytop{(z-y)}$, i.e.,
  $P(z)\leq \mytop{(z-y)}+ y $, and dually,
  $P(z)\geq \mybot{(z-y)}+ y $, showing that
  $P(z) \in E \cap B_E(z)$.

  \myitem{coapprox} $\Rightarrow$ \myitem{interval}.
  This follows from~\Cref{lem-car-B}. 

  \myitem{interval} $\Rightarrow$ \myitem{sara}.
    We first observe that
  \begin{align}
    \forall z\in \R^n, \; [\pmax{E} (z),\pmin{E} (z)]\cap E \neq \emptyset \Rightarrow \forall z \in E^{\max}, [z, \pmin{E} (z)]\cap E \neq \emptyset \enspace .
    \label{iffselect}
  \end{align}
Now, %
the condition $[z, \pmin{E} (z)] \cap E \neq \emptyset$
is equivalent
to:
$\exists u \in E$ such that $z \leq u \leq \pmin{E} (z)$.
However, $\pmin{E}(z)$ is the minimal vector $v\in E^{\min}$ such that $v\geq z$,
it follows that $\pmin{E}(z)\leq u$, and so $\pmin{E}(z)=u\in E$.

\myitem{sara} $\Rightarrow$ \myitem{saradual}. 
By hypothesis, we have that for any $z \in E^{\max}$, $\pmin{E}(z) \in E$, so in particular $[z; \pmin{E}(z)] \cap E \neq \emptyset$. Consider an arbitrary $z \in \R^n$. Then, we have that $\pmax{E} (z) \in E^{\max}$ and, consequently 
$[\pmax{E} (z), \pmin{E}  (\pmax{E}(z))] \cap E \neq \emptyset$. Recalling that $[\pmax{E} (z), \pmin{E} (\pmax{E} (z))] \subseteq [\pmax{E} (z), \pmin{E} (z)]$ since $\pmax{E} \leq I$, we have that for any $z \in \R^n$, 
$[\pmax{E}(z); \pmin{E}(z)] \cap E \neq \emptyset$. In particular, let $z \in E^{\min}$, then we have that $[\pmax{E}(z), z] \cap E \neq \emptyset$ and by the dual argument of the previous implication we obtain that $\pmax{E}(z) \in E$.

  \myitem{saradual} $\Rightarrow$ \myitem{fpqp}.  
We will denote with $\operatorname{Fix}(\barqp{E})$ the fixed points set of $\barqp{E}$. Since any element of $E$ is fixed by $\barqp{E}$, $E \subseteq \operatorname{Fix}(\barqp{E})$. We shall now prove the other inclusion. Let $z \in \R^n$ such that $\pmax{E} (\pmin{E} (z)) = z$. Since $\pmin{E}(z) \in E^{\min}$, 
$z=\pmax{E}(\pmin{E}(z)) \in E$. 

  \myitem{fpqp} $\Rightarrow$ \myitem{fpqm}. By the idempotency $\barqp{E}$ we have that $\pmax{E} \circ \pmin{E}(y) \in E$ for every $y \in \R^n$, in particular for any $y \in E^{\min}$, $\pmax{E}(y) \in E$. As in the previous implication we know that $E \subseteq \operatorname{Fix}(\barqm{E})$ and we shall now prove the other inclusion. Let $z \in \R^n$ such that $\pmin{E} (\pmax{E}(z)) = z$, so $z \in E^{\min}$ and consequently $\pmax{E}(z) \in E$. Since $E$ is fixed by $\pmin{E}$, we have that $z= \pmin{E}(\pmax{E}(z)) \in E$. 

  \myitem{fpqm} $\Rightarrow$ \myitem{shapleyretractnew}. This follows
  from \Cref{prop-systematic2}, \myitem{idemqm} and~\Cref{th-main0}.

\end{proof}

\begin{cor}\label{cor-stable}
  Let $E$ be a non-empty subset of $\R^n$, included
  in a closed ambitropical cone $F$. Then,
  $E^{\max}\cap E^{\min} \subset F$.
\end{cor}
\begin{proof}
  The set $E^{\max}\cap E^{\min}$ is fixed
  both by $\pmax{E}$ and $\pmin{E}$. If $E\subset F$,
  the fixed point set of $\pmax{E}$ is included
  in the fixed point set of $\pmax{F}$. The same
  is true for $\pmin{E}$ and $\pmin{F}$. So, $E^{\max}\cap E^{\min}$
  is included in the fixed point set of $\pmin{F}\circ \pmax{F}$,
  which by \Cref{th-main1},\myitem{fpqm}, coincides with $F$.
\end{proof}
\begin{cor}
  Suppose $C$ is a closed ambitropical cone. Then $C=C^{\max}\cap C^{\min}$.
\end{cor}
\begin{proof}
  The inclusion $C^{\max}\cap C^{\min}\subset C$ follows from \Cref{cor-stable}.
The other inclusion is trivial.
\end{proof}

\begin{example}%
  Conversely, given an additive cone $E\subset\R^n$, the condition that $E = E^{\max}\cap E^{\min}$ does not imply that $E$ is an ambitropical cone. Consider, for example, the set $E=(a+\R) \cup (b+\R)$  where $a=(1,0,0)$ and $b=(0,1,0)$.
  This set, as well as the spaces $E^{\max}$ and $E^{\min}$, are shown
  on~\Cref{fig-hull}. We see that $E=E^{\max}\cap E^{\min}$
  but since $E$ is disconnected, it cannot be ambitropical.
\end{example}

\subsection{Ambitropical hull}
The intersection of ambitropical cones is generally not ambitropical,
so the notion of ambitropical hull of a set $E$ cannot be defined in the na\"\i ve manner, as the intersection of ambitropical cones containing $E$.
However, we shall see that there is a proper notion
of ambitropical hull, unique up to isomorphism.
\begin{definition}
  Let $E\subset \R^n$. We say that $\tilde{E}\subset \R^n$
  is a {\em ambitropical hull} of $E$ if $\tilde{E}$ is
  a closed ambitropical cone which is a superset of $E$ and if it is minimal with respect to inclusion. 
  \end{definition}
\begin{prop} \label{ambitropicalhull}
  For each nonempty subset $E\subset \R^n$, the sets
  $\range \barqm{E}$
  and $\range \barqp{E}$ are closed ambitropical cones
  containing $E$ that are isomorphic.
\end{prop}

\begin{proof}
  If $E\subset \R^n$ is non-empty, the operator $\barqm{E}$
  maps $\R^n$ to $\R^n$, and it follows from its definition that
  it satisfies the axioms of Shapley operators (\Cref{def-shapley}).
  Moreover, we have seen in \Cref{prop-systematic2}, {\em\ref{idemqm}}
  that $\barqm{E}$ is idempotent. It follows that $\range \barqm{E}=
  \barqm{E}(\R^n)$ is a Shapley retract, and so $\range \barqm{E}$ is ambitropical. Moreover, $\range \barqm{E}\supset E$.
  By duality, the same is true for $\range \barqp{E}$.
  By Prop. \ref{prop-systematic2}, \ref{regular}, the map $\barqp{E}$ is a bijection from $\range \barqm{E}$ to $\range \barqp{E}$ with inverse map $\barqm{E}$.
\end{proof}
The next result shows that if $F$ is a closed ambitropical cone
containing $E$, then, there is a Shapley operator
which injects $\range \qm{E}$ into $F$. %

\begin{thm}[Ambitropical hulls]\label{prop-ambihull}
  Let $E$ be a non-empty subset of $\R^n$, and suppose
  that $F$ is a closed ambitropical cone containing $E$.
Then,
  \begin{enumerate}[label=(\roman*)]
 \item\label{e-inj1} $\barqp{E}\circ \barqm{F}\circ \barqp{E}= \barqp{E}$;
 \item\label{e-inj2} $\barqm{E}\circ \barqp{F}\circ \barqm{E}= \barqm{E}$;
  \end{enumerate}
  Moreover, $\range \barqp{E}$ is isomorphic to a subset $F^{+}$ of $F$, analogously $\range \barqm{E}$ is isomorphic to a subset $F^{-}$ of $F$ and $F^{+}$ is isomorphic to $F^{-}$.  In particular, both $\range\barqp{E}$ and $\range{\barqm{E}}$ are ambitropical hulls of $E$, and all the ambitropical hulls
  of $E$ are isomorphic.%
\end{thm}
\begin{proof}
  Observe first that since $E\subset F$,
  \begin{align}
    \pmin{E} \circ \pmin{F}=\pmin{E}
    \label{e-cc}
    \end{align}
  Indeed, $E\subset F$ implies that $\pmin{F}\leq \pmin{E}$.
  Hence, $\pmin{E}\circ \pmin{F}\leq \pmin{E}\circ \pmin{E} = \pmin{E}$.
  Moreover, since $\pmin{F}\geq I$,
  $\pmin{E}\circ \pmin{F}\geq \pmin{E}$, which shows~\eqref{e-cc}.
  Since $E^{\max}\subset F^{\max}$, $\pmax{F}$ which fixes $E^{\max} = \range \pmax{E}$, we have
    \begin{align}
    \pmax{F} \circ \pmax{E}=\pmax{E}
    \label{e-cc2}
    \end{align}
    Using~\eqref{e-cc} and~\eqref{e-cc2},
    we get $\barqp{E}\circ \barqm{F}\circ \barqp{E}
   =\pmax{E}\circ \pmin{E} \circ \pmin{F}
  \circ \pmax{F}
  \circ \pmax{E}\circ \pmin{E} =
  \pmax{E}\circ \pmin{E}
  \circ \pmax{E}\circ \pmin{E} = (\barqp{E})^2=\barqp{E}$,
  by \Cref{prop-systematic2}, {\em\ref{item-ineq}}.
  This shows {\em\ref{e-inj1}}. The proof of {\em\ref{e-inj2}} is dual.

  Let $F^{+} = \barqm{F} (\range \barqp{E}) \subseteq F$ and consider
  \[
  \barqm{F} : \range \barqp{E} \to F^{+},
 \qquad 
  \barqp{E} : F^{+} \to \range \barqp{E} \enspace;
  \]
  they are inverses to each other by {\em \ref{e-inj1}}, so $\range \barqp{E}$ and $F^{+}$ are isomorphic. 
  Analogously, by {\em\ref{e-inj1}}, $\range \barqm{E}$ is isomorphic to the subset of $F$, $F^{-} =  \barqp{F} (\range \barqm{E})$.
  Moreover, by Corollary \ref{ambitropicalhull},
  $\range \barqp{E} \cong \range \barqm{E}$ and consequently
  $F^{-} \cong F^{+}$. Finally, let $\tilde{E}$ be an ambitropical hull of $E$, then $\range \barqp{E}$ is isomorphic to a subset $\tilde{E}^+$ of $\tilde{E}$ but, since $\tilde{E}$ is minimal for inclusion, we have that $\tilde{E} \cong \tilde{E}^+ \cong \range \barqp{E}$.
 \end{proof}
We next point out an elementary metric property of ambitropical cones.
Recall that a {\em geodesic} between $x$ and $y$ with respect to a
seminorm $\|\cdot\|$ on $\R^n$ is a map $\gamma : [0,1]\to \R^n$
such that $\gamma(0)=x$, $\gamma(1)=y$, and $\|\gamma(t_2)-\gamma(t_1)\|
+ \|\gamma(t_3)-\gamma(t_2)\|= \|\gamma(t_3)-\gamma(t_1)\|$,
for all $0\leq t_1<t_2<t_3\leq 1$.
\begin{prop}\label{prop-geodesic}
  If $E$ is a closed ambitropical cone of $\R^n$, then, for any
  two points $x,y$ of $E$, there is a curve connecting $x$ and $y$ and included in $E$ that is a geodesic both in Hilbert's seminorm and in the sup-norm.
\end{prop}
\begin{proof}
  Since $E$ is closed and ambitropical, there is an idempotent
  Shapley operator $P$ such that $E=P(\R^n)$. Let $x,y\in E$,
  and consider the ordinary line segment,
  $\gamma(s) = x + s(y-x)$, for $s\in [0,1]$, which
  connects $x$ and $y$ in $\R^n$, and is a geodesic in any semi-norm,
  in particular, in Hilbert's semi-norm and in the sup-norm.
 Then, since $P$ is nonexpansive
  both in Hilbert's seminorm and in the sup-norm, the map $s\mapsto P(\gamma(s))$ is still a geodesic, both in Hilbert's seminorm and in the sup-norm, and it is included in $E$. 
  \end{proof}

\section{Ambitropical convexity and hyperconvexity}
We relate here the above notions of ambitropical cones and of
ambitropical hull with hyperconvexity.
This notion was introduced by Aronszajn and Panitchpadki~\cite{aronszajn},
We refer the reader to  \cite{Isbell1964,Dress1984,baillon,Espinola2001} for insights on hyperconvex spaces and on the related notions of injective metric spaces and tight-span.

\begin{defn}\label{def-hyperconvex}
A metric space $M$ is said to be \emph{hyperconvex} if $\bigcap_{\alpha \in \Gamma} B(x_{\alpha}, r_{\alpha}) \neq \emptyset$ for any collection of points $\{x_{\alpha}\}_{\alpha \in \Gamma}$ in $M$ and positive numbers $\{r_{\alpha}\}_{\alpha \in \Gamma}$ such that $d(x_{\alpha}, x_{\beta}) \leq r_{\alpha} + r_{\beta}$ for any $\alpha$, $\beta \in \Gamma$. 
\end{defn}

We recall that a metric space $M$ is said to be \emph{injective} if for every $X, Y$ metric spaces, where $Y$ is a subspace of $X$ and $f \colon Y \to M$ nonexpansive, there exists an extension of $f$, $f' \colon X \to M$, that is nonexpansive. (Note that injective metric spaces do not coincide with injective objects in the category of metric spaces with nonexpansive maps. Indeed, it is true that every injective object is an injective metric space but the converse does not hold since the inverse of a nonexpansive map need not be nonexpansive.)
 Hyperconvex spaces are complete metric spaces and they are exactly injective metric spaces~\cite{Isbell1964,Dress1984}. %

 \Cref{th-main0} shows that closed ambitropical cones are %
Shapley retracts of $\R^n$.
Therefore, they are particular cases of  sup-norm nonexpansive retracts of $\R^n$, so we get by~\cite[Th.~9]{aronszajn}, that they are hyperconvex subsets of $\R^n$
with sup-norm metric, which are, in addition, additive cones.  We next show that the converse implication holds.

\begin{thm}\label{th-hyper}
  Let $E$ be an additive cone of $\R^n$. The following assertions are equivalent
  \begin{enumerate}    %
    \item\label{ambitropical-2} $E$ is a closed ambitropical cone;
\item\label{hyperconvex} $E$ is hyperconvex for the sup-norm metric.
  \end{enumerate}
\end{thm}

\begin{proof}
  We only need to show that
  $\eqref{hyperconvex} \implies \eqref{ambitropical-2}$.
  Suppose
  that $E$ is an additive cone which is hyperconvex in the sup-norm metric. 
  Since an hyperconvex space is complete, and $E$ is a subspace of $\R^n$,
  $E$ must be closed in the Euclidean topology. By \Cref{def-ambi}, it is enough to prove that $E$ is a lattice in the induced order. Let $x, y \in E$ and let, for any $s \leq r \in \R$,

\begin{equation}\label{rewrite-balls}
B^r_s(u)=\{x \in E \mid u+s \leq x \leq u +r\} = B(u + (s+r)/2, (r-s)/2)
\end{equation}

Let $U = \{u \in E \mid u \leq x, \quad u \leq y\}$. 
We want to prove that there exists a maximal element $z$ of $U$, that is an element $z \in E$ such that $z \leq x$, $z \leq y$ and $u \leq z$ for every $u \in U$. 

This is equivalent to $z \in B^0_{r_x}(x)$, $z \in B^0_{r_y}(y)$ for some $r_x$, $r_y < 0$ and, for every $u \in U$, $z \in B_0^{r_u}(u)$ for some $r_u > 0$.
Let us choose $r_x$ and $r_y$ such that $x + r_x/2 \leq y$, $y +  r_y/2 \leq x$,
and $r_u$ such that, for all $u\in U$, $x \leq u +r_u$ and $y \leq u +r_u$.
Then, we assume without loss of generality that $r_x\geq r_y$, and observe
that $x+r_x/2 \in B^0_{r_x}(x)\cap  B^0_{r_y}(y)$. Moreover,
$x\in B^0_{r_x}(x) \cap B_0^{r_u}(u)$ holds for all $u\in U$.
Hence, we have that the balls in the above family have pairwise nonempty intersections. Therefore, using~\eqref{rewrite-balls}, they can be rewritten as balls satisfying the conditions in \Cref{def-hyperconvex}. So, using the hyperconvexity of $E$, we obtain an element $z \in E$ belonging to the intersection of all the balls, and such an element is the maximal element of $U$. 
\end{proof}

\begin{defn}
Let $E$ be a metric space and let denote the metric by $d$. We call $E$ a \emph{metric space with a real action} if there is an action of the additive group $(\R,+,0)$ on $E$, denoted by $(\lambda, x)\in \R\times E \mapsto \lambda \cdot x\in E$, such that 
\begin{subequations}\label{real action}
\begin{align}
\label{real action_lambda} d(\lambda \cdot x, \lambda '\cdot x) = |\lambda-\lambda'|
\enspace , \\
\label{real action_x} 
d(\lambda \cdot x, \lambda \cdot x') = d(x, x')\enspace,
\end{align}
\end{subequations}
for any $x, x' \in E$ and any $\lambda,\lambda' \in \R$. 
\end{defn}

Let us consider the category of metric spaces with a real action, with the nonexpansive, action-preserving maps. In the same spirit as for metric spaces, we say that $C$ is an \emph{injective metric space with a real action} if for any metric space with a real action $Y$ and any subset $X$ of $Y$ which is closed by the real action on $Y$, every nonexpansive, action-preserving map from $X$ to $C$ has an extension from $Y$ to $C$
that is nonexpansive and action-preserving. Note that, as for metric spaces, this definition does not coincide with the usual definition of injective objects in the above category.%

A prominent example of hyperconvex space is given by the {\em tight span} of a metric space, which provides its hyperconvex or injective hull~\cite{Isbell1964,Dress1984}.
We now recall the definition and basic results regarding the tight span.

\begin{defn}[Tight span of a metric space~\cite{Isbell1964,Dress1984}]\label{def-tightspan}
  Let $X$ be any metric space with a metric $d$.
  The {\em tight span} $T(X)$ is the set of functions $f\colon X \to \R_{\geq 0}$ satisfying one of the two equivalent properties:
 \begin{enumerate}
\item $  f(x)=\sup_{y\in X} (d(x, y)-f(y))$, for all $x \in X$;
\item $f(x) +f(y) \geq d(x, y)$, for all $x, y \in X$, and $f$ is minimal among the functions with this property.
\end{enumerate}
\end{defn}
The second property in~\Cref{def-tightspan} was
used in the original construction
of Isbell, where a map $f$ satisfying this property is called an extremal function. The first property in~\Cref{def-tightspan} was used by Dress, who
noted the equivalence with the  second one and established a number of additional properties~\cite[Th.~3]{Dress1984}. The injectivity hull property of the tight span and some other properties were proved in~\cite[Section 2]{Isbell1964}.
We gather below some of the properties 
established in~\cite[Th.~3]{Dress1984} or in~\cite[Section2]{Isbell1964},
that will be used in the sequel.

We equip $T(X)$ with the supremum distance 
\[
d_{\infty}(f,g) = \sup_{x \in X} |f(x) - g(x)| \enspace.
\]
\begin{thm}[\protect{\cite[Th.~3]{Dress1984} and~\cite[Section 2]{Isbell1964}}]
\label{tightspanprop}
For any $x\in X$, we denote by $e(x)=d(x,\cdot)$ the map $y\in X\mapsto d(x,y)$.
We have the following properties
\begin{enumerate}
\item \label{tightspanprop1}
Any element $f$ of $T(X)$ is $1$-Lipschitz continuous,
that is $|f(x)-f(y)|\leq d(x,y)$.
\item \label{tightspanprop2}For all $x\in X$ and $f\in T(X)$, we have $d_{\infty}(e(x),f)=f(x)$.
\item \label{tightspanprop3}For all $x\in X$, $e(x)\in T(X)$, 
and the map  $e:X\to T(X),\; x\mapsto e(x)$ is an isometry.
Then, $X$ is isometric to the range $e(X)\subset T(X)$ of $e$, and 
can be identified to a subset of $T(X)$.
\item \label{tightspanprop4} Any nonexpansive map $\phi$ from $T(X)$ to itself that fixes $e(X)$
is the identity map.
\item \label{tightspanprop5}$T(X)$ is an hyperconvex, or equivalently, injective metric space.%
\item \label{tightspanprop6}$T(X)$ is the injective hull of $X$ in the category of metric spaces,  meaning that for any injective space $Y$ such that $X\subset Y$, or such that there is an isometry $\iota$ from $X$ to $Y$, there exists 
an isometry $\varphi$ from $T(X)$ to $Y$ such that $\iota=\varphi\circ e$.
\end{enumerate}
\end{thm}

Recall also that in \Cref{tightspanprop}, \Cref{tightspanprop6} follows from \Cref{tightspanprop5} and \Cref{tightspanprop4}.

We now show that when $X$ is a metric space with a real action,
the tight span $T(X)$ is canonically equipped with a structure of metric space with real action. 

\begin{prop}\label{tightaction}
  Assume that $X$ is a metric space with a real action
  $(\lambda, x)\in \R\times X \mapsto \lambda \cdot x\in X$.
For all $f \in T(X)$ and $\lambda\in \R$, we set
\[
\lambda \cdot f:X\to \R, \quad y\in X\mapsto
(\lambda \cdot f)(y) = f((-\lambda)\cdot y) \enspace . 
\]
Then, $\lambda \cdot f\in T(X)$. Moreover,
the map $(\lambda,f)\in \R\times T(X)\mapsto \lambda \cdot f\in T(X)$
yields a real action on $T(X)$, and, together with the supremum distance
$d_\infty$, this equips $T(X)$ with a structure of metric space with real action.
Moreover, the map $e:X\to T(X)$ of \eqref{tightspanprop3} of \Cref{tightspanprop} is an action-preserving isometry.
\end{prop}
\begin{proof}
 Using that \eqref{real action_x} holds for the metric and action of $X$, we shall deduce that $\lambda \cdot f\in  T(X)$ holds for any $f\in T(X)$ and $\lambda\in\R$. Indeed, for any $x \in X$,
we have
\begin{align*}
& \sup_{y\in X} (d(x, y)-(\lambda \cdot f)(y))
= \sup_{y\in X} (d(x, y)-f((-\lambda )\cdot y))
= \sup_{z\in X} (d(x, \lambda \cdot z)-f( z))\\
&\quad = \sup_{z\in X} (d((-\lambda )\cdot x, z)-f( z))
= f((-\lambda )\cdot x)= (\lambda \cdot f)(x)\enspace ,\end{align*}
where the first equality in the second line follows from \eqref{real action_x},
and the second one holds since $f\in T(X)$. 
So, by~\Cref{def-tightspan} (first property), $\lambda\cdot f \in T(X)$.
It follows that the map $(\lambda,f) \mapsto \lambda \cdot f$ defines
an action of the group $(\R,+,0)$ on $T(X)$.

We now show that the axioms~\eqref{real action} are satisfied, i.e., in the present setting:
  \begin{subequations}\label{action on T}
\begin{align}
\label{action on T1}&d_{\infty} (\lambda \cdot f, \lambda' \cdot f)=|\lambda - \lambda'|\enspace ,\\
\label{action on T2}&d_{\infty} (\lambda \cdot f, \lambda \cdot f') = d_{\infty}(f,f')\enspace,
\end{align}
\end{subequations}
for all $f,f'\in T(X)$ and $\lambda,\lambda'\in \R$.
Property~\eqref{action on T2} follows from a trivial change of variable:
 \begin{align*}
   &d_{\infty} (\lambda \cdot f, \lambda \cdot f')
   = \sup_{x \in X} |f((-\lambda) \cdot x) - f'((-\lambda)\cdot x)| 
 = \sup_{x \in X}|f(x) - f'(x)| = d_{\infty}(f,f')\enspace .
 \end{align*}
It remains to show \eqref{action on T1}. We have 
\begin{align}
  &d_{\infty} (\lambda \cdot f, \lambda' \cdot f)
  =\sup_{x \in X} |f((-\lambda) \cdot x) - f((-\lambda')\cdot x)|
\leq \sup_{x \in X} d((-\lambda) \cdot x, (-\lambda')\cdot x)=|\lambda - \lambda'|\enspace ,\label{ineq-action-T}
\end{align}
where the inequality in \eqref{ineq-action-T} uses that any element of $T(X)$ is $1$-Lipschitz continuous,
see \Cref{tightspanprop1} of \Cref{tightspanprop}.
If this inequality is strict for some $\lambda,\lambda'\in\R$ and $f\in T(X)$,
then 
\[ \nu:= \sup_{x \in X} |f((-\lambda) \cdot x) - f((-\lambda')\cdot x)|
< |\lambda - \lambda'|\enspace .\]
Denote $\mu=\lambda'-\lambda$, then applying a change of variable, 
we obtain that 
\[ |f(\mu \cdot x)-f(x)|\leq \nu <|\mu|,\quad \text{for all}\; x\in X\enspace.
\]
In particular, $f(\mu \cdot x)\leq f(x)+ \nu$, and applying successively 
this inequality to the elements $(n \mu)\cdot x$ such that $n$ is an integer, we
get $f((n \mu) \cdot x)\leq f(x)+ n\nu$, for all $n\geq 1$.
Since $f\in T(X)$, we also have $f((n \mu) \cdot x)+f(x)\geq d((n \mu) \cdot x,x)= n|\mu|$, where the last equation follows from \eqref{real action_lambda}.
We deduce that $n|\mu|\leq 2 f(x) +n \nu$ for all $n\geq 1$, which
is impossible since $\nu<|\mu|$. Therefore, the inequality in 
\eqref{ineq-action-T} is an equality, whih proves \eqref{action on T1}.

By \eqref{tightspanprop3} of \Cref{tightspanprop}, the map $e:X\to T(X), x\mapsto d(x,\cdot)$ is an  isometry.
We also have that
$e(\lambda\cdot x)(y)= d(\lambda\cdot x,y)= d(x,(-\lambda \cdot y))=
e(x)(-\lambda \cdot y)=(\lambda \cdot e(x))(y)$.
Hence, $e(\lambda\cdot x)= \lambda \cdot e(x)$.
This shows that $e$ is also action-preserving.
\end{proof}

\begin{thm}\label{hyperconvexhull}
Injective metric spaces with a real action are exactly hyperconvex spaces with a real action. 
\end{thm}

\begin{proof} To show that hyperconvex spaces with a real action are injective,
  we proceed as in the proof of Theorem 4.2 of \cite{Espinola2001},
  which shows the analogous property in the absence of a real action. In the following we will outline the parts which require to be adjusted. 
Let $H$ be an hyperconvex metric space with a real action, we want to prove that $H$ is injective. 
We shall consider a metric space with real action $(M,d)$ and $D$, a subset of $M$ which is closed by the action. Let $(H,d')$ be a metric space with real action and $g \colon D \to H$ a nonexpansive, action preserving map. To prove that $H$ is injective we shall show that there exists a nonexpansive and action-preserving extension of $g$ from $M$ to $H$. %

Let define $\mathcal{C}$ as follows:
\[
\mathcal{C} = \{(T_F, F) \mid D \subseteq F \subseteq M,\; F \text{ is closed by the action of } M,\; T_F \colon F \to H,\; T_F \text{ preserves the action} \}
\]
$\mathcal{C}$ is nonempty since it contains $(g,D)$. It is ordered by 
$(T_F, F)\leq (T_{F'},F')$ if $F\subset F'$ and $T_{F'}$ is an extension of $T_F$ over $F'$.
We need to show that any maximal element $(T_{F_1}, F_1)$ of $\mathcal{C}$ is
such that $F_1=M$.  We proceed by contradiction, assuming that there exists
$z\in M\setminus F_1$, and constructing $(T_F,F)\in {\mathcal C}$,
such that $(T_{F_1}, F_1)< (T_F,F)$ and $z\in F$.
Since the set $F$ must be close under the action, we take $F=F_1 \cup \{\lambda \cdot z \mid \lambda \in \R \}$. Also, the extended function $T_F$ need to preserve the action. So we need to choose $x\in H$  and define
$T_F(\lambda \cdot z) = \lambda \cdot x$ for all $\lambda \in \R$. Since we are assuming that $z \notin F_1$, we have that for every $\lambda$, $\lambda \cdot z \notin F_1$, so we do not have ambiguities in the definition of $T_F$. Since $H$ is hyperconvex, the arguments of the proof of Theorem 4.2 of \cite{Espinola2001} shows that there exists $x\in H$ such that the map $T_F$ restricted to the set $F_1\cup\{z\}$ is nonexpansive. We need to prove that $T_F$ is nonexpansive on all the metric space $F$.
Since $T_F$ is already nonexpansive on $F_1$, we only need to check the following conditions:
\begin{align*}
& d'(T_F(\lambda\cdot z),T_F(\lambda'\cdot z))\leq d(\lambda \cdot z,
\lambda' \cdot z)\enspace ,\\
&d'(T_F(z_1),T_F(\lambda\cdot z)\leq d(z_1,\lambda \cdot z)
\enspace ,
\end{align*}
for all $z_1\in F_1$ and $\lambda,\lambda' \in \R$.
Since $T_F(\lambda \cdot z) = \lambda \cdot x$, the first inequality is
an equality and follows from   \eqref{real action_lambda}.
Since $T_F$ is preserving the action on $F$, $F_1$ is closed under the action,
and $T_F$ is nonexpansive on $F_1\cup\{z\}$, 
the first inequality follows from  \eqref{real action_x}.

To show the converse implication,
we shall make use of the properties of the tight span. 

Consider now a metric space with  real action $X$ and assume that it is injective. Then, by \eqref{tightspanprop5} of \Cref{tightspanprop} 
 and \Cref{tightaction}, $T(X)$ is a hyperconvex set 
with a real action, and the map $e:X\to T(X), x\mapsto d(x,\cdot)$ is an action-preserving isometry.
By injectivity of $X$ there exists an action-preserving nonexpansive map from $h \colon T(X) \to X$
such that $h\circ e$ is the identity. Indeed, we can consider the identity map on $X$ as the map which will be expanded. 
Then, the map $e\circ h:T(X)\to T(X)$ is a nonexpansive map which fixes $e(X)$.

By \Cref{tightspanprop4} of \Cref{tightspanprop}, we have that the map $e\circ h$ is the identity, so $X=T(X)$ and it is hyperconvex. 
\end{proof}

We deduce the following result.

\begin{cor}\label{cor-injhull}
  Assume that $X$ is a metric space with a real action. 
Then, $T(X)$ is the injective hull of $X$, meaning that for any injective metric space with a real action $Y$ such that $X\subset Y$, or such that there is an action preserving isometry $\iota$ from $X$ to $Y$, there exists 
an action preserving isometry $\varphi$ from $T(X)$ to $Y$ such that $\iota=\varphi\circ e$.
\end{cor}
\begin{proof}
Let $X$ is a metric space with a real action, and let $Y$ be an injective metric space with a real action such that there is an action preserving isometry $\iota$ from $X$ to $Y$. 
As above, by \eqref{tightspanprop5} of \Cref{tightspanprop} 
 and \Cref{tightaction}, $T(X)$ is a hyperconvex set 
with a real action, and the map $e:X\to T(X), x\mapsto d(x,\cdot)$ is an action-preserving isometry. Moreover, by \Cref{hyperconvexhull}, $T(X)$ is injective.

By injectivity of $Y$ there exists an action-preserving nonexpansive map from $g \colon T(X) \to Y$ such that $g\circ e=\iota$.
By injectivity of $T(X)$ there exists an action-preserving nonexpansive map from $h \colon Y\to T(X)$ such that $h\circ \iota =e$.
Then, the map $h\circ g:T(X)\to T(X)$ is a nonexpansive map which 
satisfies $h\circ g\circ e= h\circ \iota= e$, so it fixes $e(X)$.

By \Cref{tightspanprop4} of \Cref{tightspanprop}, we have that the map $h\circ g$ is the identity, so $g$ is an isometry such that
$\iota=g\circ e$.
\end{proof}

\begin{thm}\label{Hypambihull}
Let $X \in \R^n$ be an additive cone. Then, the ambitropical hull of $X$ and the injective hull of $X$ are in bijective correspondence under an action-preserving isometry.
\end{thm}

\begin{proof}
  Let $A\subset \R^n$ be an ambitropical hull of $X$, with $\imath:X\to A$ the canonical injection. By~\Cref{th-hyper}, $A$ is a hyperconvex additive cone,
  and by~\Cref{hyperconvexhull}, it is an injective space with real action,
  and by~\Cref{cor-injhull}, there is an action-preserving isometry $\varphi$ from $T(X)$ to a subset of $A$ and such that $\varphi \circ e = \imath$. Hence,
  $\varphi(T(X))$ is injective, and by~\Cref{th-hyper} and~\Cref{hyperconvexhull}, it is an ambitropical cone. By minimality of $A$, $A=\varphi(T(X))$ which entails that $T(X)$ and $A$ are in bijective correspondence under the action
  preserving isometry $\varphi$. 
  \todo{added details, to be checked}
\end{proof}
\todo[inline]{SG: perhaps move what follows near the examples or delete it}
In general, it is already a difficult problem to construct the tight span of finite metric spaces with more than a few points. 
A characterization of the tight span of arbitrary subsets of the plane with the maximum metric is given in \cite{KILIC2016693}. 
Thanks to the above theorem, we get a way to compute the tight span of an additive cone, since this last one is isometric to its ambitropical hull, see e.g. \Cref{fig-extreme} for an illustration.
\todo[inline]{SV: added Kilic reference here, added part in the introduction about the results of this section} 
\section{Special classes of ambitropical cones}
We next review several canonical
classes of sets in tropical geometry, showing these
are special cases of ambitropical cones,
that can be characterized by suitable reinforcements
of \Cref{th-main0}.

The simplest examples of ambitropical cones consist of alcoved
polyhedra, discussed
in \Cref{sec-prelim}.  Actually, the following
result shows that alcoved polyhedra
are characterized by the property of being sublattices
of $\R^n$. Observe that all properties but one do not
assume polyhedrality, polyhedrality comes as a
consequence of the other properties.
\begin{prop}\label{prop-equiv2} Let $C\subset \R^n$. The following statements
  are equivalent:
  \begin{enumerate}
  \item\label{e-item1} $C$ is an alcoved polyhedron;
  \item\label{e-item0} $C$ is a closed tropical cone and a closed
    dual tropical cone,
    \item\label{e-item2} $C$ is a closed ambitropical cone in which the
      infimum and supremum laws coincide with the ones of $\R^n$.
    \item \label{e-item3} There is a tropically linear Shapley operator
      $T$ such that $C=\{x\in \R^n\mid T(x)\leq x\}$.
        \item \label{e-item4} There is a dually tropically linear Shapley operator $T$ such that $C=\{x\in \R^n\mid T(x)\geq x\}$.
      \end{enumerate}
\end{prop}
\begin{proof}[Proof of \Cref{prop-equiv2}]
  An alcoved polyhedron is stable by pointwise supremum and pointwise infimum
  of vectors, so~\eqref{e-item1} implies~\eqref{e-item0}. Trivially,
  ~\eqref{e-item0} implies~\eqref{e-item2}.

  Suppose now that~\eqref{e-item2} holds. Then, by~\Cref{prop-closed}, $C$
  is a conditionally complete lattice, and the lattice
  operations of $C$ are the pointwise supremum and pointwise infimum
  of vectors. Define, for all $i,j\in [n]$,
  \[
  M_{ij}:= \sup \{\lambda \mid v_i\geq \lambda + v_j ,\qquad \forall v\in C\}
  \]
  The latter set is nonempty, it is closed
  and bounded from above, so that the supremum is achieved.
  In particular, we have $M_{ij}\in\tmax$. By construction, $C\subset \alcoved{M}$, and $M_{ij}=M_{ij}^*$.
  Observe also that the inequality $v_i \geq \lambda +v_j$, 
  is equivalent to $w_i \geq \lambda$ where $w:=v- v_j\delta _j\in C$
  is such that $w_j=0$, denoting by $\delta_j=(0,\dots,0,1,0,\dots 0)$ the $j$-th vector of the canonical basis of $\R^n$.
  It follows that:
  \[
  M_{ij} = \inf C_{ij},\;\text{where}\;
  C_{ij}:= \{ v_i \mid v\in C_j\}
\;\text{and}\;
  C_j:= \{ v \in C\mid  v_j =0 \} \enspace .
  \]
  Denoting by $u^j$ the $j$th column of the matrix $M$, we deduce that
  \[
  u^j =\inf C_j \in \lowerclosure{C}\enspace .
  \]
  Define, $A_j:= \{v\in \lowerclosure{C}\mid v_j =0\}$.
  Since $C$ is a conditionally complete lattice, the set
  $A_j\subset\lowerclosure{C}$ is stable by taking infima.
  Hence, the set $\operatorname{Min}A_j$  consists
  of a single point, $u^j$. By \Cref{th-sum},
  every element of $C$ is a tropical linear
  combination of vectors $u^j$.
  This
  implies that $\alcoved{M}=\{M^* y \mid y\in \R^n\}= C$.
  So~\eqref{e-item2} implies~\eqref{e-item1}.

  If $C=\alcoved{M}$ is an alcoved polyhedron,
  it follows from~\eqref{e-alcoved2} that 
  $C=\{x\in \R^n\mid T(x) \leq x \}$ where $T(x) =M^*x$
  is a Shapley operator. So~\eqref{e-item1}
  implies~\eqref{e-item3}.

  Conversely, if $C=\{x\in \R^n\mid T(x)\leq x \}$ for some 
  tropically linear Shapley operator, then, for all $x,y\in C$,
  since $T$ is order preserving,
  $T(x\wedge y)\leq T(x)\wedge T(y) \leq x\wedge y$,
  and since $T$ is tropically linear,
  $T(x\vee y)=T(x)\vee T(y) \leq x\vee y$,
  showing that $x\wedge y$ and $x\vee y$ belong to $C$.
  Moreover, $C$ is closed, since $T$ is continuous (in fact, $T$ is sup-norm
  nonexpansive). 
  This shows that~\eqref{e-item3} implies~\eqref{e-item2}.
  
  \eqref{e-item1} and \eqref{e-item4} are equivalent.
  Indeed, observe that $C$ is of the form $\{x\in \R^n\mid T(x)\leq x\}$
  for some tropically linear Shapley operator iff $-C$
  is of the form $\{x\in \R^n\mid P(x)\geq x\}$
  for some dually tropically linear Shapley operator
  (consider the involution $T\mapsto P, \; P(x):=-T(-x)$
  on the space of Shapley operators). Moreover, $C$ is
  an alcoved polyhedron iff $-C$ is an alcoved polyhedron.
  Hence, the announced equivalence follows from
  the equivalence of~\eqref{e-item1} and~\eqref{e-item3},
  already established.
\end{proof}

Note that in the above proof we made use of a particular instantiation of the "flip" application $X \to -X$ which sends ambitropical cones to ambitropical cones. 

Tropical cones, and dual tropical cones,
are also remarkable examples of ambitropical cones. The relation
between these cones and sub or super-fixed point sets
of Shapley operators was already noted in~\cite{AGGut10}.
\begin{prop}\label{prop-tropical}
   Let $C\subset \R^n$. The following statements
  are equivalent:
  \begin{enumerate}
  \item\label{it-cone} $C$ is a closed tropical cone;
  \item\label{it-ambi} $C$ is a closed ambitropical cone in which
    the supremum law coincides with the one of $\R^n$;
  \item\label{it-cone2} there is a Shapley operator $T$
    such that $C=\{x\in \R^n\mid T(x) \geq x\}$.
\end{enumerate}
\end{prop}
\begin{proof}
  \eqref{it-cone}$\Rightarrow$\eqref{it-ambi}.
  Suppose that $C$ is a closed tropical cone,
  and let $X$ denote a non-empty subset of $C$ bounded
  from above by an element of $\R^n$. Then,
  for all finite subsets $F\in\finitesets{X}$,
  $\sup F$ belongs to $C$, because $C$ is stable
  by supremum, and $\sup X = \lim_{F\in \finitesets{X}} \sup F\in C$ because
  $C$ is closed. It follows that $X$ has a supremum in $C$ which
  coincides with its supremum in $\R^n$.
  Suppose now that $X$ is bounded from below by an element $z$ of $\R^n$.
  Consider $Y:= \{y\in C\mid y \leq x,\forall x\in X\}$. Then, $Y$ is non-empty
  and it is bounded from above. It follows from the previous
  observation that $\sup Y$ is the supremum of $Y$, in $C$.
  Moreover, $\sup Y$ is precisely the infimum of $X$ in $C$,
  showing that $C$ is ambitropical.

  \eqref{it-ambi}$\Rightarrow$\eqref{it-cone2}.
  Since $C$ is a closed ambitropical cone,
  then, it is the fixed point set of $z\mapsto \qm{C}(z)=\sup\{x\in C\mid x\leq z\}$, and $\qm{C}\leq I$. So, $C=\{z\in \R^n\mid \qm{C}(z) =z\}
  = \{z\in \R^n\mid \qm{C}(z)\geq z\}$.

  \eqref{it-cone2}$\Rightarrow$\eqref{it-cone}.
  Suppose that $C=\{x\in \R^n\mid T(x)\geq x\}$.
  Since $T$ is continuous, $C$ is closed. Moreover, since $T$
  is order preserving,
  for all $x,y\in C$, $T(x\vee y) \geq T(x)\vee T(y) \geq x\vee y$, showing that $x \vee y \in C$. 
\end{proof}
We state the following dual version of~\Cref{prop-tropical}.
\begin{prop}
   Let $C\subset \R^n$. The following statements
  are equivalent:
  \begin{enumerate}
  \item $C$ is a closed dual tropical cone;
  \item $C$ is an ambitropical cone in which
    the infimum law coincides with the one of $\R^n$;
      \item\label{it-cone2-dual} there is a Shapley operator $T$
    such that $C=\{x\in \R^n\mid T(x) \leq x\}$.\hfill\qed
\end{enumerate}
\end{prop}

We next define the subclass of {\em homogeneous} tropical cones
-- which will arise as tangent spaces or recession sets of
ordinary tropical cones.
\begin{defn}\label{def-homogeneous}
  An ambitropical cone $C$ is {\em homogeneous}
  if for all $\alpha>0$ and for all $x\in C$,
  $\alpha x\in C$.
  Recall that a Shapley operator $T:\R^n\to\R^n$ is {\em homogeneous}
  if $T(\alpha x) = \alpha T(x)$ holds
  for all $\alpha>0$ and for all $x\in \R^n$.
\end{defn}

\begin{prop}\label{prop-homogenous}
   Let $C\subset \R^n$. The following statements
  are equivalent:
  \begin{enumerate}
  \item\label{it-homogeneous} $C$ is a closed homogeneous ambitropical cone;
  \item\label{it-homretract} there is an idempotent homogeneous Shapley
    operator whose fixed point set is $C$;
  \item\label{it-homfp} there is a homogeneous Shapley operator
    whose fixed point set is $C$.
\end{enumerate}
\end{prop}
\begin{proof}
  The implication \myitem{it-homretract}$\Rightarrow$\myitem{it-homfp}
  is immediate. If $C$ is the fixed point set of an
  homogeneous Shapley operator $T$, then, we know from \Cref{th-main0}
  that $C$ is an ambitropical cone, and it follows from the homogeneity
  of $T$ that $C$ is homogeneous. This shows the implication \myitem{it-homfp}$\Rightarrow$\myitem{it-homogeneous}. If $C$ is a closed homogeneous
  ambitropical cone, we know from \Cref{th-main0} that $C$
  is the range of the idempotent Shapley operator $\barqm{C}=\qm{C}$.
  Observe that, for all $\alpha>0$ and $x\in \R^n$,
  $\qm{C}(\alpha x)=\supc \{y\in C\mid y\leq \alpha x\}
  = \supc\{\alpha \alpha^{-1} y\mid y\in C, \;\alpha^{-1}y \leq x\}=
  \supc\{\alpha z\mid z\in C, \; z \leq x\}=
  \alpha \qm{C}(x)$
  since $y\mapsto \alpha^{-1}y$ is a bijection from $C$ to $C$,
  which is order preserving and whose inverse also preserves the order.
  This shows the implication \myitem{it-homogeneous}$\Rightarrow$\myitem{it-homretract}.
\end{proof}

\section{Correspondence between fixed points of Shapley operators and calibrated policies}
\todo[inline]{section to be revised}
\subsection{Finitely generated Shapley operators}
The following
definition is taken from~\cite{gg}.
\begin{defn}
  A {\em min-max function} in the variables $x_1,\dots,x_n$
  is a map $f:\R^n\to \R$, defined by a term
  in the context-free grammar $X\to x_1,\dots,x_n, X\vee X, X\wedge X, X+c$
  where $c$ stands for any real constant.
\end{defn}
For instance, $f(x_1,x_2,x_3) = ((x_1\vee (x_2+3))\wedge ((x_3-1) \vee x_1))\vee x_2$ is a min-max function. It follows readily from the definition
that the set of min-max functions is stable by the operations
$\vee$, $\wedge$, and by the translation by a constant. Using
the distributivity law, we may always rewrite
a min-max function in conjunctive normal form,
i.e.,
\begin{align}
  \label{e-cnf}
  f(x) = \wedge_{k\in [K]} \vee_{i\in[n]} (c_{ki}+x_i)
\end{align}
for some integer $K$ and coefficients $c_{ki}\in \tmax$,
such that for all $k\in [K]$, $c_{ki}\neq -\infty$ for some $i\in[n]$.
Similarly, we may rewrite $f$ in disjunctive normal form:
\begin{align}
  \label{e-dnf}
  f(x) = \vee_{k\in [K']} \wedge_{i\in[n]} (c'_{ki}+x_i)
\end{align}
for some integer $K'$ and coefficients $c'_{ki}\in \tmin$,
such that for all $k\in [K']$, $c'_{ki}\neq +\infty$ for some $i\in[n]$.
{\em Monotone Boolean functions} are special cases of min-max functions,
they are obtained by restricting the above grammar rule to exclude
the derivation $X\mapsto X+c$. For instance, $(x_1\vee x_2)\wedge x_3$
is a monotone Boolean function.
\begin{defn}
  A Shapley operator is {\em finitely generated}
  if its coordinates are min-max functions.
\end{defn}
It will be convenient to write
finitely generated Shapley operators
in an algebraic way, along the lines of~\cite{AGGut10},
making apparent the game interpretation.
Let $A\in \tmaxn{m\times p}$. The {\em adjoint}
of the tropically linear map associated with the matrix $A$ is the dual tropically linear map 
\[
y\mapsto A^\sharp y,\qquad  (A^\sharp y )_k=\wedge_{i\in [m]} (-A_{ik}+y_i) ,
\qquad k\in [p]\enspace ,
\]
which sends $\R^m\to \R^p$ if $A$ has no identically $-\infty$ column.
Observe that, for $x\in \R^p$ and $y\in \R^m$,
\[
Ax \leq y \iff x\leq A^\sharp y \enspace .
\]
If $B\in \tmaxn{m\times n}$
The tropically linear map
\[
x\mapsto Bx, \qquad (Bx)_i = \vee_{j\in [n]} (B_{ij}+x_j), \qquad i\in [m]
\]
sends $\R^n$ to $\R^m$ if $B$ has not identically $-\infty$ row.
This motivates the following definition.
\begin{defn}
  We say that a pair of matrices $A\in \tmaxn{m\times p}, B\in \tmaxn{m\times n}$
  is {\em proper} if $A$ has no identically $-\infty$ column,
  and $B$ has no identically $-\infty$ row.
\end{defn}
Given a proper pair of matrices $A\in \tmaxn{m\times p}, B\in \tmaxn{m\times n}$,
we consider the operator $T=A^\sharp \circ B$, with coordinates
\begin{align}
T_i(x) = \wedge_{j\in [m]} \Big(-A_{ji} + \vee_{k\in [n]} (B_{jk} + x_k) \Big),\qquad i\in [p] \enspace .\label{e-def-fgshapley}
\end{align}
This is a finitely generated Shapley operator from $\R^n\to \R^p$.

\begin{prop}\label{prop-explicitform}
  Every finitely generated Shapley operator can be
  written as $T=A^\sharp\circ B$ for some proper pair of
  matrices $A,B$.
  \end{prop}
\begin{proof}
  Every coordinate of $T$ can be represented by a min-max function
  in disjunctive normal form, and this representation
  is of the form~\eqref{e-def-fgshapley}
  (with $A_{ij}\in \{0,-\infty\}$).
\end{proof}
Recall that a map $T:\R^n\to\R^p$ is (positively) {\em homogeneous} (of degree $1$) if $T(\alpha x)=\alpha T(x)$ holds for all $\alpha>0$
and $x\in \R^n$.
\begin{prop}\label{prop-monboolshapley}
  If $T$ is a finitely generated homogeneous Shapley operator,
  then the coordinates of $T$ are given by monotone Boolean functions,
  and all the finitely generated homogeneous Shapley operators
  arise is this way.
\end{prop}
\begin{proof} Suppose that $T=A^\sharp\circ B$ as above. Then,
$T_i(x) = \alpha^{-1}T_i(\alpha x)$, and so, 
  \begin{align*}
  T_i(x) &= \lim_{\alpha\to\infty}
  \wedge_{j\in [m]} \Big(-\alpha^{-1} A_{ji} + \vee_{k\in [n]} (\alpha^{-1} B_{jk} + x_k) \Big)\\
  & =   \wedge_{j\in [m],\; A_{ij}\neq -\infty} \big(\vee_{k\in [n],\;B_{jk}\neq -\infty}  x_k \big)
  \end{align*}
  which is a monotone Boolean function.
  \end{proof}
\begin{prop}
  Let $T_{1}$, $T_{2}$ be two finitely generated Shapley operators $\R^n\to \R^p$. Then the operators $T_{1} \vee T_{2}$ and $T_{1} \wedge T_{2}$ are finitely generated. Suppose now that $T_2$ is a finitely generated
  Shapley operator from $\R^q\to \R^n$. Then,
  $T_{1} \circ T_{2}$ is a finitely generated Shapley operator.
\end{prop}
\begin{proof}
  By definition, min-max functions are stable by the laws $\vee $ and $\wedge$.
  They are also stable by composition, meaning that if $f(x_1,\dots,x_n)$
  is a min-max function in the variables $x_1,\dots,x_n$, and
  if for all $i\in [n]$, $g_i(y_1,\dots,y_q)$ is a min-max function
  in the variables $y_1,\dots,y_q$, then $f(g_1(y_1,\dots,y_q),\dots,
  g_n(y_1,\dots,y_q))$ is a min-max function in the same variables.
  This implies the announced properties.  
  \end{proof}
\subsection{Eigenvectors of Shapley operators and Deterministic Mean Payoff Games} \label{Eigenvectors}
We next recall that
the operators of the form~\eqref{e-def-fgshapley} are precisely the dynamic
programming operators of {\em deterministic} zero-sum
games without discount, referring the reader to~\cite{AGGut10}
for background. We shall also present an equivalence between
the fixed points of the Shapley operator and a remarkable class
of optimal positional policies, {\em calibrated} policies.

We associate to the proper pair of matrices $(A,B)$ a game
in which two players, Max and Min, move alternatively a token in a digraph.
The set of states is the disjoint union of the sets $[n]$ and $[m]$.
If the token is in a state $i\in [n]$, Player Min chooses a new state $j\in [m]$
such that $A_{ji}$ is finite,  moves the token to this state, and receives a
payment $A_{ji}$ from Max. If the token is in a state $j\in[m]$, Player Max chooses a new state $k\in[n]$ such that $B_{jk}$ is finite, moves the token to this state, and receives a payment $B_{jk}$ from Min. The assumption that the pair
$(A,B)$ is proper guarantees that each player has always at least one available
action. Given an initial state $i\in[n]$, and an integer $k$, 
one can consider the {\em game in horizon $k$}, in which each the two
players makes $k$ moves, alternatively, so that the total
payment received by Player Max is
\[
R^k = -A_{j_1i_0}+B_{j_1i_1} - A_{j_2i_1}+ B_{j_2i_2} + \dots -A_{j_ki_{k-1}}+B_{j_ki_k}
\]
where $i_0,j_1,i_1,j_2,\dots,j_k,i_k$ is the sequence of states that are visited
and $i_0=i$ is the initial state.
A {\em history} of the game, at a given stage, consists of the sequence
of visited states, so if $l$ turns have been played, and if it Min's turn
to play, the history is $i_0,j_1,i_1,j_2,\dots,j_l,i_l$ whereas
if it is Max's turn to play, the history is $i_0,j_1,i_1,j_2,\dots,j_l,i_l,j_{l+1}$. We assume that the game is in perfect information, meaning that the
two players observe the history. 
A (pure) {\em strategy} of a player is a map which associates an action
to each history. Therefore, the total payment
is a function of the strategies $\sigma$ and $\pi$ of the
two players and of the initial state $i$, i.e., $R^k= R^k_i(\sigma,\pi)$. It follows
from dynamic programming arguments, 
that the game in horizon $k$ with initial state $i$ has a value,
$v_i^k$, and that there are associated optimal strategies $\sigma^*$
and $\pi^*$, meaning that the following saddle point
property holds:
\[
R_i^k(\sigma^*,\pi) \leq v_i^k
\leq R_i^k(\sigma,\pi^*)
\]
for all strategies $\sigma$ of Player Min and $\pi$ of Player Max,
which entails in particular that $v_i^k=R_i^k(\sigma^*,\pi^*)$.
Indeed, the {\em value vector} $v^k=(v^k_i)_{i\in [n]}\in \R^n$ satisfies
the dynamic programming equation
\[
v^0=0,\qquad v^k=T(v^{k-1}) \enspace ,
\]
and the actions that achieve the min or max in the
expression $T_i(v^{k-1})$ provide optimal decisions
for the two players, depending only on the current state
and on the time remaining to play.
We refer the reader to~\cite{sorinetal} for background on game theory
and on the dynamic programming approach.

A {\em deterministic policy} of Player Min (resp.\ Max) is a map
$\sigma: [n] \to  [m]$ (resp.\ $\pi: [m]\to [n]$).
A deterministic policy of Player Min induces
a stationary positional strategy, obtained by moving
to state $\sigma(i)$ when in state $i\in[n]$,
and similarly for Player Max, the state being now $j\in[m]$. 

We are interested in the {\em mean payoff game}, in which Player Min
wants to minimize the mean payment per time unit made to Player
Max, and Player Max wants to maximize it. Liggett and Lippman~\cite{liggettlippman}, and Ehrenfeucht and Mycielski~\cite{ehrenfeucht}
showed that there exists a vector $\chi=(\chi_i)_{i\in[n]}\in \R^n$ and
deterministic policies $\sigma^*$ and $\pi^*$, of Player Min and Max,
respectively,
such that, for all strategies $\sigma$ and $\pi$ of these
two players,  and for all initial states $i\in[n]$,
    \begin{equation}
    \limsup_{k\to\infty}k^{-1} R_{i}^k(\sigma^*,\pi)
    \leq \chi_{i}
    \leq     \liminf_{k\to\infty}k^{-1} R_{i}^k(\sigma,\pi^*)
    \enspace .
    \label{e-saddle}\end{equation}
    The number $\chi_i$ is known as the {\em value} of the mean payoff game
    with initial state $i$. Moreover, the vector $\chi$ coincides with the limit
    $\lim_{k\to \infty} v^k/k$. 

    A remarkable case arises when the non-linear eigenproblem
    \[
    T(u) = \lambda  + u, \qquad u\in \R^n, \lambda \in \R
    \]
    is solvable. Then, the value of the mean payoff game $\chi_i$
    is independent of the initial state $i$, and optimal
    deterministic policies can be obtained from
    the eigenvector $u$. To explain this relation,
    it will be convenient to extend the notion of policy
    as follows: a {\em nondeterministic policy} of Player Min (resp.\ Max) is a map
$\sigma: [n] \to \powerset{[m]} \setminus \emptyset$,
such that $\sigma(i)\subset \{j\in [m]\mid A_{ji} \,\text{finite}\}$
for all $i\in[n]$. Similarly, 
 a {\em nondeterministic policy} of Player Max is a map
 $\pi: [m] \to \powerset{[n]} \setminus \emptyset$,
such that $\pi(j)\subset \{i\in [n]\mid B_{ji}\, \text{finite}\}$
for all $j\in[m]$. A nondeterministic policy of Player Min
{\em induces} a whole collection of strategies (not necessarily
stationary or positional), obtained by restricting
the moves of Min to the $i\to j\in \sigma(i)$.
In particular, it induces {\em deterministic policies},
obtained by selecting, for all $i\in[n]$, a single
element in $\sigma(i)$.
We denote by $\sigmaset$ the set of nondeterministic
policies of Player Min, and by $\piset$ the set of nondeterministic
policies of Player Max. %
The following definition extends to the two-player case
the notion of {\em calibrated trajectory} introduced by Fathi
in the context of weak-KAM theory, see~\cite{Fathi2004,Fat}.

\begin{definition}[calibrated policies]
  Given $u\in \R^n$, we say that a pair of non-deterministic
  policies $(\sigma^u,\pi^u)$
  is {\em $u$-calibrated} if, for some $\lambda\in \R$,
      \begin{enumerate}[label=\roman*)]
    \item By playing any strategy induced by $\sigma^u$, Player Min can guarantee
      that, whatever strategy Max chooses, and for all horizons $k$
      and initial states $i_0$, 
      \begin{align}
    -A_{j_0i_0}+B_{j_0i_1}+ \dots  -A_{j_{j-1}i_{k-1}}+B_{j_{k-1}i_{k}} \leq u_{i_0}-u_{i_k} + k\lambda  \enspace, 
    \label{e-primal}\end{align}
        \item By playing any strategy induced by $\pi^u$, Player Max can guarantee
          that, whatever strategy Min chooses,  and for all horizons $k$ and initial states $i_0$,
      \begin{align}   -A_{j_0i_0}+B_{j_0i_1}+ \dots  -A_{j_{j-1}i_{k-1}}+B_{j_{k-1}i_{k}}
\geq  u_{i_0}-u_{i_k} + k\lambda  \enspace; \label{e-dual}
\end{align}
 \end{enumerate}
where $i_0,j_0,i_1,j_1\dots,i_k$ is the sequence of states that are visited.
\end{definition}
In the special one-player case, assuming for instance that Player Max is a dummy
(with only one possible action in each state), we can replace the
inequality by an equality in~\eqref{e-primal},
and then, we recover the original notion of calibrated trajectory.
In particular, Fathi established a correspondence between
the global viscosity solutions of the
ergodic Hamilton-Jacobi PDE and calibrated trajectories, see~\cite[Prop.~3.5]{Fathi2004}
and~\cite[Prop~4.1.10]{Fat}. The following elementary proposition
states an analogous property in the two-player
setting. %

\begin{prop}\label{prop-calibrated}
  Suppose that $T(u) = \lambda  + u$, define
  $\pi^*(j)= \argmax_{i} B_{ji}+u_i$ and $\sigma^*(i)=\argmin_{j}-A_{ji}+(Bu)_j$.
  Then, the pair of policies $(\sigma^*,\pi^*)$ is $u$-calibrated.
  Moreover, all pairs of $u$-calibrated policies arise in this way.
\end{prop}
To establish this result, it will be convenient to introduce,
for all pair of deterministic policies $(\sigma,\pi)$
of the two players, the operators
$T^\sigma$ and $^{\pi}T$, $\R^n\to\R^n$, such that
\[
T^\sigma_i(x) = -A_{\sigma(i)i}  + \vee_{k}(B_{\sigma(i)k} +x_k),
\qquad
^{\pi}T_i(x) = \wedge_{j} (-A_{ji}  + B_{j\pi(j)} +x_{\pi(j)}) \enspace .
\]
These operators represent the one-player games obtained
by fixing a policy of one of the players.\todo{SG: perhaps could be defined before}
\begin{proof}
  If $T(u) = \lambda  + u$, by definition
  of $\sigma^*$ and $\pi^*$, we have, for all deterministic
  policies $\sigma$ and $\pi$ compatible with $\sigma^*$ and $\pi^*$,
  \[
  T^\sigma(u) = \lambda + u,\qquad
    ^{\pi}T(u) = \lambda + u \enspace.
  \]
  It follows that for all $k$, $(T^\sigma)^k(u) = k\lambda  +u$,
  which implies that~\eqref{e-primal} holds. Similarly,
  $(^{\pi}T)^k(u) = k\lambda  +u$ yields~\eqref{e-dual}.

  Conversely, if $\sigma^*$ and $\pi^*$ are $u$-calibrated, by specializing~\eqref{e-primal}
  to $k=1$, we deduce that $T(u)\leq \lambda +u$, and similarly,
    we deduce from~\eqref{e-dual} that $T(u)\geq \lambda +u$.
\end{proof}

Any deterministic policies $(\sigma^*,\pi^*)$ induced 
by a pair $(\sigma^u,\pi^u)$ of $u$-calibrated policies
are optimal in the mean payoff game, meaning they satisfy the saddle
point property~\eqref{e-saddle}.  However, being $u$-calibrated is a finer
property than being optimal in the mean payoff game, since it involves
not only the mean payoff but also the deviation to the mean payoff.
This is easily
seen in the one-player case. Then, playing a deterministic policy induces
state trajectories which ultimately cycle. As long as the the ultimate cycle
reached from each initial state is unchanged, the mean-payoff optimality
is preserved, but not the property of being calibrated.
\begin{example}
Consider the one player game shown in \Cref{fig-fathi},
  in which the Player is Max. This can be represented
  by the pair of matrices $(A,B)$ where $A$ is the tropical
  identity map, meaning that player Min is a dummy, and $B$
  is the matrix of payments shown in the figure. The value
  of the mean payoff game is equal to $1$, since it
  is optimal for Player Max to reach the cycle
  $4\to 4$ which has the best weight-to-length ratio,
  equal to $1$.
  The matrix $B$, and so, the operator $T=A^\sharp \circ B$, has only one
  eigenvector up to an additive constant,  given by
  $u = (-2, -2, -1, 0)$, and so, by Proposition \ref{prop-calibrated}, there is only one $u$-calibrated strategy, namely $1 \rightarrow 3$, $3 \rightarrow 4$ and $2 \rightarrow 4$. However, the strategy $1 \rightarrow 2$, $3 \rightarrow 4$ and $2 \rightarrow 4$ is also optimal, since the same ultimate cycle is reached by any initial state,
but it is not $u$-calibrated. 
 \end{example}

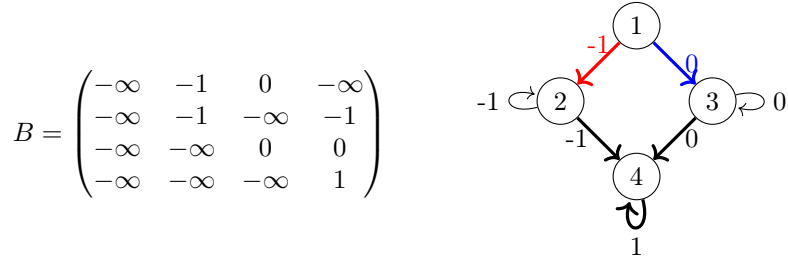
\begin{figure}[htbp]
  \begin{tabular}{cc}
\begin{minipage}[c]{0.4\textwidth}
\[
B =
\begin{pmatrix}
   -\infty & -1 & 0  & - \infty \\
    - \infty & -1 & - \infty  & -1 \\
    - \infty & - \infty & 0 & 0 \\
    - \infty & - \infty & - \infty & 1
\end{pmatrix}
\]
\end{minipage}
& 
\begin{minipage}[c]{0.4\textwidth}
 \begin{tikzpicture}[->]
    \node[draw,circle] (1) at (1,1) {1};
    \node[draw,circle] (2) at (0,0) {2};
    \node[draw,circle] (3) at (2,0) {3};
    \node[draw,circle] (4) at (1,-1) {4};

    \path[color=red, very thick] (1) edge              node[above] {-1} (2);
    \path[color=blue, very thick] (1) edge              node[above,right] {0} (3);
    \path[color=black, very thick] (2) edge              node[above,left] {-1} (4);
    \path[color=black, very thick] (3) edge              node[above, right] {0} (4);
    \path (2) edge [loop left] node {-1} (2);
    \path (3) edge [loop right] node {0} (3);
    \path[very thick] (4) edge [loop below] node {1} (4);
  \end{tikzpicture}
\end{minipage}
\end{tabular}
  \caption{A one player game (maximizing the average cost), with two optimal policies, only one of which
    (going right, in blue) is calibrated.}
\label{fig-fathi}
\end{figure}

\begin{example}
  An example of calibrated policies arises from the notion of Blackwell
  optimality, originally introduced in the one player setting~\cite{blackwell}. Let $0<\alpha<1$ be a discount factor, and consider the value vector $v^\alpha$ of the
  discounted version of the game, so that $v^\alpha$ is the unique solution of the fixed
  point equation $v^\alpha =T(\alpha v^\alpha)$.  It follows from the result
  of Kohlberg~\cite{Koh80} that, as soon as the mean payoff of the game is independent of the initial
  state, $v^\alpha$ has a Laurent series expansion satisfying in particular $v^\alpha= \lambda /(1-\alpha)  + u +O(1-\alpha)$, where $\lambda\in \R$ is the mean payoff, as $\alpha \to 1^-$, and we have $T(u) = \lambda +u$. There are policies $\sigma^*,\pi^*$ such that $T(\alpha v^\alpha)= T^{\sigma^*}(\alpha v^\alpha)= {}^{\pi^*} T(\alpha v^\alpha)$ for all values of $\alpha$ close enough to $1$, meaning that $\sigma^*,\pi^*$
  are optimal in all the discounted games with a discount factor sufficiently close to one (see~\cite[Th.~8]{entropygames} for a proof of this property in the two player setting). These policies are said to be Blackwell optimal. They are $u$-calibrated.
\end{example}

\section{Ambitropical polyhedra}
We now consider ambitropical cones with a polyhedral structure. 
\begin{defn}
  An {\em ambitropical polyhedron} is an ambitropical cone that is a finite
  union of alcoved polyhedra. An {\em ambitropical polytope} is an ambitropical polyhedron that is bounded in Hilbert's seminorm.
\end{defn}
Observe that
an ambitropical polyhedron is closed.
We shall first show that fixed point sets of finitely generated
Shapley operators are ambitropical polyhedra,
and then, we will show that all ambitropical polyhedra
arise in this way and provide a notion of generating family.
\subsection{Polyhedral structure of the fixed point sets of finitely generated Shapley operators}

Recall that a \emph{polyhedral complex} $\mathcal{K}$ is a set of polyhedra, called \emph{cells}, that satisfies the following conditions:
  every face of a polyhedron from $\mathcal{K}$ is also in $\mathcal{K}$;
  the intersection of any two polyhedra $\sigma_{1}, \sigma_{2} \in \mathcal {K}$ is a face of both $\sigma_{1}$ and $\sigma_{2}$.
  A polyhedral complex is a \emph{fan} if every cell is a cone.
  The {\em support} of a polyhedral complex is the union of its cells.
(Here, cone is understood in the sense of convex analysis, not
in the sense of ordered additive cones.)

Suppose $E$ is the fixed point set of a finitely generated Shapley
operator. So $E=\{x \mid  x = A^{\sharp}\circ Bx\}$, for some proper pair of matrices $A,B\in \tmaxn{m\times n}$.
We will show that $E$ is the support of a polyhedral
complex, and that the cells of this complex correspond
to {\em calibrated policies} of the two players.

\todo[inline]{to be cleaned}

It will be convenient to lift the ambitropical cone, considering
\[
F:= \{(x,y)\in \R^n\times \R^m\mid x=A^\sharp y,\; y= Bx \}
\]
so that $E=\operatorname{proj}_x(F)$, where $\operatorname{proj}_x$
is the projection $(x,y)\mapsto x$ from $\R^n\times \R^m$ to $\R^n$.

Given any pair of nondeterministic
policies $(\sigma,\pi)\in\sigmaset\times\piset$
(see~\S\ref{Eigenvectors}), we denote by $Z_{\sigma,\pi}$ the set 
of couples $(x,y)\in \R^n\times \R^m$
verifying the following relations
\begin{subequations}
\begin{align}
  x &\leq A^{\sharp}y,\qquad \qquad   x_{i} \geq -A_{ji} + y_{j},\quad \forall i\in[n], \forall j \in \sigma(i) \label{e-ab}\\
   y & \geq Bx \qquad \qquad 
y_{j} \leq B_{jk} + x_{k},\quad \forall j\in[m], \forall k \in \pi(j)
\enspace .\label{e-ab2}
\end{align}\label{e-def-Z}
\end{subequations}

Let us consider $\sigma^{-1}(j)=\{ i \in [n] \mid  j \in \sigma(i)\}$ and  $\pi^{-1}(i)=\{ j \in [m] \mid  i \in \pi(j)\}$. Observe
that $\cup_{j \in [m]} \sigma^{-1}(j)=[n]$ and $\cup_{i \in [n]} \pi^{-1}(i)=[m]$,
because $\sigma(i)$ and $\pi(j)$ are non-empty, for all $i\in [n]$
and $j\in [m]$.

\begin{prop}
  We have 
 \begin{align}
F = \bigcup_{(\sigma,\pi) \in \sigmaset\times \piset} Z_{\sigma,\pi} \enspace .
\label{e-cover}
\end{align}
\end{prop}
\begin{proof}
  The relations~\eqref{e-ab} entail that $x=A^\sharp y$,
  and similarly~\eqref{e-ab2} entail that $y=B x$.
  So, for all $(\sigma,\pi)\in\sigmaset\times\piset$,
  $Z_{\sigma,\pi}\subset F$.
  Moreover, for all $(x,y)\in F$, taking $\sigma(i):=\argmin_{j\in [m]} (-A_{ji}+y_j)$, we can check that~\eqref{e-ab} is satisfied by $x  =A^\sharp y$.
Similarly, for all $x$ in $\R^n$, taking $\pi(j):=\argmax_{i\in [n]} (B_{ji}+x_i)$, we can check that~\eqref{e-ab2} is satisfied by $y=Bx$. So, $(x,y)\in Z_{\sigma,\pi}$. 
\end{proof}
\begin{prop}\label{prop-projZ}
  The image $X_{\sigma,\pi}:= \operatorname{proj}_x(Z_{\sigma,\pi})$
is characterized by the relations
\begin{subequations}\begin{align}
[(A \vee B)x]_{j} &\leq B_{jk} + x_{k}, \qquad  \forall j \in[m],\forall k \in \pi(j) \label{e-revised-i}\\
[(A \vee B)x]_{j} &\leq A_{ji} + x_{i},\qquad  \forall j \in[m],\forall i\in \sigma^{-1}(j) \enspace .\label{e-revised-ii}
  \end{align}\label{e-revised}
  \end{subequations}
  Moreover, both $Z_{\sigma\pi}$ and $X_{\sigma,\pi}$ are alcoved polyhedra.
\end{prop}
\begin{proof}
  Let $(x,y)\in Z_{\sigma,\pi}$.
  Since $x \leq A^{\#}y$ is equivalent to $Ax \leq y$,
  we deduce that  $(A\vee B) x\leq y $.
Moreover,~\eqref{e-def-Z} entail that
\begin{align}
  A_{ji}+ x_{i} & \geq y_{j},\quad \forall j\in[m], \forall i \in \sigma^{-1}(i) \enspace .
\end{align}
It follows that $x$ satisfies~\eqref{e-revised}. Conversely, suppose
that~\eqref{e-revised} holds, and let $y:=Bx$. The inequalities
in~\eqref{e-revised-i} entail that $[Bx]_j\leq [(A\vee B)x]_j
\leq B_{jk}+x_k \leq [Bx]_j$
for all $j\in [m]$ and $k\in\pi(j)$, so that $B_{jk}+x_k=y_j=[(A\vee B)x]_j$.
Then,
it follows from these inequalities that $Ax\leq y$, and so $x\leq A^\sharp y$.
Then, we deduce from~\eqref{e-revised-ii} that
$y_j \leq A_{ji}+ x_i$ for all $j\in[n]$ and $i\in \sigma^{-1}(j)$,
which can be rewritten as $-A_{ji}+y_j \leq x_i$ for
all $i\in [n] $ and $j\in \sigma(i)$, and so $A^\sharp y\leq x$, which
shows that $(x,y)\in Z_{\sigma,\pi}$.

It is immediate from the form of the constraints in~\eqref{e-def-Z} and~\eqref{e-revised} that $Z_{\sigma,\pi}$ and $X_{\sigma,\pi}$ are alcoved
polyhedra.
\end{proof}

We define the {\em type} of a point $x\in \R^n$
to be the pair of partially defined maps $\tau:=(\sigma,\pi)$
where for all $j\in [m]$,
$\pi(j)$ denotes the set of
$k\in[n]$ such that the
relation~\eqref{e-revised-i} holds,
and $\sigma^{-1}(j)$ denotes the set
of $i\in [n]$ such that~\eqref{e-revised-ii} holds.
We say that a type is {\em proper}
if both $\sigma$ and $\pi$ are policies
(this requires the maps $\sigma$ and $\pi$
to be totally defined).
We denote by $\mathcal{T}$ the set of proper types associated
to points $x\in\R^n$.

\begin{thm}\label{th-polycomplex}
  Suppose $E$ is the fixed point set of the finitely generated
  Shapley operator. Then, the collection of alcoved polyhedra
  $(X_{\tau})_{\tau \in \tauset}$
  constitutes a polyhedral complex
  whose support is $E$. Moreover, the cell $X_\tau$ consists
  precisely of those fixed points $u$ such that the pair $\tau$
  of nondeterministic policies in the mean payoff game associated
  to $T$ is $u$-calibrated. %
\end{thm}
\begin{proof}
  It follows from~\eqref{e-cover}, $E=\operatorname{proj}_x(F)$
  and~\Cref{prop-projZ} that
  $E=\bigcup_{\tau} X_\tau$. We have to show that
  the collection of polyhedra $\{X_\tau\}_{\tau \in \tauset}$ is a polyhedral complex.

  Consider $\tau=(\sigma,\pi)\in\mathcal{T}$ and
  $\tau'=(\sigma',\pi')\in \mathcal{T}$. Let
  $\pi''\in\piset$ be such that $\pi''(j)=\pi(j)\cup\pi'(j)$
  for all $j\in[m]$, and let $\sigma''\in\sigmaset$ be such
  that $\sigma''(i)=\sigma(i)\cup\sigma'(i)$ for all
  $i\in [n]$. It is immediate from~\eqref{e-revised} that
  $X_\tau\cap X_{\tau'}= X_{\tau''}$. 
  Let now $x$ be a point in the relative
  interior of $X_{\tau''}$. For $j\in [m]$, let $\pi'''(j)$ be defined
  as the set of $k$ such that $[(A\vee B)x)]_j\leq B_{jk}+x_k$.
  For $i\in [n]$, let $\sigma'''(i)$ be defined as the set
  of $j$ such that $[(A\vee B)x]_j\leq A_{ji}+x_i$,
  so that $\tau'''$ is the type of $x$.
  Observe that $\sigma'''(i)\supset \sigma''(i)\supset \sigma'(i)\neq \emptyset$, which entails that $\sigma'''$ is proper. Similarly, $\pi'''$ is proper.
  Moreover, the inclusion $X_\tau \cap X_{\tau'} \supset X_{\tau'''}$ is trivial.
  The reverse inclusion follows from the observation that
  $\operatorname{relint}(X_\tau\cap X_{\tau'}) \subset X_{\tau'''}$.
  This follows from the fact that $\tau'''$ is the type of any
  point $x\in   \operatorname{relint}(X_\tau\cap X_{\tau'}) $.

  Finally, observing that a face $F'$ of $X_\tau$ is obtained by saturating
  some of the inequalities~\eqref{e-revised}, and taking for $\tau'$
  the type of an arbitrary point in the relative interior of this face,
  it is immediate that $F'=X_{\tau'}$.

  Moreover, considering the proof of~\Cref{prop-calibrated}, we see
  that a pair of nondeterministic policies $\tau=(\sigma,\pi)\in \mathcal{T}$
  is $u$-calibrated if and only if, for all
   deterministic policies $\sigma_{\operatorname{d}}$, $\pi_{\operatorname{d}}$
  induced by $\sigma$ and $\pi$, we have
  $u=T(u)=T^{\sigma_{\operatorname{d}}}(u)={}^{\pi_{\operatorname{d}}}T(u)$,
  and this means precisely that $u\in X_\tau$.\todo{SG: to be checked}
\end{proof}
\begin{remark}
  The polyhedral complex of~\Cref{th-polycomplex} generalizes
  the complex introduced by Develin and Sturmfels to represent
  tropical polyhedra~\cite{TropConv}. The latter
  complex is recovered by considering the special case in which $A=B$,
  so that $T=B^\sharp \circ B$. Then, the range of $T$ is precisely the
  dual tropical cone generated by the opposite of the columns
  of $B$. By~\Cref{prop-projZ}, the cell $X^{\sigma,\pi}$ is given by $\{x\mid (Bx)_{j} \leq B_{jk}+x_k,\;k\in \pi(j), \; (Bx)_j \leq B_{ji}+x_i, \; i\in \sigma^{-1}(j)\}$, and then, we see that this cell coincides
  with $X^{\bar{\pi}^{-1},\bar{\pi}}$, in which $\bar{\pi}$ is the nondeterministic
  policy whose graphs is the union of those of $\pi$ and $\sigma^{-1}$.
  The cells $X^{\bar{\pi}^{-1},\bar{\pi}}$ are precisely the ones
    that constitute the polyhedral complex of~\cite{TropConv}
    and the policy $\bar{\pi}$ is equivalent to the {\em combinatorial type}
    defined there.
  \todo[inline]{SG: to be checked. Add that if the cell is of non-empty interior, then the type consists of deterministic policies. Link with the develin-sturmfels complex when $T=V V^\sharp$. }
\end{remark}
\todo[inline]{SG: add discussion of~\cite{sturmfels2012combinatorial,tran2014polytropes}}

\subsection{Polyhedral complexes associated with ambitropical polyhedra}
\todo[inline]{SG: to be adapted to the move}
Whereas ambitropical cones arise as fixed point sets
of Shapley operators, we shall see that ambitropical polyhedra
arise as fixed point sets of finitely generated Shapley operators.

In the special case of alcoved polyhedra,
the following lemma shows that these Shapley operators are simple, and
its proof shows that they are associated with one-player games.
\begin{lem}\label{say-fg}
Let $E\subset \R^n$ be an alcoved polyhedron, $\qm{E}(x)=\supe \{y \in E; y \leq x\}$ and $\qp{E}(x)=\infe \{y \in E; y \geq x\}$. Then $\qm{E}$ and $\qp{E}$ are finitely generated Shapley operators. 
\end{lem}

\begin{proof}
  Observe first that $\supe$ coincides with the $\sup$
  law of $\R^n$ and that similarly $\infe$ coincides with the
  $\inf$ law of $\R^n$, because $E$ is an alcoved polyhedron
  (stable by these sup and inf laws). We can find a matrix $M\in \tmaxn{n\times n}$ such that
  $E= \alcoved{M}$, i.e., $E=\{x\in \R^n\mid x\geq Mx\}$. We
  claim that $\qp{E} (x) = M^* x$. Indeed, by~\Cref{prop-def-alcoved}, $M^*x\in \alcoved{M}$.
  Moreover, since $M^*\geq I$, $M^*x \geq x$, and so $\qp{E}(x) \leq M^*x$.
  Now, if $z\geq x$ for some $z\in\alcoved{M}$,
  we have $z=M^*z\geq M^*x$, showing that $M^*x\leq \qp{E}(x)$.
  The operator $x\mapsto M^* x$ is finitely generated.
  A dual argument shows that $\qm{E}$
  which
  is also finitely generated.
\end{proof}
We now compute a finitely generated Shapley retraction
on an ambitropical polyhedron represented as a union of alcoved polyhedra.
\begin{lem}
  \label{lem-explicit0}
  Suppose that $E$ is the union of a finite family of alcoved polyhedra
  $(E_k)_{k\in K}$. Then,
    \[
  \pmax{E} = \sup_{l\in K} \qm{E_l},\qquad 
  \pmin{E} = \inf_{l\in K} \qp{E_l} \enspace .
  \]
\end{lem}
\begin{proof}
  By definition of $\qm{E_l}$, for all $x\in \R^n$,
  $E_l\ni \qm{E_l}(x)\leq x$.
  Since $E_l\subset E^{\max}$ and $E^{\max}$ is stable by supremum,
  we deduce that $E^{\max}\ni \sup_{l\in K} \qm{E_l}(x)\leq x$,
  and so, $\pmax{E}(x) \geq \sup_{l\in K} \qm{E_l}(x)$.
  Moreover, any element $z$ of $E^{\max}$ can be written as
  $z=\sup_{l\in K}z^l$ for some $z^l\in E_l$. If $z\leq x$,
  it follows that $z^l\leq \qm{E_l}(x)$, from which we deduce
  that $z\leq \sup_{l\in K}\qm{E_l}(x)$. Since this holds
  for all $E^{\max}\ni z\leq x$, it follows that $P^{\max}(x) \leq
  \sup_{l\in K}\qm{E_l}(x)$. The proof of the characterization
  of $\pmin{E}$ is dual.
\end{proof}
\begin{cor}\label{lem-explicit}
  Let $E$ be the union of a finite family of alcoved
  polyhedra $(E_k)_{k\in K}$. Then
  \begin{align}
    \qm{E}(x) & =
    \inf_{k \in K} \qp{E_k}(\sup_{l\in K} \qm{E_l}(x)),
   \qquad 
    \qp{E}(x) =
    \sup_{k \in K} \qm{E_k}(\inf_{l\in K} \qp{E_l}(x)) \enspace .
    \label{e-c1}
  \end{align}
\end{cor}
\begin{proof}
  This follows from~\Cref{prop-caracq} and~\Cref{lem-explicit0}.
\end{proof}

\begin{thm}\label{th-ambitropicalpoly}
Let $E$ be a subset of $\mathbb{R}^{n}$, then the following are equivalent:
\begin{enumerate}
\item\label{ie-2} $E$ is an ambitropical polyhedron
\item\label{ie-3} There exists a finitely generated Shapley operator $P$ such that $P=P^{2}$ and $E=\{ x \in \mathbb{R}^{n} | x = P(x) \}$.
\item\label{ie-1} $E$ is the fixed point set of a finitely generated
  Shapley operator;
\item\label{ie-4} $E$ is a closed ambitropical cone and $\bar{E}^{\max}$
  and $\bar{E}^{\min}$ are finitely generated as modules over $\tmax$ and $\tmin$, respectively.
\end{enumerate}
\end{thm}

\begin{proof}

\eqref{ie-2}$\Rightarrow$\eqref{ie-3}.
The set $E$ is the range of the idempotent
Shapley operator $\barqm{E}$, and 
it follows from \Cref{lem-explicit}
and \Cref{say-fg} that this Shapley operator is finitely generated.

$\eqref{ie-3} \Rightarrow \eqref{ie-1}$ is obvious.

$\eqref{ie-1} \Rightarrow \eqref{ie-4}$. Since $E$ is the fixed
point set of a Shapley operator, $E$ is a closed ambitropical cone.
If this Shapley operator is finitely generated, then by~\Cref{th-polycomplex},
$E$ can be written as the union of a finite family of alcoved polyhedra $(X_\tau)_{\tau \in\tauset}$.
Then, $\bar{E}^{\max}=\lowerclosure{E^{\max}}$ coincides with the union of the lower closures 
$\lowerclosure{X_\tau}$, for $\tau \in \tauset$. By~\Cref{prop-def-alcoved},
every $\lowerclosure{X_\tau}$ is a finitely generated $\tmax$-semimodule.
It follows that $\bar{E}^{\max}$ is a finitely generated $\tmax$-semimodule.

$\eqref{ie-4} \Rightarrow \eqref{ie-2}$.
By~\Cref{lem-explicit}, $E$ is the fixed point set of the operator
$\barqm{E}$ given in~\eqref{e-c1}. By~\Cref{say-fg},
every operator $Q_{E_k}^\pm$ is finitely generated.
Hence, $\barqm{E}$ is finitely generated. Then, by~\Cref{th-polycomplex},
$E$ is a finite union of alcoved polyhedra.
\end{proof}
The description of an ambitropical polyhedron as the fixed
point set of a Shapley operator is analogous
to the ``external'' description of a polyhedron. We next
show that ambitropical polyhedra admits an alternative
description, by generators.
\begin{definition}
  A {\em description by generators} of a closed ambitropical
  cone $E$ consists of a pair $(U^{\max},U^{\min})$
  such that $U^{\max}$ is a tropical generating set
  of $\bar{E}^{\max}$ and $U^{\min}$ is a dual tropical generating set
  of $\bar{E}^{\min}$.  We say that the description
  is {\em finite} if the sets $U^{\max}$ and $U^{\min}$ are finite.
\end{definition}
The generating sets $U^{\max}$ and $U^{\min}$ uniquely determine
$E^{\max}$ and $E^{\min}$, and so they uniquely determine the ambitropical cone
$E$, which coincides with $\range{P^{\max}_E\circ P^{\min}_E}$.
\begin{cor}\label{cor-fingen}
A closed ambitropical cone
is an ambitropical polyhedron if and only if it has a finite
description by generators. Moreover, it is an ambitropical
polytope if and only if these generators belong to $\R^n$.
\end{cor}
\begin{proof}
  If $E$ is an ambitropical polyhedron, it follows from~\Cref{th-ambitropicalpoly}, \eqref{ie-4} that $E$ admits a finite description by generators.

  Conversely, if $E$ is a closed ambitropical cone that
  admits a finite description by generators,
  $(U^{\max},U^{\min})$, by~\eqref{explicit-pmax} and its dual,
  $P^{\max}_E$ and $P^{\min}_E$ are finitely generated
  Shapley operators, and so, $E$ is the fixed
  point set of the finitely generated Shapley operator
  $P^{\max}_E\circ P^{\min}_E$. So, $E$ is an ambitropical polyhedron by Theorem \ref{th-ambitropicalpoly} (\ref{ie-3}).

  Now, if $E$ is bounded in Hilbert's seminorm, we have
$\bar{E}^{\max}=E^{\max}\cup \{(-\infty,\dots,-\infty)\}$,
  and dually, $\bar{E}^{\min}=E^{\min}\cup \{(+\infty,\dots,+\infty)\}$,
  we get $U^{\max}\subset E^{\max}$ and $E^{\min}\subset U^{\min}$,
  showing that the generators belong to $\R^n$. Conversely,
  suppose that the elements of $U^{\max}$ and $U^{\min}$ belong
  to $\R^n$, and let $R$ denote the maximal Hilbert seminorm
  of these elements. We observe that balls in Hilbert's seminorm
  are invariant by the operations of suprema and infima, and
  so, considering the formula~\eqref{explicit-pmax},
  we deduce that $E\subset \range{P_E^{\max}}$ is included
  in the ball of radius $R$ in Hilbert's seminorm.\todo{SG: to be checked - SV:checked}
  \end{proof}
When $E$ is an ambitropical polytope, for any description
by generators $(U^{\max}, U^{\min})$, $U^{\max}$ and $U^{\min}$
are necessarily subsets of $E$, and so we actually get
a proper notion of {\em internal} representation of $E$ by generators.

In particular, we have
the following characterization of ambitropical polytopes.
\todo[inline]{SG: I revised this part of the ms and added the following cor, check}
\begin{cor}\label{cor-ambihull}
  Every finite subset of $\R^n$ admits an ambitropical hull
  that is an ambitropical polytope, and all the ambitropical polytopes arise
  in this way.  %
\end{cor}
\begin{proof}
  By~\Cref{prop-ambihull}, an ambitropical hull of a finite subset $E=\{u^1,\dots,u^k\}$ of $\R^n$
  is given by the range of the finitely generator
  operator $\barqp{E}= \pmax{E} \circ \pmin{E}$,
  which, by~\Cref{cor-fingen}, is bounded in Hilbert's seminorm,
  and so, it is an ambitropical polytope.

  Conversely, suppose that $F$ is an ambitropical polytope. Then, $F$
  is a finite union of alcoved polyhedra $F_k$ that are bounded in Hilbert
  seminorm. Each of these alcoved polyhedra $F_k$ has a finite set
  of tropical generators. By taking the union of these finite sets
  we obtain a (possibly redundant) finite set $F^+$
  of tropical generators of $E^{\max}$. The dual constructions yields
  a finite set $F^-$ of dual tropical generators
  of $E^{\min}$.  Observe that the
  explicit construction  of the retraction $\barqp{E}$ given in~\Cref{lem-explicit} involves elementary operators $\qpm{E_k}$ which only depend on the
  primal and dual tropical generators of $E_k$.  Hence, by taking
  for $E$ the union of the two sets $F^\pm$, we get that $\barqp{E}=\barqp{F}$,
  showing that $F$ is an ambitropical hull of the finite set $E$.
\end{proof}
\begin{example}
  \Cref{cor-ambihull} is illustrated in \Cref{fig-extreme}, showing
  an ambitropical hull of the points $a_1,\dots,a_9$ given by the columns
  of the matrix
  \[
  \bordermatrix{
   ~& a_1&a_2&a_3& a_4&a_5&a_6&a_7&a_8&a_9\cr
    ~&4 & 5 & 3  & 1 &0  & 0&  0&  0&4\cr
    ~&0 & 2 & 4  & 3 &4  & 2&  2& -1&0\cr
    ~& 0 & 0 & 0 & 0  &2  & 4&  2&  0&3
}
  \]
  \end{example}
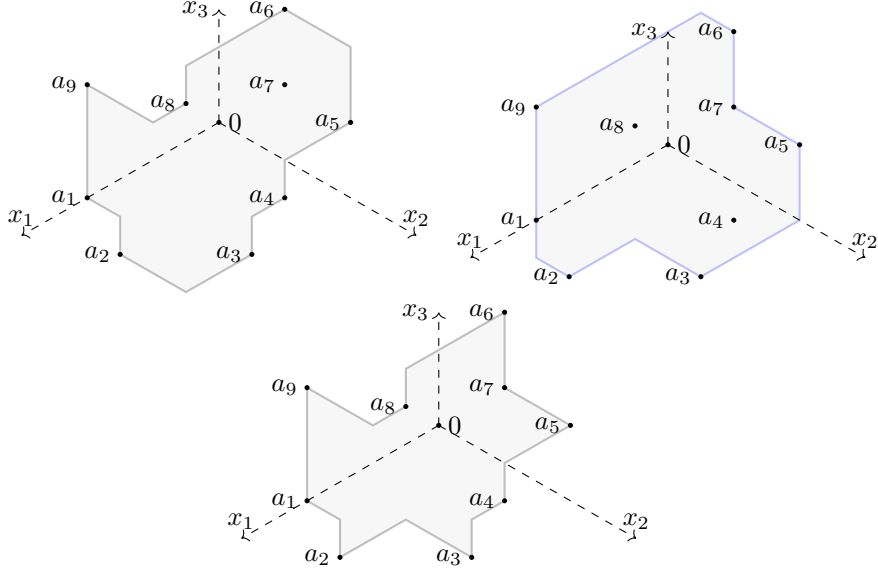
\begin{figure}[htbp]

\def\coord#1#2#3{{-sqrt(3)/2*(#1-#2)} ,{ -(1/2)*#1 - (1/2)*#2 + #3}}

\begin{tikzpicture}[scale=0.5]

\coordinate (a1) at (\coord{4}{0}{0});
\coordinate (a12u) at (\coord{4}{1}{0});
\coordinate (a12l) at (\coord{4}{0}{-1});
\coordinate (a2) at (\coord{5}{2}{0});
\coordinate (a23u) at (\coord{5}{4}{0});
\coordinate (a23l) at (\coord{3}{2}{0});
\coordinate (a3) at (\coord{3}{4}{0});
\coordinate (a34u) at (\coord{2}{3}{0});
\coordinate (a35l) at (\coord{0}{4}{0});
\coordinate (a4) at (\coord{1}{3}{0});
\coordinate (a45u) at (\coord{0}{2}{0});
\coordinate (a5) at (\coord{0}{4}{2});
\coordinate (a56u) at (\coord{0}{4}{4});
\coordinate (a6) at (\coord{0}{2}{4});
\coordinate (a68u) at (\coord{1}{0}{2});
\coordinate (a7) at (\coord{0}{2}{2});
\coordinate (a8) at (\coord{0}{-1}{0});
\coordinate (a89u) at (\coord{1}{-1}{0});
\coordinate (a69l) at (\coord{0}{1}{4});
\coordinate (a9) at (\coord{4}{0}{3});
\coordinate (orig) at (\coord{0}{0}{0});

\filldraw[gray!20,draw=black!80,opacity=0.3, thick] (a1) -- (a12u) -- (a2) -- (a23u) -- (a3) -- (a34u) -- (a4) -- (a45u) -- (a5) -- (a56u) -- (a6) -- (a68u) -- (a8) --(a89u) -- (a9) -- cycle;


\filldraw (a1) circle (1.5pt) node[left] {$a_1$};
\filldraw (a2) circle (1.5pt) node[left] {$a_2$};
\filldraw (a3) circle (1.5pt) node[left] {$a_3$};
\filldraw (a4) circle (1.5pt) node[left] {$a_4$};
\filldraw (a5) circle (1.5pt) node[left] {$a_5$};
\filldraw (a6) circle (1.5pt) node[left] {$a_6$};
\filldraw (a7) circle (1.5pt) node[left] {$a_7$};
\filldraw (a8) circle (1.5pt) node[left] {$a_8$};
\filldraw (a9) circle (1.5pt) node[left] {$a_9$};

\draw[dashed,->] (\coord{0}{0}{0}) -- (\coord{6}{0}{0}) node[above] {$x_1$};
\draw[dashed,->] (\coord{0}{0}{0}) -- (\coord{0}{6}{0}) node[above] {$x_2$};
\draw[dashed,->] (\coord{0}{0}{0}) -- (\coord{0}{0}{3}) node[above, left] {$x_3$};
\filldraw[black] (\coord{0}{0}{0}) circle (1.5pt) node[below,right] {0};

\end{tikzpicture}
\begin{tikzpicture}[scale=0.5]

\coordinate (a1) at (\coord{4}{0}{0});
\coordinate (a12u) at (\coord{4}{1}{0});
\coordinate (a12l) at (\coord{4}{0}{-1});
\coordinate (a2) at (\coord{5}{2}{0});
\coordinate (a23u) at (\coord{5}{4}{0});
\coordinate (a23l) at (\coord{3}{2}{0});
\coordinate (a3) at (\coord{3}{4}{0});
\coordinate (a34u) at (\coord{2}{3}{0});
\coordinate (a35l) at (\coord{0}{4}{0});
\coordinate (a4) at (\coord{1}{3}{0});
\coordinate (a45u) at (\coord{0}{2}{0});
\coordinate (a5) at (\coord{0}{4}{2});
\coordinate (a56u) at (\coord{0}{4}{4});
\coordinate (a6) at (\coord{0}{2}{4});
\coordinate (a68u) at (\coord{1}{0}{2});
\coordinate (a7) at (\coord{0}{2}{2});
\coordinate (a8) at (\coord{0}{-1}{0});
\coordinate (a89u) at (\coord{1}{-1}{0});
\coordinate (a69l) at (\coord{0}{1}{4});
\coordinate (a9) at (\coord{4}{0}{3});
\coordinate (orig) at (\coord{0}{0}{0});


\filldraw[gray!20,draw=blue!80,opacity=0.3, thick] (a1) -- (a12l) -- (a2) -- (a23l) -- (a3) -- (a35l) -- (a5) -- (a7) -- (a6) -- (a69l) -- (a9) -- cycle;

\filldraw (a1) circle (1.5pt) node[left] {$a_1$};
\filldraw (a2) circle (1.5pt) node[left] {$a_2$};
\filldraw (a3) circle (1.5pt) node[left] {$a_3$};
\filldraw (a4) circle (1.5pt) node[left] {$a_4$};
\filldraw (a5) circle (1.5pt) node[left] {$a_5$};
\filldraw (a6) circle (1.5pt) node[left] {$a_6$};
\filldraw (a7) circle (1.5pt) node[left] {$a_7$};
\filldraw (a8) circle (1.5pt) node[left] {$a_8$};
\filldraw (a9) circle (1.5pt) node[left] {$a_9$};

\draw[dashed,->] (\coord{0}{0}{0}) -- (\coord{6}{0}{0}) node[above] {$x_1$};
\draw[dashed,->] (\coord{0}{0}{0}) -- (\coord{0}{6}{0}) node[above] {$x_2$};
\draw[dashed,->] (\coord{0}{0}{0}) -- (\coord{0}{0}{3}) node[above, left] {$x_3$};
\filldraw[black] (\coord{0}{0}{0}) circle (1.5pt) node[below,right] {0};

\end{tikzpicture}
\begin{tikzpicture}[scale=0.5]

\coordinate (a1) at (\coord{4}{0}{0});
\coordinate (a12u) at (\coord{4}{1}{0});
\coordinate (a12l) at (\coord{4}{0}{-1});
\coordinate (a2) at (\coord{5}{2}{0});
\coordinate (a23u) at (\coord{5}{4}{0});
\coordinate (a23l) at (\coord{3}{2}{0});
\coordinate (a3) at (\coord{3}{4}{0});
\coordinate (a34u) at (\coord{2}{3}{0});
\coordinate (a35l) at (\coord{0}{4}{0});
\coordinate (a4) at (\coord{1}{3}{0});
\coordinate (a45u) at (\coord{0}{2}{0});
\coordinate (a5) at (\coord{0}{4}{2});
\coordinate (a56u) at (\coord{0}{4}{4});
\coordinate (a6) at (\coord{0}{2}{4});
\coordinate (a68u) at (\coord{1}{0}{2});
\coordinate (a7) at (\coord{0}{2}{2});
\coordinate (a8) at (\coord{0}{-1}{0});
\coordinate (a89u) at (\coord{1}{-1}{0});
\coordinate (a69l) at (\coord{0}{1}{4});
\coordinate (a9) at (\coord{4}{0}{3});
\coordinate (orig) at (\coord{0}{0}{0});

\filldraw[gray!20,draw=black!80,opacity=0.3, thick] (a1) -- (a12u) -- (a2) -- (a23l) -- (a3) -- (a34u) -- (a4) -- (a45u) -- (a5) -- (a7) -- (a6) -- (a68u) -- (a8) --(a89u) -- (a9) -- cycle;


\filldraw (a1) circle (1.5pt) node[left] {$a_1$};
\filldraw (a2) circle (1.5pt) node[left] {$a_2$};
\filldraw (a3) circle (1.5pt) node[left] {$a_3$};
\filldraw (a4) circle (1.5pt) node[left] {$a_4$};
\filldraw (a5) circle (1.5pt) node[left] {$a_5$};
\filldraw (a6) circle (1.5pt) node[left] {$a_6$};
\filldraw (a7) circle (1.5pt) node[left] {$a_7$};
\filldraw (a8) circle (1.5pt) node[left] {$a_8$};
\filldraw (a9) circle (1.5pt) node[left] {$a_9$};

\draw[dashed,->] (\coord{0}{0}{0}) -- (\coord{6}{0}{0}) node[above] {$x_1$};
\draw[dashed,->] (\coord{0}{0}{0}) -- (\coord{0}{6}{0}) node[above] {$x_2$};
\draw[dashed,->] (\coord{0}{0}{0}) -- (\coord{0}{0}{3}) node[above, left] {$x_3$};
\filldraw[black] (\coord{0}{0}{0}) circle (1.5pt) node[below,right] {0};

\end{tikzpicture}
  \caption{A finite collection of points $E=\{a_1,\dots,a_9\}$; the tropical cone $E^{\max}$ that it generates (left); the dual tropical cone $E^{\min}$ (right); and the ambitropical hull $\range{P_E^{\max}\circ P_E^{\min}} $ (middle).}
  \label{fig-extreme}
  \end{figure}
\section{Homogeneous ambitropical polyhedra}
We next study the class of ambitropical polyhedra
that are {\em homogeneous} in the sense of~\Cref{def-homogeneous}.
We shall see that such polyhedra arise when studying ``locally'' 
ambitropical polyhedra, or when considering their behavior at infinity.
Moreover, they admit a combinatorial characterization,
in terms of posets. 
\begin{defn}
  Let $C$ be an ambitropical polyhedron in $\R^n$ and $u\in C$. The {\em tangent cone} of $C$ at point $u$, denoted by $\tangent_u(C)$, is the set of vectors $v$ such that $u+sv\in C$ holds for all $s\geq 0$ small enough.
\end{defn}
Let us recall the following definition from variational analysis.
\begin{defn}
  A function $T: \R^n\to \R^p$ is {\em semidifferentiable}
  at point $u\in \R^n$ if there exists a continuous map $T'_u$,
  (positively) homogeneous (so $T'_u(\alpha x) = \alpha T'_u(x)$ holds for all $\alpha>0$ and $x\in \R^n$) such that
  \begin{align}\label{e-exp}
  T(u+h) = T(u) + T'_u(h) + o(\|h\|) \enspace .
  \end{align}
\end{defn}
Then, $T'_u(h)$ must coincide with the 
one sided
directional derivative:
\[
T'_u(h) = \lim_{s\to 0^+}s^{-1}(T(u+sh)-T(u))  \enspace.
\]
Conversely, if $T$ is Lipschitz continuous,
an application of Ascoli's theorem shows that if this directional derivate exists for all $h\in \R^n$, then, $T$ is semidifferentiable at point $u$,
see e.g.~\cite[Lemma~3.2]{agn12}.
We shall be consider especially the situation in which
$T$ is continuous and {\em piecewise linear},
  meaning that $\R^n$ can be covered by finitely
many polyhedra on each of which the restriction of $T$
is an affine map. Then, $T$ is automatically
Lipschitz continuous, and the directional derivative
always exists, showing that $T$ is semidifferentiable.
In this case, the semiderivative $h\mapsto T'_u(h)$ is also
piecewise linear, and this entails that the local expansion~\eqref{e-exp}
is exact for $h$ small enough:
\begin{prop} \label{exact}
  Let $T: \R^n \to \R^n$ be piecewise linear, and let $u\in \R^n$. Then,
  there exists a neighborhood $V$ of $0$ such that, for all
  $h\in V$,
  \begin{align}
  T(u+h) = T(u)+T'_u(h) \enspace .\label{e-texact}
  \end{align}\hfill\qed
\end{prop}
The chain rule extends to semidifferentiable
maps: if $f,g$ are Lipschitz continuous, if $f$ is semidifferentiable
at point $u$, and if $g$ is semidifferentiable at point $f(u)$,
then \begin{align}
  (g\circ f)'_u=g'_{f(u)}\circ f'_u\label{chain} \enspace,
\end{align}
see Lemma~3.4 of~\cite{agn12}.

Let us also recall the rule of computation of directional derivatives
of suprema and infima.
  If $f$ is a function $\R^n\to \R$ that can be written
  as a supremum of a finite family of functions $f = \max_{i\in I} f_i$
  and if each $f_i$ has one sided directional derivatives
  at point $u$, then
  \begin{align}\label{semiderivative-max}
  f'_u(h) = \max_{i\in I^*(u)} (f'_i)_u(h) ,\qquad
  \text{where}\;I^*(u) = \{i\in I\mid f(u)=f_i(u)\} \enspace,
  \end{align}
  see~\cite[Exercise~10.27]{rock98}.
  A dual rule applies to a function $f = \min_{i\in I} f_i$.

  We call {\em homogeneous ambitropical polyhedron}
  an ambitropical polyhedron that is a homogeneous
  ambitropical cone. 
\begin{prop}
  Suppose $C$ is an ambitropical polyhedron, and let $u\in C$. Then,
  the tangent cone $\tangent_u(C)$ is a homogeneous ambitropical polyhedron.
\end{prop}
\begin{proof}
  It follows from the definition of $\tangent_u(C)$ that
  if $v\in \tangent_u(C)$, then $sv\in \tangent_u(C)$ for all $s>0$.
  Since $C$ is an ambitropical polyhedron, we have $C=\{x\in \R^n\mid T(x)=x\}$
  where $T$ is a finitely generated Shapley operator.
  We claim that
  \[ \tangent_u(C)= \{h\in \R^n\mid T'_u(h)=h\} \enspace .
  \]
Let $h \in \R^n$ such that $T'_u(h)=h$. For $s>0$ small enough such that $sh$ belongs to the neighborhood $V$ of Proposition \ref{exact}, we have that $T(u+sh)= T(u) + T'_u(sh) =  u+s T'_u(h) = u + sh$. 
It follows that
  $h\in \tangent_u(C)$. Conversely
  if $h\in \tangent_u(C)$, $u+sh\in C$ holds for all $s$ small
  enough, hence, $T(u+sh)=u+sh$ holds for all such $s$,
  and using~\eqref{e-texact}, we deduce that $T'_u(h)=h$.
  Therefore, $\tangent_u(C)$ is the fixed point set
  of the homogeneous Shapley operator $T'_u$.
  Using the chain rule~\eqref{chain}, the
  rule of semidifferentiation of suprema~\eqref{semiderivative-max},
  and the dual rule of semidifferentiation of infima,
  we deduce that 
  $T'_u$ is finitely generated. 

  \end{proof}
\begin{defn}
  Suppose $C$ is a finite union of (ordinary)
  polyhedra. Then, the {\em recession cone}
  of $C$, $\hat{C}$, is the set of vectors $v$ such that
  there is a vector $x\in C$ such that $x+sv$ belongs to $C$ for
  all $s\geq 0$.
\end{defn}
If $T$ is piecewise linear $\R^n\to\R^p$, an in particular
if $T$ is finitely generated, then the {\em recession function}
\[
\hat{T}(x):= \lim_{s\to \infty} s^{-1}T(sx) 
\]
is well defined. Observe that $\hat{T}$ is finitely generated
as soon as $T$ is finitely generated
(this follows from the proof of~\Cref{prop-monboolshapley}).
\begin{prop}
  Suppose $C$ is an ambitropical polyhedron. Then, the recession
  cone $\hat{C}$ is a homogeneous ambitropical polyhedron.
\end{prop}
\begin{proof}
  If $v\in \hat{C}$, then $\alpha v\in \hat{C}$ holds for
  all $\alpha>0$, showing that $C$ is homogeneous. Suppose
  that $C=\{x\in \R^n\mid T(x)=x\}$, where $T$ is a finitely generated
  Shapley operator. We claim
  that $\hat{C}=\{v\in \R^n\mid \hat{T}(v)=v\}$.
  Let $v\in\hat{C}$.
  Then, the exists $y\in C$ such that $y+sv\in C$ holds for
  all $s\geq 0$. So $T(y+sv)=y+sv$. Dividing by $s$, using
  the nonexpansive character of $T$, and letting $s$ tend
  to infinity, we deduce that $\hat{T}(v)=v$.
  Conversely, suppose that $\hat{T}(v)=v$. Then,
  we can find a vector $y$ such that the ray
  $[0,\infty)\ni s\mapsto y+sv$ is included in a
    region in which $T$ is affine. Then, $T(y+sv)=w+sCv$
    for some matrix $C$ and for some vector $w$.
    Specializing at $s=0$, we deduce that $w=T(y)$.
    We also have $T(y+sv)/s\to Cv$ as $s\to\infty$, and so $\hat{T}(v)=Cv$.
    It follows that $T(y+sv) = T(y)+ s\hat{T}(v)$.
    Hence, $T(y+sv)=y+sv$, showing that $y+sv\in C$,
    for all $s\geq 0$, and so $v\in \hat{C}$. Since
    $\hat{T}$ is a finitely generated homogeneous Shapley operator,
    this entails that $\hat{C}$ is a homogeneous ambitropical polyhedron.
  \end{proof}

\begin{thm}\label{th-homogeneousambitropicalpoly}
Let $E$ be a subset of $\mathbb{R}^{n}$, then the following are equivalent:
\begin{enumerate}
\item\label{e-2} $E$ is a homogeneous ambitropical polyhedron;
  \item\label{e-3} There exists a homogeneous  finitely generated Shapley operator such that $P=P^{2}$ and $E=\{ x \in \mathbb{R}^{n} | x = P(x) \}$;
\item\label{e-1} There exists a homogeneous finitely generated Shapley operator $T$ such that $E=\{ x \in \mathbb{R}^{n} | x = T(x) \}$.
\end{enumerate}

\end{thm}
\begin{proof}
  \eqref{e-2}$\Rightarrow$\eqref{e-3}. If $E$ is a homogeneous ambitropical polyhedron,
  we can write $E$ as a finite union $\cup_{k} E_k$ where the $E_k$ are homogeneous alcoved
  polyhedra. Then, the operators $Q_{E_k}^\pm$ are homogeneous and finitely generated.
  We conclude as in the proof of the implication~\eqref{ie-2}$\Rightarrow$\eqref{ie-3}
  of~\Cref{th-ambitropicalpoly}.

  \eqref{e-3}$\Rightarrow$\eqref{e-1}: trivial.

  \eqref{e-1}$\Rightarrow$\eqref{e-2}: by~\Cref{th-ambitropicalpoly}, $E$ is an ambitropical
  polyhedron. Since $E=\{ x \in \mathbb{R}^{n} | x = T(x) \}$, and $T$ is homogeneous,
  $E$ is homogeneous.

  \end{proof}

Given a homogeneous ambitropical polyhedron $C$ of $\R^n$,
we define the {\em skeleton} of $C$, $\skeletton{C}$,
to be the intersection of $C$ with $\{0,1\}^{n}$.
Given an (ordered) partition $\mathcal{I}=(I_1,\dots,I_S)$ of $[n]$,
we define the Weyl cell of $C$ to be
\[
W^{\mathcal{I}} = \{x\in \R^n\mid (i\in I_r, j\in I_s, r\leq s )\implies x_i\leq x_j \}\enspace .
\]
E.g., $\{x\in \R^4\mid x_1\leq x_2=x_3\leq x_4\}$ is the Weyl cell corresponding to the partition $(\{1\}, \{2,3\}, \{4\})$ of the set $\{1, 2, 3, 4\}$. 
When each of the sets $I_1,\dots,I_S$ has exactly one element,
$W^{\mathcal{I}}$ is a Weyl chamber of $A_n$ type,
i.e., a set of the form $\{x\in \R^n\mid x_{\sigma(1)}\leq \dots \leq x_{\sigma(n)}\}$ for some permutation $\sigma$.

We shall need the following observation, which is a variation
on the construction of the canonical triangulation of order polytopes
by Stanley~\cite[\S~5]{stanley86}.
\begin{lem}\label{lem-weyl}
  Any homogeneous ambitropical polyhedron is a union of Weyl cells.
\end{lem}
\begin{proof}
  Any ambitropical polyhedron is a finite union of alcoved polyhedra,
  and if this polyhedron is homogeneous, the alcoved polyhedra
  must be homogeneous. It suffices to show that a homogeneous
  alcoved polyhedron is a finite union of Weyl cells.
  A homogeneous alcoved polyhedron is of the form $E=\{x\mid x_i\leq x_j,\forall (i,j)\in L\}$
  where $L$ is a subset of $[n]\times [n]$. We shall think
  of $L$ as a relation on the set $[n]$, and,
  since, $E$ is unchanged if $L$ is
  replaced by its reflexive and transitive closure, we assume that $L$ is a preorder.
  Then, we define
  the equivalence relation 
  $\equiv_L$, on $[n]$, such that for any $i, j \in [n]$, $i \equiv_L j$ if and only if $(i,j) \in L$ and $(j, i) \in L$. 
  The equivalence classes determined
  by this relation are nonempty subsets of $[n]$,
that constitute a partition of $[n]$. The relation $L$ determines a partial
order $\leq_L$ on the set of these equivalence classes,
the order being defined by $I \leq_L J$ if $(i,j)\in L$ for all $i\in I$ and $j\in J$,
and for all equivalence classes $I,J$.
We choose a linear extension of this partial order, allowing us to
write the equivalence classes as $I_1,\dots,I_S$, in such a way
that $I_{k}\leq_L I_{l}\implies k\leq l$.
  Such a linear extension determines then a Weyl cell $W^{\mathcal{I}}$
  with $\mathcal{I}= (I_{1},\dots,I_{S})$. %
  By construction, $W^{\mathcal{I}} \subset E$. Moreover,
  if $x\in E$, then,  taking $J_1:=\argmin_{i\in [n]}x_i$, we see
  that $J_1$ must be a union of equivalence classes
  $I_{i_1},\dots,I_{i_{m_1}}$. Similarly,
  $J_2= \argmin_{i\in [n]\setminus I_{1}} x_i$ must be a union
  of equivalence classes $I_{i_{m_1+1}},\dots,I_{i_{m_2}}$. Continuing
  in this way, setting $J_k:= \argmin_{i\in [n]\setminus I_{k-1}} x_i=I_{i_{m_{k-1}}}\cup \dots \cup I_{m_k}$until $J_1\cup\dots \cup J_k=[n]$, we get that $x\in W^{\mathcal{I}}$
    with $\mathcal{I}= (I_{i_1},\dots,I_{i_s})$.
    This shows that $E$ is the union of the Weyl cells $W^{\mathcal{I}}$ arising from all
    the linear extensions of the order $\leq_L$. Note that different extensions
    give different Weyl cells.
\end{proof}

The following theorem characterizes the polyhedral
complexes associated with homogeneous ambitropical polyhedra,
showing that they are in bijection with lattices
included in $\{0,1\}^n$. These lattices have been studied by Crapo~\cite{crapo}.%
\todo[inline]{check that it works for cells not necessarily of maximal dimension, them it describes the full polyhedral complex. Explain how a chain yields a cell (this part of proof is missing)}
\begin{thm}[Homogeneous ambitropical cones are equivalent to lattices in $\{0,1\}^n$]\label{th-mainweyl}
  The map $C\mapsto \skeletton C$ establishes
  a bijective correspondence between homogeneous ambitropical polyhedra of $\R^n$
  and subsets of the partially ordered set $(\{0,1\}^n,\leq)$ that contain the bottom and the top element, and that are lattices in the induced order.
  Moreover, there is a one-to-one correspondence
  between the chains in $\skeletton C$ and the Weyl cells included in $C$;
  the cardinality of each of these chains coincides with the dimension of the corresponding Weyl cell plus one unit; and the collection of these Weyl cells constitutes a polyhedral
  fan with support $C$. 
\end{thm}
\begin{proof}
  If $C$ is a homogeneous ambitropical polyhedron,
  then, by~\Cref{th-homogeneousambitropicalpoly},
  there is a homogeneous finitely generated
  Shapley operator $P=P^2:\R^n\to\R^n$ such that $C=\{x\in \R^n\mid x=P(x)\}$.
  The coordinates of a homogeneous finitely generated Shapley operator can be written as  min-max formula without additive translations (i.e., as a monotone
  Boolean formula),  and so, $P$ admits a restriction $\{0,1\}^n\to \{0,1\}^n$.
  Since $\skeletton(C)= C\cap \{0,1\}^n$,
  this entails that $\skeletton(C)=P(\{0,1\}^n)$ is an order
  preserving retract of $\{0,1\}^n$. Moreover, since $P$ is positive homogeneous, the identically zero vector is fixed by $P$, and since $P$ commutes
  with the addition of a constant vector, the unit vector is also fixed by $P$.
This implies $P(\{0,1\}^n)$ is a lattice in the induced order of $\{0,1\}^n$,
containing the bottom and top elements.

We now claim that for all permutations $\sigma$ of $[n]$,
the action of $P$ on the chamber
$W^{\sigma}:= \{x\mid x_{\sigma(1)}\leq \dots \leq x_{\sigma(n)}\}$
is uniquely determined by its action on $\{0,1\}^n$.
In fact, $P$ is linear on the chamber
$W^\sigma$, and since this chamber is a cone with a generating
family consisting of vectors in $\{0,1\}^n$, it follows
that $P$ is uniquely determined by its restriction
to $\{0,1\}^n$. In particular, $P$ fixes
the full chamber $W^\sigma$ if and only if
it fixes each generator of each chamber belonging
to $\{0,1\}^n$. So, the fixed point set of $P$,
which is $C$, is uniquely determined by
the fixed point set of $P$ restricted
to $\{0,1\}^n$, which is $\skeletton(C)$.
This shows that the correspondence between
a homogeneous ambitropical polyhedron and
its skeleton is bijective.

We now show the following claim: every subset $S$ of $\{0,1\}^n$
that is a lattice in the induced order
and contains the bottom and top elements
of $\{0,1\}^n$, can be realized
as a skeleton of a homogeneous
ambitropical polyhedron. The map $T(x):=\sup_S \{u\in \{0,1\}^n\mid u\leq x\}$
is an order preserving self-map of $\{0,1\}^n$
such that
$S=\{x\in \{0,1\}^n\mid x=T(x)\}$.%
Now, by a standard result, any order
preserving map $f$ from $\{0,1\}^n$ to $\{0,1\}$
such that $f(0,\dots,0)=0$ and $f(1,\dots,1)=1$
can be represented by a monotone Boolean function.
Indeed, let $F:=\{y\in \{0,1\}^n\mid f(y)= 1\}$,
and consider the monotone Boolean function
\[
g(x) := \bigvee_{y\in F} \bigwedge_{i\in [n],\,y_i =1} x_i \enspace.
\]
We have that $f(x)=g(x)$ holds for all $x\in \{0,1\}^n$. This establishes
the claim. 

By~\Cref{lem-weyl}, $C$ is a finite union of Weyl cells.
We observe that the intersection of the Weyl cell $W^{\mathcal{I}}$ with $\{0,1\}^n$
is a chain. Indeed, the bottom element of $W^{\mathcal{I}}$ is the zero vector. The smallest
non-zero element of $W^{\mathcal{I}}$ is the vector $x$ gotten by setting $x_i=0$ for all $i\in \cup_{s<S} I_s$
  and $x_i=1$ for all $i\in I_S$. The smallest
  element of $W^{\mathcal{I}}$ greater than $x$ is the vector $y$ gotten by setting $y_i=0$ for all $i\in \cup_{s<S-1} I_s$
  and $y_i=1$ for all $i\in \cup_{s\geq S-1} I_s$, etc.
  This yields a chain of length $S+1$. 
  Moreover, consider the linear space $H^{\mathcal{I}}=\{x\in \R^n\mid
  x_i =x_j, \forall i,j\in I_s,\forall s\in [S]\}$. Every set $I_s$
  yields $|I_s|-1$ independent linear relations, so, $H^{\mathcal{I}}$
  is of dimension $n-(\sum_{s=1}^S|I_s|-1)=n+S-\sum_{s=1}^S|I_s|=S$.
  We have $W^{\mathcal{I}}\subset H^{\mathcal{I}}$, and since for all $0<\alpha_1<\dots<\alpha_S$,
  the vector $x$ such that $x_i = \alpha_s$ for all $i\in I_s$, we deduce that $W^{\mathcal{I}}$
  is of dimension at least $S$.  It follows that $W^{\mathcal{I}}$ is of dimension
  equal to $S$.

  Finally, a face of the Weyl cell associated with an ordered
  partition $I_1,\dots,I_S$ is again a Weyl cell, associated
  with a new partition obtained by merging several
  consecutive sets of the partition in a single class
  (this corresponds to the operation of taking a subchain
  in the skeleton).
  Moreover, taking the intersection of Weyl cells corresponds
  to taking the intersection of the associated chains, which entails
  that the collection of these Weyl cells constitutes
  a polyhedral fan.
  \todo{SV. I checked the proof and it is ok, maybe we should recall the definition of polyhedral fan? Moreover, should we mention that any Weyl cell is a polyhedral cone?}
\end{proof}

We deduce that the class of ambitropical sets is not closed under projection:
\begin{example}\label{not-closed-projection}
Consider the subset $L$ of $\{0,1\}^5$ given by bottom, top and the following elements $(0,1,0,0,1),\\ (0,0,1,0,1), (0,1,1,1,0), (1,1,1,0,1)$, with induced order. The set $L$ is a lattice, so we know by the previous result that it corresponds to an homogeneous ambitropical polyhedron $C$ of $\R^5$, of which it is the skeleton $L=C\cap \{0,1\}^5$.  Let us consider now the projection $\proj(C)$ on $\R^4$ of $C$ which is obtained by taking the first 4 coordinates, and observe that it is a homogeneous polyedron. Assume that it is an ambitropical cone. Then, by \Cref{th-mainweyl} again,
its skeleton would be a lattice.
But the skeleton of $\proj(C)$ is the projection  of the skeleton $L$ of $C$, which is not a lattice since the sup of the two elements $(0,1,0,0)$ and $(0,0,1,0)$ is not well defined, because the two elements $(0,1,1,1)$ and $(1,1,1,0)$ are
minimal upper bounds.
This shows that $\proj(C)$ is not an ambitropical cone.
\end{example}

\begin{example}\label{example1}
  Consider the finitely generated Shapley operator $T:\R^3\to \R^3$,
  \[ T(y)=\big( (x_1\wedge x_2) \vee (x_1\wedge x_3) \vee (x_2 \wedge x_3) ,x_2,x_3\big) \enspace .
  \]
  The fixed point set $E$ of $T$ is the butterfly shaped polyhedral complex with
  two full dimensional cells $E_1$ and $E_2$, shown in \Cref{fig-example}.
  Explicitly,
$  E_1 = \{x\in \R^3\mid x_2\geq x_1 \geq x_3\}$,
  $  E_2 = \{x\in \R^3\mid x_3\geq x_1 \geq x_2\}$.
  Using~\Cref{lem-explicit0},
  we get that the tropical projections are given by:
\begin{align*}
\pmax{E}
\begin{pmatrix}
x_1 \\
x_2 \\
x_3
\end{pmatrix}
=
\begin{pmatrix}
(x_1 \wedge (x_2 \vee x_3)) \\
x_2 \\
x_3 
\end{pmatrix} \text{ and }
\pmin{E}
\begin{pmatrix}
x_1 \\
x_2 \\
x_3
\end{pmatrix}
=
\begin{pmatrix}
(x_1 \vee (x_2 \wedge x_3)) \\
x_2 \\
x_3 
\end{pmatrix}
\end{align*}

\end{example}
\begin{example}
  A union of Weyl cells that is not ambitropical is shown
  at the left of~\Cref{fig-example2b}. 
This union is not ambitropical because
  it does not contain the unique geodesic between two specific
  points of this union, contradicting the conclusion of~\Cref{prop-geodesic}.
  \end{example}
\begin{example}\label{example2}
Consider the ambitropical polyhedron $E$ in Figure~\ref{fig-example2}. We have that in this case 
\[
Q^{-}_{E}
\begin{pmatrix}
x_1 \\
x_2 \\
x_3
\end{pmatrix}
=
\begin{pmatrix}
(x_1 \wedge x_2 \wedge (1+x_3)) \vee (x_1 \wedge (1+x_2) \wedge x_3) \vee (x_2 \wedge x_3 \wedge (1+x_1)) \\
x_2 \wedge (1+x_1) \wedge (1+x_3) \\
x_3 \wedge (1+x_2) \wedge (1+x_1)
\end{pmatrix}
\]
   The construction of the sets $E^{\max}$ and $E^{\min}$,
   as well as~\Cref{th-main1}, showing that $P^{\max}\circ P^{\min}$
   and $P^{\min}\circ P^{\max}$ are retractions
   on an ambitropical set $E$, are illustrated in the figure.
   \end{example}
\colorlet{mygreen}{gray!50!black}
\renewcommand{\baryx}{x_1}
\renewcommand{\baryy}{x_2}
\renewcommand{\baryz}{x_3}

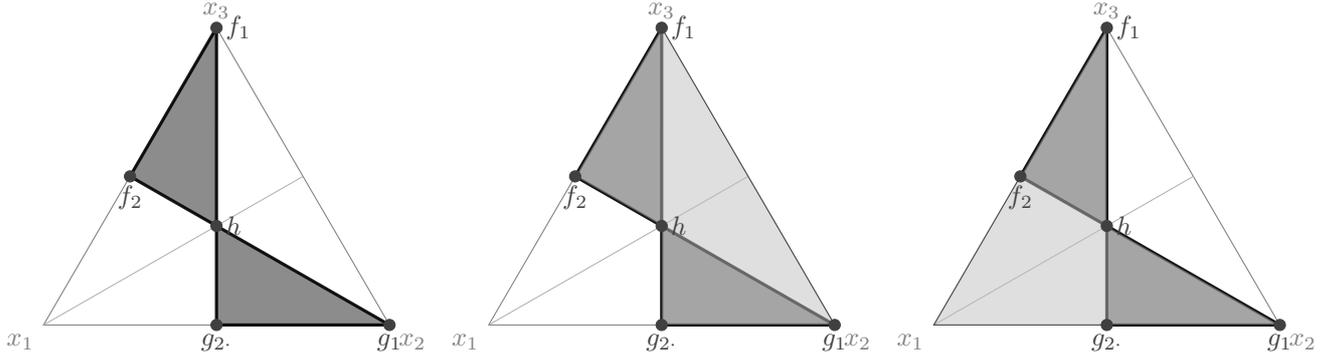
\begin{figure}[htbp]
\begin{center}
\begin{minipage}[b]{0.25\textwidth}
\begin{tikzpicture}%
[scale=0.65,>=triangle 45
,vtx/.style={mygreen},
ray/.style={myred}]
\equilateral{7}{100};
\barycenter{g1}{\expo{-100}}{\expo{0}}{\expo{-100}};
\barycenter{g2}{\expo{0}}{\expo{0}}{\expo{-100}};
\barycenter{h}{\expo{0}}{\expo{0}}{\expo{0}};
\barycenter{f1}{\expo{-100}}{\expo{-100}}{\expo{0}};
\barycenter{f2}{\expo{0}}{\expo{-100}}{\expo{0}};
\filldraw[gray,draw=black,opacity=0.9,very thick] (g1) -- (g2) -- (h) -- cycle;
\filldraw[gray,draw=black,opacity=0.9,very thick] (f1) -- (f2) -- (h) -- cycle;
\filldraw[vtx] (g1) circle (0.75ex) node[below] {$g_{1}$};
\filldraw[vtx] (g2) circle (0.75ex) node[below] {$g_{2\cdot }$};
\filldraw[vtx] (h) circle (0.75ex) node[right] {$h$};
\filldraw[vtx] (f1) circle (0.75ex) node[right] {$f_1$};
\filldraw[vtx] (f2) circle (0.75ex) node[below] {$f_2$};
\end{tikzpicture}
\end{minipage}
\hskip 5em
\begin{minipage}[b]{0.25\textwidth}
\begin{tikzpicture}%
[scale=0.65,>=triangle 45
,vtx/.style={mygreen},
ray/.style={myred}]
\equilateral{7}{100};
\barycenter{g1}{\expo{-100}}{\expo{0}}{\expo{-100}};
\barycenter{g2}{\expo{0}}{\expo{0}}{\expo{-100}};
\barycenter{h}{\expo{0}}{\expo{0}}{\expo{0}};
\barycenter{f1}{\expo{-100}}{\expo{-100}}{\expo{0}};
\barycenter{f2}{\expo{0}}{\expo{-100}}{\expo{0}};
\filldraw[gray,draw=black,opacity=0.9,very thick] (g1) -- (g2) -- (h) -- cycle;
\filldraw[gray,draw=black,opacity=0.9,very thick] (f1) -- (f2) -- (h) -- cycle;
\filldraw[lightgray,draw=black,opacity=0.5] (g1) -- (g2) -- (h) -- (f2) -- (f1) -- (g1) -- cycle;
\filldraw[vtx] (g1) circle (0.75ex) node[below] {$g_{1}$};
\filldraw[vtx] (g2) circle (0.75ex) node[below] {$g_{2\cdot }$};
\filldraw[vtx] (h) circle (0.75ex) node[right] {$h$};
\filldraw[vtx] (f1) circle (0.75ex) node[right] {$f_1$};
\filldraw[vtx] (f2) circle (0.75ex) node[below] {$f_2$};

\end{tikzpicture}
\end{minipage}
\hskip 5em
\begin{minipage}[b]{0.25\textwidth}
\begin{tikzpicture}%
[scale=0.65,>=triangle 45
,vtx/.style={mygreen},
ray/.style={myred}]
\equilateral{7}{100};
\barycenter{e1}{\expo{0}}{0}{0};
\barycenter{e2}{0}{\expo{0}}{0};
\barycenter{e3}{0}{0}{\expo{0}};
\barycenter{g1}{\expo{-100}}{\expo{0}}{\expo{-100}};
\barycenter{g2}{\expo{0}}{\expo{0}}{\expo{-100}};
\barycenter{h}{\expo{0}}{\expo{0}}{\expo{0}};
\barycenter{f1}{\expo{-100}}{\expo{-100}}{\expo{0}};
\barycenter{f2}{\expo{0}}{\expo{-100}}{\expo{0}};
\filldraw[gray,draw=black,opacity=0.9,very thick] (g1) -- (g2) -- (h) -- cycle;
\filldraw[gray,draw=black,opacity=0.9,very thick] (f1) -- (f2) -- (h) -- cycle;
\filldraw[lightgray,draw=black,opacity=0.5] (g1) -- (e1) -- (e3) -- (h) -- (g1) -- cycle;
\filldraw[vtx] (g1) circle (0.75ex) node[below] {$g_{1}$};
\filldraw[vtx] (g2) circle (0.75ex) node[below] {$g_{2\cdot }$};
\filldraw[vtx] (h) circle (0.75ex) node[right] {$h$};
\filldraw[vtx] (f1) circle (0.75ex) node[right] {$f_1$};
\filldraw[vtx] (f2) circle (0.75ex) node[below] {$f_2$};

\end{tikzpicture}
\end{minipage}
\end{center}
\caption{A homogeneous ambitropical polyhedron $E$ consisting of two unbounded alcoved polyhedra (left). The tropical polyhedral cones $E^{\max}$ (center) and $E^{\min}$ (right). See \Cref{example1}.}
\label{fig-example}
\end{figure}
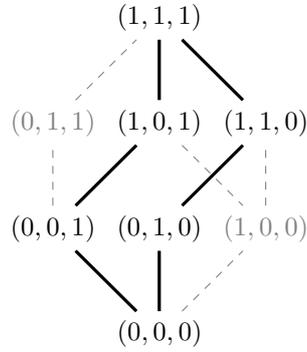
\begin{figure}\begin{center}
\begin{tikzpicture}[scale=0.7]
  \node (max) at (0,4) {{$(1,1,1)$}};
  \node (a) at (-2,2) {{\color{gray}$(0,1,1)$}};
  \node (b) at (0,2) {$(1,0,1)$};
  \node (c) at (2,2) {$(1,1,0)$};
  \node (d) at (-2,0) {$(0,0,1)$};
  \node (e) at (0,0) {$(0,1,0)$};
  \node (f) at (2,0) {{\color{gray}$(1,0,0)$}};
  \node (min) at (0,-2) {{$\mathbf(0,0,0)$}};
  \draw[dashed,gray] (min) -- (d) -- (a) -- (max) -- (b) -- (f)
  (e) -- (min) -- (f) -- (c) -- (max)
  (d) -- (b);
  \draw[very thick] (min) --(d) -- (b) -- (max);
    \draw[very thick] (min) --(e) -- (c) -- (max) ;
\end{tikzpicture}
  \end{center}
  \caption{The skeleton (in bold) of the ambitropical polyhedron of \Cref{fig-example} (the two nodes $(1,0,0)$ and $(0,1,1)$ in gray do not belong to the skeleton). The two maximal chains (of length $4$) yield the representation of $C$ as the union of two Weyl cells (each being of dimension $3$).}
  \end{figure}

\begin{figure}
\begin{center}
\begin{minipage}[b]{0.25\textwidth}
\begin{tikzpicture}%
[scale=0.65,>=triangle 45,vtx/.style={mygreen},
ray/.style={myred}]
\equilateral{7}{100};
\barycenter{g1}{\expo{-2}}{\expo{0}}{\expo{-2}};
\barycenter{g2}{\expo{0}}{\expo{0}}{\expo{-2}};
\barycenter{h}{\expo{0}}{\expo{0}}{\expo{0}};
\barycenter{f1}{\expo{-2}}{\expo{-2}}{\expo{0}};
\barycenter{f2}{\expo{0}}{\expo{-2}}{\expo{0}};
\filldraw[gray,draw=black,opacity=0.9,very thick] (g1) -- (g2) -- (h) -- cycle;
\filldraw[gray,draw=black,opacity=0.9,very thick] (f1) -- (f2) -- (h) -- cycle;
\filldraw[vtx] (g1) circle (0.75ex) node[below] {$g_{1}$};
\filldraw[vtx] (g2) circle (0.75ex) node[below] {$g_{2\cdot }$};
\filldraw[vtx] (h) circle (0.75ex) node[right] {$h$};
\filldraw[vtx] (f1) circle (0.75ex) node[right] {$f_1$};
\filldraw[vtx] (f2) circle (0.75ex) node[below] {$f_2$};
\end{tikzpicture}
\end{minipage}
\hskip 5em
\begin{minipage}[b]{0.25\textwidth}
\begin{tikzpicture}%
[scale=0.65,>=triangle 45
,vtx/.style={mygreen},
ray/.style={myred}]
\equilateral{7}{100};
\barycenter{g1}{\expo{-2}}{\expo{0}}{\expo{-2}};
\barycenter{g2}{\expo{0}}{\expo{0}}{\expo{-2}};
\barycenter{h}{\expo{0}}{\expo{0}}{\expo{0}};
\barycenter{f1}{\expo{-2}}{\expo{-2}}{\expo{0}};
\barycenter{f2}{\expo{0}}{\expo{-2}}{\expo{0}};
\barycenter{f3}{\expo{-2}}{\expo{0}}{\expo{0}};

\filldraw[gray,draw=black,opacity=0.9,very thick] (g1) -- (g2) -- (h) -- cycle;
\filldraw[gray,draw=black,opacity=0.9,very thick] (f1) -- (f2) -- (h) -- cycle;
\filldraw[lightgray,draw=black,opacity=0.5] (g1) -- (g2) -- (h) -- (f2) -- (f1) -- (f3) -- (g1) -- cycle;
\filldraw[vtx] (g1) circle (0.75ex) node[below] {$g_{1}$};
\filldraw[vtx] (g2) circle (0.75ex) node[below] {$g_{2\cdot }$};
\filldraw[vtx] (h) circle (0.75ex) node[right] {$h$};
\filldraw[vtx] (f1) circle (0.75ex) node[right] {$f_1$};
\filldraw[vtx] (f2) circle (0.75ex) node[below] {$f_2$};

\end{tikzpicture}
\end{minipage}
\hskip 5em
\begin{minipage}[b]{0.25\textwidth}
\begin{tikzpicture}%
[scale=0.65,>=triangle 45
,vtx/.style={mygreen},
ray/.style={myred}]
\equilateral{7}{100};
\barycenter{e1}{\expo{0}}{0}{0};
\barycenter{e2}{0}{\expo{0}}{0};
\barycenter{e3}{0}{0}{\expo{0}};
\barycenter{g1}{\expo{-2}}{\expo{0}}{\expo{-2}};
\barycenter{g2}{\expo{0}}{\expo{0}}{\expo{-2}};
\barycenter{h}{\expo{0}}{\expo{0}}{\expo{0}};
\barycenter{f1}{\expo{-2}}{\expo{-2}}{\expo{0}};
\barycenter{f2}{\expo{0}}{\expo{-2}}{\expo{0}};
\barycenter{f4}{\expo{2}}{\expo{0}}{\expo{0}};

\filldraw[gray,draw=black,opacity=0.9,very thick] (g1) -- (g2) -- (h) -- cycle;
\filldraw[gray,draw=black,opacity=0.9,very thick] (f1) -- (f2) -- (h) -- cycle;
\filldraw[lightgray,draw=black,opacity=0.5] (g1) -- (g2) -- (f4) -- (f2) -- (f1) -- (h) -- (g1) -- cycle;
\filldraw[vtx] (g1) circle (0.75ex) node[below] {$g_{1}$};
\filldraw[vtx] (g2) circle (0.75ex) node[below] {$g_{2\cdot }$};
\filldraw[vtx] (h) circle (0.75ex) node[right] {$h$};
\filldraw[vtx] (f1) circle (0.75ex) node[right] {$f_1$};
\filldraw[vtx] (f2) circle (0.75ex) node[below] {$f_2$};

\end{tikzpicture}
\end{minipage}
\end{center}

\caption{An ambitropical polyhedron $E$ consisting of two alcoved polyhedra (left). The tropical polyhedral cones $E^{\max}$ (center) and $E^{\min}$ (right). The homogeneous ambitropical polyhedron of \Cref{fig-example} is precisely the tangent cone $\mathcal{T}_{(0,0,0)}E$.}
\label{fig-example2}
\end{figure}
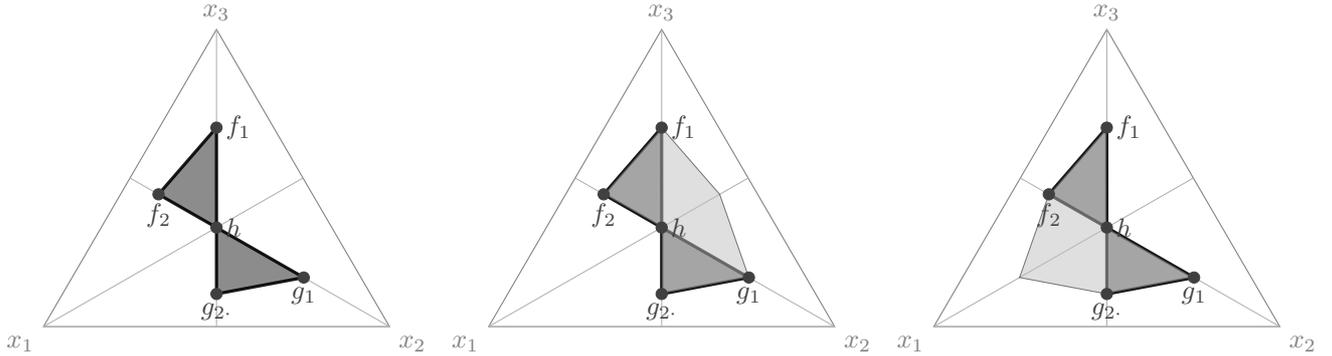
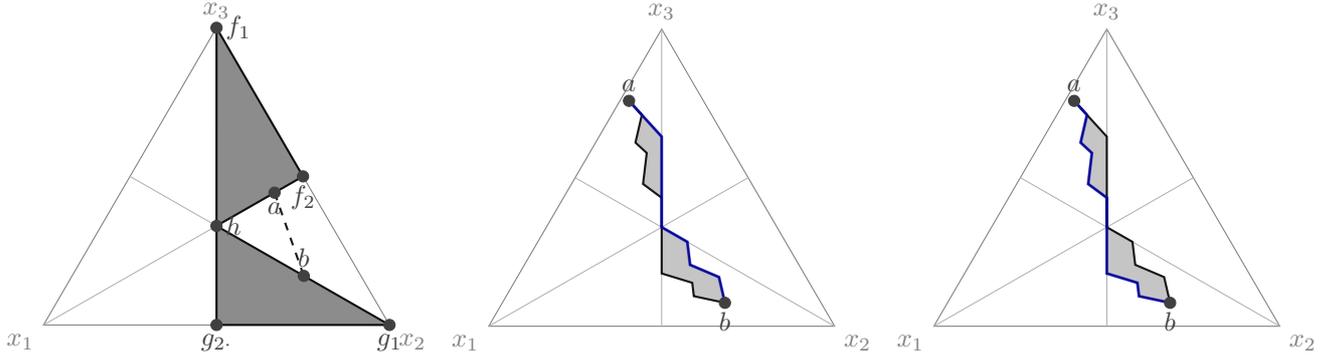
\begin{figure}
\begin{center}
\begin{minipage}[b]{0.25\textwidth}
\begin{tikzpicture}%
[scale=0.65,>=triangle 45,vtx/.style={mygreen},
ray/.style={myred}]
\equilateral{7}{100};
\barycenter{g1}{\expo{-100}}{\expo{0}}{\expo{-100}};
\barycenter{g2}{\expo{0}}{\expo{0}}{\expo{-100}};
\barycenter{h}{\expo{0}}{\expo{0}}{\expo{0}};
\barycenter{f1}{\expo{-100}}{\expo{-100}}{\expo{0}};
\barycenter{f2}{\expo{-100}}{\expo{0}}{\expo{0}};
\barycenter{a}{\expo{-2}}{\expo{0}}{\expo{0}};
\barycenter{b}{\expo{0}}{\expo{2}}{\expo{0}};
\filldraw[gray,draw=black,opacity=0.9,thick] (g1) -- (g2) -- (h) -- cycle;
\filldraw[gray,draw=black,opacity=0.9,thick] (f1) -- (f2) -- (h) -- cycle;
\draw[dashed,draw=black,opacity=0.9,thick] (a) -- (b);
\filldraw[vtx] (g1) circle (0.75ex) node[below] {$g_{1}$};
\filldraw[vtx] (g2) circle (0.75ex) node[below] {$g_{2\cdot }$};
\filldraw[vtx] (a) circle (0.75ex) node[below] {$a$};
\filldraw[vtx] (b) circle (0.75ex) node[above] {$b$};
\filldraw[vtx] (h) circle (0.75ex) node[right] {$h$};
\filldraw[vtx] (f1) circle (0.75ex) node[right] {$f_1$};
\filldraw[vtx] (f2) circle (0.75ex) node[below] {$f_2$};
\end{tikzpicture}
\end{minipage}
\hskip 5em
\begin{minipage}[b]{0.25\textwidth}
\begin{tikzpicture}%
[scale=0.65,>=triangle 45
,vtx/.style={mygreen},
ray/.style={myred}]
\equilateral{7}{100};
\barycenter{g1}{\expo{0.6*-2}}{\expo{0.6*0}}{\expo{0.6*-2}};
\barycenter{g2}{\expo{0.6*0}}{\expo{0.6*0}}{\expo{0.6*-2}};
\barycenter{h}{\expo{0.6*0}}{\expo{0.6*0}}{\expo{0.6*0}};
\barycenter{f1}{\expo{0.6*-2}}{\expo{0.6*-2}}{\expo{0.6*0}};
\barycenter{f2}{\expo{0.6*0}}{\expo{0.6*-2}}{\expo{0.6*0}};
\barycenter{f3}{\expo{0.6*-2}}{\expo{0.6*0}}{\expo{0.6*0}};

\barycenter{a1}{\expo{0.6*0}}{\expo{0.6*-5}}{\expo{0.6*3}};
\barycenter{a2}{\expo{0.6*0}}{\expo{0.6*-2}}{\expo{0.6*3}};
\barycenter{a3}{\expo{0.6*0}}{\expo{0.6*0}}{\expo{0.6*3}};
\barycenter{a4}{\expo{0.6*0}}{\expo{0.6*0}}{\expo{0.6*0}};
\barycenter{a5}{\expo{0.6*0}}{\expo{0.6*1}}{\expo{0.6*0}};
\barycenter{a6}{\expo{0.6*0}}{\expo{0.6*1}}{\expo{0.6*-1}};
\barycenter{a7}{\expo{0.6*0}}{\expo{0.6*2}}{\expo{0.6*-1}};
\barycenter{a8}{\expo{0.6*0}}{\expo{0.6*2}}{\expo{0.6*-3}};

\barycenter{b1}{\expo{0.6*0}}{\expo{0.6*-2}}{\expo{0.6*2}};
\barycenter{b1p}{\expo{0.6*0}}{\expo{0.6*-1}}{\expo{0.6*2}};
\barycenter{b2}{\expo{0.6*0}}{\expo{0.6*-1}}{\expo{0.6*1}};
\barycenter{b2p}{\expo{0.6*0}}{\expo{0.6*0}}{\expo{0.6*1}};
\barycenter{b3}{\expo{0.6*0}}{\expo{0.6*0}}{\expo{0.6*-2}};
\barycenter{b4}{\expo{0.6*0}}{\expo{0.6*1}}{\expo{0.6*-2}};
\barycenter{b5}{\expo{0.6*0}}{\expo{0.6*1}}{\expo{0.6*-3}};

\filldraw[lightgray,draw=black,opacity=0.9, thick] (a1) -- (a2) -- (a3) -- (b2p) -- (a4) -- (b2p) -- (b2) -- (b1p) -- (b1) -- (a2) --cycle;

\filldraw[lightgray,draw=black,opacity=0.9, thick] (a4) -- (a5) -- (a6) -- (a7)  -- (a8) -- (b5) -- (b4) -- (b3) --cycle;

\draw[blue,opacity=0.5,very thick] (a1) -- (a2) -- (a3) -- (b2p) -- (a4) -- (a5) -- (a6) -- (a7)  -- (a8);

\filldraw[vtx] (a1) circle (0.75ex) node[above] {$a$};
\filldraw[vtx] (a8) circle (0.75ex) node[below] {$b$};

\end{tikzpicture}
\end{minipage}
\hskip 5em
\begin{minipage}[b]{0.25\textwidth}
\begin{tikzpicture}%
[scale=0.65,>=triangle 45
,vtx/.style={mygreen},
ray/.style={myred}]
\equilateral{7}{100};
\barycenter{g1}{\expo{0.6*-2}}{\expo{0.6*0}}{\expo{0.6*-2}};
\barycenter{g2}{\expo{0.6*0}}{\expo{0.6*0}}{\expo{0.6*-2}};
\barycenter{h}{\expo{0.6*0}}{\expo{0.6*0}}{\expo{0.6*0}};
\barycenter{f1}{\expo{0.6*-2}}{\expo{0.6*-2}}{\expo{0.6*0}};
\barycenter{f2}{\expo{0.6*0}}{\expo{0.6*-2}}{\expo{0.6*0}};
\barycenter{f3}{\expo{0.6*-2}}{\expo{0.6*0}}{\expo{0.6*0}};

\barycenter{a1}{\expo{0.6*0}}{\expo{0.6*-5}}{\expo{0.6*3}};
\barycenter{a2}{\expo{0.6*0}}{\expo{0.6*-2}}{\expo{0.6*3}};
\barycenter{a3}{\expo{0.6*0}}{\expo{0.6*0}}{\expo{0.6*3}};
\barycenter{a4}{\expo{0.6*0}}{\expo{0.6*0}}{\expo{0.6*0}};
\barycenter{a5}{\expo{0.6*0}}{\expo{0.6*1}}{\expo{0.6*0}};
\barycenter{a6}{\expo{0.6*0}}{\expo{0.6*1}}{\expo{0.6*-1}};
\barycenter{a7}{\expo{0.6*0}}{\expo{0.6*2}}{\expo{0.6*-1}};
\barycenter{a8}{\expo{0.6*0}}{\expo{0.6*2}}{\expo{0.6*-3}};

\barycenter{b1}{\expo{0.6*0}}{\expo{0.6*-2}}{\expo{0.6*2}};
\barycenter{b1p}{\expo{0.6*0}}{\expo{0.6*-1}}{\expo{0.6*2}};
\barycenter{b2}{\expo{0.6*0}}{\expo{0.6*-1}}{\expo{0.6*1}};
\barycenter{b2p}{\expo{0.6*0}}{\expo{0.6*0}}{\expo{0.6*1}};
\barycenter{b3}{\expo{0.6*0}}{\expo{0.6*0}}{\expo{0.6*-2}};
\barycenter{b4}{\expo{0.6*0}}{\expo{0.6*1}}{\expo{0.6*-2}};
\barycenter{b5}{\expo{0.6*0}}{\expo{0.6*1}}{\expo{0.6*-3}};

\filldraw[lightgray,draw=black,opacity=0.9, thick] (a1) -- (a2) -- (a3) -- (b2p) -- (a4) -- (b2p) -- (b2) -- (b1p) -- (b1) -- (a2) --cycle;

\filldraw[lightgray,draw=black,opacity=0.9, thick] (a4) -- (a5) -- (a6) -- (a7)  -- (a8) -- (b5) -- (b4) -- (b3) --cycle;

\draw[blue,opacity=0.5,very thick] (a1) -- (a2) -- (b1) -- (b1p) -- (b2) -- (b2p) -- (a4) -- (b3)  -- (b4) -- (b5) -- (a8);

\filldraw[vtx] (a1) circle (0.75ex) node[above] {$a$};
\filldraw[vtx] (a8) circle (0.75ex) node[below] {$b$};

\end{tikzpicture}
\end{minipage}
\end{center}

\caption{Illustrating the metric convexity property of ambitropical cones. A fan which is not ambitropical (left). The geodesic in Hilbert's projective metric connecting the two points $a$ and $b$ (dotted segment) is unique, but it is not included in the fan. An ambitropical fan (middle, and right). An example of geodesic connecting $a$ and $b$, included in the fan is shown (dark blue broken line, middle). Another example of such a geodesic is shown at right.}
\label{fig-example2b}
\end{figure}

\begin{figure}
\begin{center}
\begin{minipage}[b]{0.25\textwidth}
\begin{tikzpicture}%
[scale=0.65,>=triangle 45
,vtx/.style={mygreen},
ray/.style={myred}]
\equilateral{7}{100};

\barycenter{a}{\expo{0.6*4}}{\expo{0.6*0}}{\expo{0.6*0}};
\barycenter{a1}{\expo{0.6*3}}{\expo{0.6*0}}{\expo{0.6*0}};
\barycenter{a2}{\expo{0.6*3}}{\expo{0.6*1}}{\expo{0.6*0}};
\barycenter{a3}{\expo{0.6*2}}{\expo{0.6*1}}{\expo{0.6*0}};
\barycenter{a4}{\expo{0.6*2}}{\expo{0.6*2}}{\expo{0.6*0}};
\barycenter{a5}{\expo{0.6*1}}{\expo{0.6*2}}{\expo{0.6*0}};
\barycenter{a6}{\expo{0.6*1}}{\expo{0.6*3}}{\expo{0.6*0}};
\barycenter{a7}{\expo{0.6*0}}{\expo{0.6*3}}{\expo{0.6*0}};
\barycenter{b}{\expo{0.6*0}}{\expo{0.6*4}}{\expo{0.6*0}};

\barycenter{o}{\expo{0.6*0}}{\expo{0.6*0}}{\expo{0.6*0}};
\barycenter{p}{\expo{0.6*4}}{\expo{0.6*4}}{\expo{0.6*0}};

\filldraw (a) circle (0.75ex) node[above] {$a$};
\filldraw (b) circle (0.75ex) node[below] {$b$};
\filldraw (p) node[below] {$E^{\max}$};
\filldraw (o) node[above] {$E^{\min}$};

\draw[red,very thick] (a) -- (o) -- (b);
\draw[blue,very thick] (a) -- (p) -- (b);
\draw[black,very thick] (a) -- (a1) -- (a2) -- (a3) -- (a4) -- (a5) -- (a6) -- (a7) -- (b);

\end{tikzpicture}
\end{minipage}
\end{center}

\caption{Three examples of ambitropical hulls of a set with two elements $E=\{a,b\}$. The range of $\barqp{E}$ is the tropical cone $E^{\max}$ shown in blue, whereas the range of $\barqm{E}$ is the dual tropical cone $E^{\min}$ shown in red. Another ambitropical hull is represented by the zigzag line (in black). By~\Cref{prop-ambihull}, all these ambitropical hulls are isomorphic.}
\label{fig-hull}
\end{figure}
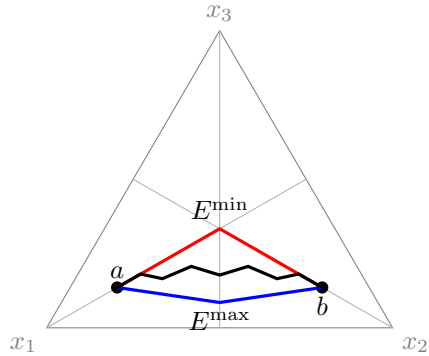
\begin{example}
  The construction of the ambitropical hulls by means of~\Cref{prop-ambihull}
  is illustrated in~\Cref{fig-hull}. Here, $E=\{a,b\}$
  where $a=(1,0,0)$ and $b=(0,1,0)$. In this special case,
  the sets $E^{\max}$ and $E^{\min}$ are the ranges of $\barqp{E}$
  and $\barqm{E}$, respectively, and so, they provides
  ambitropical hulls of $\{a,b\}$. There is an infinite family
  of ambitropical hulls interpolating between $E^{\max}$ and $E^{\min}$.
    One element of this family is shown in black. It constitutes a polyhedral
    complex whose cells are cones, and whose rays are generated
    by the vectors $a,b$ and by the integer vectors
    $(i,4-j,0)$ for $i=1,2,3$ and $(i-1,4-j,0)$
    for $i=1,\dots,4$.
\end{example}
\begin{example}
We now give an example in dimension $4$. Consider
\[
E^I = \{x\in \R^4\mid x_1\leq x_2\leq x_3\leq x_4\}
\]
and for the circular permutation $\gamma$ with cycle $(1,2,3,4)$,
\[
E^\gamma=\{x\in \R^4\mid x_4\leq x_1\leq x_2\leq x_3 \} \enspace .
\]
The union of the chambers $E^I$ and $E^\gamma$ is shown in \Cref{fig-3d},
as well as the range of $\qm{E}$, which is larger
than this union, implying that $E$ is not ambitropical.
AN example of non-trivial ambitropical cone in $\R^4$ is shown in~\Cref{fig-3d5}.
\end{example}

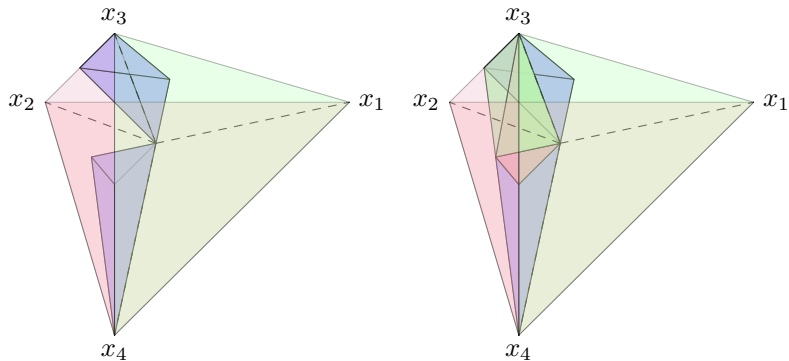
\begin{figure}
\begin{center}
\begin{tikzpicture}[line join = round, line cap = round]
\pgfmathsetmacro{\factor}{1/sqrt(2)};
\coordinate [label=right:$x_1$] (e1) at (2,0,-2*\factor);
\coordinate [label=left:$x_2$] (e2) at (-2,0,-2*\factor);
\coordinate [label=above:$x_3$] (e3) at (0,2,2*\factor);
\coordinate [label=below:$x_4$] (e4) at (0,-2,2*\factor);
\coordinate (e) at (0,0,0*\factor);
\coordinate (f12) at (0,0,-2*\factor);
\coordinate (f34) at (0,0,2*\factor);
\coordinate (f23) at (-1,1,0*\factor);
\coordinate (g234) at (-2/3,0,2*\factor/3);
\coordinate (g123) at (0,2/3,-2*\factor/3);
\draw[-, fill=blue!30, opacity=.7] (e)--(e4)--(f34)--cycle;
\draw[-, fill=blue!30, opacity=.7] (e)--(f34)--(g234)--cycle;
\draw[-, fill=blue!30, opacity=.7] (e)--(e4)--(g234)--cycle;
\draw[-, fill=blue!30, opacity=.7] (f34)--(e4)--(g234)--cycle;
\draw[-, fill=blue!30, opacity=.7] (e)--(e3)--(f23)--cycle;
\draw[-, fill=blue!30, opacity=.7] (e)--(f23)--(g123)--cycle;
\draw[-, fill=blue!30, opacity=.7] (e)--(e3)--(g123)--cycle;
\draw[-, fill=blue!30, opacity=.7] (f23)--(e3)--(g123)--cycle;

\foreach \i in {e1,e2,e3,e4}
    \draw[dashed] (0,0)--(\i);
\draw[-, fill=red!30, opacity=.3] (e1)--(e4)--(e2)--cycle;
\draw[-, fill=green!30, opacity=.3] (e1) --(e4)--(e3)--cycle;
\draw[-, fill=purple!30, opacity=.3] (e2)--(e4)--(e3)--cycle;
\end{tikzpicture}
\begin{tikzpicture}[line join = round, line cap = round]
\pgfmathsetmacro{\factor}{1/sqrt(2)};
\coordinate [label=right:$x_1$] (e1) at (2,0,-2*\factor);
\coordinate [label=left:$x_2$] (e2) at (-2,0,-2*\factor);
\coordinate [label=above:$x_3$] (e3) at (0,2,2*\factor);
\coordinate [label=below:$x_4$] (e4) at (0,-2,2*\factor);
\coordinate (e) at (0,0,0*\factor);
\coordinate (f12) at (0,0,-2*\factor);
\coordinate (f34) at (0,0,2*\factor);
\coordinate (f23) at (-1,1,0*\factor);
\coordinate (g234) at (-2/3,0,2*\factor/3);
\coordinate (g123) at (0,2/3,-2*\factor/3);
\draw[-, fill=blue!30, opacity=.7] (e)--(e4)--(f34)--cycle;
\draw[-, fill=blue!30, opacity=.7] (e)--(f34)--(g234)--cycle;
\draw[-, fill=blue!30, opacity=.7] (e)--(e4)--(g234)--cycle;
\draw[-, fill=blue!30, opacity=.7] (f34)--(e4)--(g234)--cycle;
\draw[-, fill=blue!30, opacity=.7] (e)--(e3)--(f23)--cycle;
\draw[-, fill=blue!30, opacity=.7] (e)--(f23)--(g123)--cycle;
\draw[-, fill=blue!30, opacity=.7] (e)--(e3)--(g123)--cycle;
\draw[-, fill=blue!30, opacity=.7] (f23)--(e3)--(g123)--cycle;

\draw[-, fill=red!30, opacity=.5] (e)--(e3)--(f34)--cycle;
\draw[-, fill=red!30, opacity=.5] (e)--(f34)--(g234)--cycle;
\draw[-, fill=red!30, opacity=.5] (e)--(e3)--(g234)--cycle;
\draw[-, fill=red!30, opacity=.5] (f34)--(e3)--(g234)--cycle;

\draw[-, fill=green!30, opacity=.5] (e)--(e3)--(f23)--cycle;
\draw[-, fill=green!30, opacity=.5] (e)--(f23)--(g234)--cycle;
\draw[-, fill=green!30, opacity=.5] (e)--(e3)--(g234)--cycle;
\draw[-, fill=green!30, opacity=.5] (f23)--(e3)--(g234)--cycle;

\foreach \i in {e1,e2,e3,e4}
    \draw[dashed] (0,0)--(\i);
\draw[-, fill=red!30, opacity=.3] (e1)--(e4)--(e2)--cycle;
\draw[-, fill=green!30, opacity=.3] (e1) --(e4)--(e3)--cycle;
\draw[-, fill=purple!30, opacity=.3] (e2)--(e4)--(e3)--cycle;
\end{tikzpicture}
\end{center}
\caption{The union $E$ of the chambers $x_4\geq x_3\geq x_2\geq x_1$ and $x_3\geq x_2 \geq x_1\geq x_4$ (left) is not ambitropical. Indeed, the fixed point set of $\qm{E}$ (middle) is the union of these chambers with the chambers $x_3\geq x_4\geq x_2\geq x_1$ and $x_3\geq x_2 \geq x_4\geq x_1$, which coincides with the alcoved polyhedron defined by $x_3\geq x_2\geq x_1$.  The latter coincides with the range of $\pmax{E}$ and with the range of $\pmin{E}$.} 
\label{fig-3d}
\end{figure}

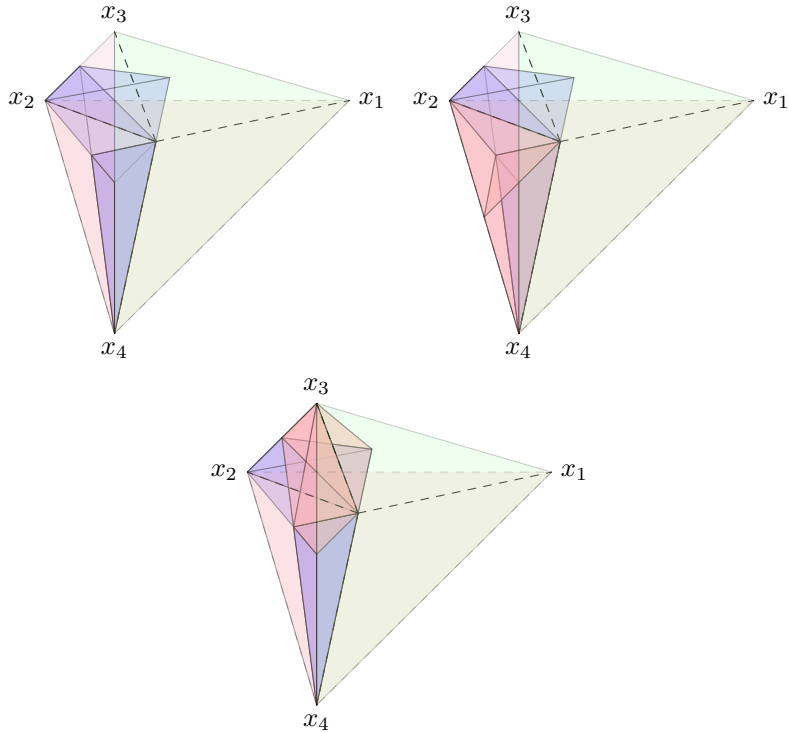
\begin{figure}
\begin{center}
\begin{tikzpicture}[line join = round, line cap = round]
\pgfmathsetmacro{\factor}{1/sqrt(2)};
\coordinate [label=right:$x_1$] (e1) at (2,0,-2*\factor);
\coordinate [label=left:$x_2$] (e2) at (-2,0,-2*\factor);
\coordinate [label=above:$x_3$] (e3) at (0,2,2*\factor);
\coordinate [label=below:$x_4$] (e4) at (0,-2,2*\factor);
\coordinate (e) at (0,0,0*\factor);
\coordinate (f12) at (0,0,-2*\factor);
\coordinate (f24) at (-1,-1,0*\factor);
\coordinate (f34) at (0,0,2*\factor);
\coordinate (f23) at (-1,1,0*\factor);
\coordinate (g234) at (-2/3,0,2*\factor/3);
\coordinate (g123) at (0,2/3,-2*\factor/3);
\coordinate (g124) at (0,-2/3,-2*\factor/3);
\coordinate (g134) at (2/3,0,2*\factor/3);
\draw[-, fill=blue!30, opacity=.8] (e)--(e4)--(f34)--cycle;
\draw[-, fill=blue!30, opacity=.8] (e)--(f34)--(g234)--cycle;
\draw[-, fill=blue!30, opacity=.8] (e)--(e4)--(g234)--cycle;
\draw[-, fill=blue!30, opacity=.8] (f34)--(e4)--(g234)--cycle;
\draw[-, fill=blue!30, opacity=.3] (e)--(e2)--(f23)--cycle;
\draw[-, fill=blue!30, opacity=.3] (e)--(f23)--(g234)--cycle;
\draw[-, fill=blue!30, opacity=.3] (e)--(e2)--(g234)--cycle;
\draw[-, fill=blue!30, opacity=.3] (f23)--(e2)--(g234)--cycle;

\draw[-, fill=blue!30, opacity=.3] (e)--(e2)--(f23)--cycle;
\draw[-, fill=blue!30, opacity=.3] (e)--(f23)--(g123)--cycle;
\draw[-, fill=blue!30, opacity=.3] (e)--(e2)--(g123)--cycle;
\draw[-, fill=blue!30, opacity=.3] (f23)--(e2)--(g123)--cycle;

\foreach \i in {e1,e2,e3,e4}
    \draw[dashed] (0,0)--(\i);
    \draw[dashed, fill=red!30, opacity=.2] (e1)--(e4)--(e2)--cycle;
    \draw[-,opacity=.2] (e2) -- (e4);
\draw[-, fill=green!30, opacity=.2] (e1) --(e4)--(e3)--cycle;
\draw[-, fill=purple!30, opacity=.2] (e2)--(e4)--(e3)--cycle;
\end{tikzpicture}
\begin{tikzpicture}[line join = round, line cap = round]
\pgfmathsetmacro{\factor}{1/sqrt(2)};
\coordinate [label=right:$x_1$] (e1) at (2,0,-2*\factor);
\coordinate [label=left:$x_2$] (e2) at (-2,0,-2*\factor);
\coordinate [label=above:$x_3$] (e3) at (0,2,2*\factor);
\coordinate [label=below:$x_4$] (e4) at (0,-2,2*\factor);
\coordinate (e) at (0,0,0*\factor);
\coordinate (f12) at (0,0,-2*\factor);
\coordinate (f24) at (-1,-1,0*\factor);
\coordinate (f34) at (0,0,2*\factor);
\coordinate (f23) at (-1,1,0*\factor);
\coordinate (g234) at (-2/3,0,2*\factor/3);
\coordinate (g123) at (0,2/3,-2*\factor/3);
\coordinate (g124) at (0,-2/3,-2*\factor/3);
\coordinate (g134) at (2/3,0,2*\factor/3);

\draw[-, fill=blue!30, opacity=.8] (e)--(e4)--(f34)--cycle;
\draw[-, fill=blue!30, opacity=.8] (e)--(f34)--(g234)--cycle;
\draw[-, fill=blue!30, opacity=.8] (e)--(e4)--(g234)--cycle;
\draw[-, fill=blue!30, opacity=.8] (f34)--(e4)--(g234)--cycle;
\draw[-, fill=blue!30, opacity=.3] (e)--(e2)--(f23)--cycle;
\draw[-, fill=blue!30, opacity=.3] (e)--(f23)--(g234)--cycle;
\draw[-, fill=blue!30, opacity=.3] (e)--(e2)--(g234)--cycle;
\draw[-, fill=blue!30, opacity=.3] (f23)--(e2)--(g234)--cycle;

\draw[-, fill=blue!30, opacity=.3] (e)--(e2)--(f23)--cycle;
\draw[-, fill=blue!30, opacity=.3] (e)--(f23)--(g123)--cycle;
\draw[-, fill=blue!30, opacity=.3] (e)--(e2)--(g123)--cycle;
\draw[-, fill=blue!30, opacity=.3] (f23)--(e2)--(g123)--cycle;

\draw[-, fill=red!30, opacity=.3] (e)--(e4)--(f24)--cycle;
\draw[-, fill=red!30, opacity=.3] (e)--(f24)--(g234)--cycle;
\draw[-, fill=red!30, opacity=.3] (e)--(e4)--(g234)--cycle;
\draw[-, fill=red!30, opacity=.3] (f24)--(e4)--(g234)--cycle;

\draw[-, fill=red!30, opacity=.3] (e)--(e2)--(f24)--cycle;
\draw[-, fill=red!30, opacity=.3] (e)--(f24)--(g234)--cycle;
\draw[-, fill=red!30, opacity=.3] (e)--(e2)--(g234)--cycle;
\draw[-, fill=red!30, opacity=.3] (f24)--(e2)--(g234)--cycle;

\foreach \i in {e1,e2,e3,e4}
\draw[dashed] (0,0)--(\i);
    \draw[dashed, fill=red!30, opacity=.2] (e1)--(e4)--(e2)--cycle;
    \draw[-,opacity=.2] (e2) -- (e4);
\draw[-, fill=green!30, opacity=.2] (e1) --(e4)--(e3)--cycle;
\draw[-, fill=purple!30, opacity=.2] (e2)--(e4)--(e3)--cycle;
\end{tikzpicture}

\begin{tikzpicture}[line join = round, line cap = round]
\pgfmathsetmacro{\factor}{1/sqrt(2)};
\coordinate [label=right:$x_1$] (e1) at (2,0,-2*\factor);
\coordinate [label=left:$x_2$] (e2) at (-2,0,-2*\factor);
\coordinate [label=above:$x_3$] (e3) at (0,2,2*\factor);
\coordinate [label=below:$x_4$] (e4) at (0,-2,2*\factor);
\coordinate (e) at (0,0,0*\factor);
\coordinate (f12) at (0,0,-2*\factor);
\coordinate (f24) at (-1,-1,0*\factor);
\coordinate (f34) at (0,0,2*\factor);
\coordinate (f23) at (-1,1,0*\factor);
\coordinate (g234) at (-2/3,0,2*\factor/3);
\coordinate (g123) at (0,2/3,-2*\factor/3);
\coordinate (g124) at (0,-2/3,-2*\factor/3);
\draw[-, fill=blue!30, opacity=.8] (e)--(e4)--(f34)--cycle;
\draw[-, fill=blue!30, opacity=.8] (e)--(f34)--(g234)--cycle;
\draw[-, fill=blue!30, opacity=.8] (e)--(e4)--(g234)--cycle;
\draw[-, fill=blue!30, opacity=.8] (f34)--(e4)--(g234)--cycle;
\draw[-, fill=blue!30, opacity=.3] (e)--(e2)--(f23)--cycle;
\draw[-, fill=blue!30, opacity=.3] (e)--(f23)--(g234)--cycle;
\draw[-, fill=blue!30, opacity=.3] (e)--(e2)--(g234)--cycle;
\draw[-, fill=blue!30, opacity=.3] (f23)--(e2)--(g234)--cycle;

\draw[-, fill=blue!30, opacity=.3] (e)--(e2)--(f23)--cycle;
\draw[-, fill=blue!30, opacity=.3] (e)--(f23)--(g123)--cycle;
\draw[-, fill=blue!30, opacity=.3] (e)--(e2)--(g123)--cycle;
\draw[-, fill=blue!30, opacity=.3] (f23)--(e2)--(g123)--cycle;

\draw[-, fill=red!30, opacity=.3] (e)--(e3)--(f23)--cycle;
\draw[-, fill=red!30, opacity=.3] (e)--(f23)--(g123)--cycle;
\draw[-, fill=red!30, opacity=.3] (e)--(e3)--(g123)--cycle;
\draw[-, fill=red!30, opacity=.3] (f23)--(e3)--(g123)--cycle;

\draw[-, fill=red!30, opacity=.3] (e)--(e3)--(f34)--cycle;
\draw[-, fill=red!30, opacity=.3] (e)--(f34)--(g234)--cycle;
\draw[-, fill=red!30, opacity=.3] (e)--(e3)--(g234)--cycle;
\draw[-, fill=red!30, opacity=.3] (f34)--(e3)--(g234)--cycle;

\draw[-, fill=red!30, opacity=.3] (e)--(e3)--(f23)--cycle;
\draw[-, fill=red!30, opacity=.3] (e)--(f23)--(g234)--cycle;
\draw[-, fill=red!30, opacity=.3] (e)--(e3)--(g234)--cycle;
\draw[-, fill=red!30, opacity=.3] (f23)--(e3)--(g234)--cycle;

\foreach \i in {e1,e2,e3,e4}
\draw[dashed] (0,0)--(\i);
    \draw[dashed, fill=red!30, opacity=.2] (e1)--(e4)--(e2)--cycle;
    \draw[-,opacity=.2] (e2) -- (e4);
\draw[-, fill=green!30, opacity=.2] (e1) --(e4)--(e3)--cycle;
\draw[-, fill=purple!30, opacity=.2] (e2)--(e4)--(e3)--cycle;
\end{tikzpicture}

\end{center}
\caption{The union $E$ of the three chambers $\{x_4\geq x_3\geq x_2\geq x_1\}$, $\{x_2\geq x_3 \geq x_4\geq x_1\}$ and
  $\{x_2\geq x_3 \geq x_1\geq x_4\}$ (left) is an ambitropical cone. Tropical cone $E^{\max}$ (middle) and $E^{\min}$ (right).   } 
\label{fig-3d5}
\end{figure}

\section*{Acknowledgments}
The second author thanks Jean-Bernard Baillon, for
having pointed out related works concerning hyperconvexity,
and he also thanks Gleb Koshevoy for helpful comments and references.
The first two authors thank Antoine Hochart
for helpful discussions.
\bibliographystyle{alpha}
\bibliography{biblio,tropicalvolume,zhengbiblio,references-hochart}
\end{document}